\PassOptionsToPackage{dvipsnames}{xcolor}
\documentclass{amsart}

\usepackage[hmargin=3.5cm,top=3cm,bottom=3cm,includehead]{geometry}

\usepackage{babel}
\usepackage{pifont}
\usepackage{tikz}

\usepackage{amsmath, amssymb, amsthm, mathrsfs, bbm}
\usepackage{graphicx,color,enumitem} 
\usepackage{multicol}
\usepackage{cancel}
\usepackage{abraces}
\usepackage{stmaryrd}
\usepackage{comment}

\usepackage[pdfborder={0 0 0}]{hyperref}

\usepackage{xparse}
\let\realItem\item 
\makeatletter
\NewDocumentCommand\myItem{ o }{%
   \IfNoValueTF{#1}%
      {\realItem}
      {\realItem[#1]\def\@currentlabel{#1}}
}
\makeatother

\usepackage{enumitem}    
\setlist[enumerate]{
    before=\let\item\myItem,       
    label=\textnormal{(\arabic*)}, 
    widest=(2')                    
}

\usepackage[dvipsnames]{xcolor}

\newcommand\N{{\mathbb N}}

\newcommand\R{{\mathbb R}}

\renewcommand{\d}{\mathrm{d}}
\newcommand{\dv}{\mathrm{d} v}
\newcommand{\dt}{\mathrm{d} t}

\newcommand{\dx}{\mathrm{d} x}

\def\AA{{\mathcal A}}
\def\BB{{\mathcal B}}
\def\CC{{\mathcal C}}
\def\DD{{\mathcal D}}
\def\EE{{\mathcal E}}

\def\HH{{\mathcal H}}

\def\LL{{\mathcal L}}
\def\MM{{\mathcal M}}
\def\NN{{\mathcal N}}
\def\OO{{\mathcal O}}
\def\PP{{\mathcal P}}
\def\QQ{{\mathcal Q}}

\def\TT{{\mathcal T}}
\def\UU{{\mathcal U}}
\def\VV{{\mathcal V}}
\def\WW{{\mathcal W}}
\def\VV{{\mathcal V}}

\def\ZZ{{\mathcal Z}}

\def\AAA{{\mathscr A}}
\def\BBB{{\mathscr B}}

\def\DDD{{\mathscr D}}
\def\EEE{{\mathscr E}}
\def\FFF{{\mathscr F}}
\def\GGG{{\mathscr G}}
\def\HHH{{\mathscr H}}

\def\LLL{{\mathscr L}}
\def\MMM{{\mathscr M}}

\def\PPP{{\mathscr P}}
\def\QQQ{{\mathscr Q}}
\def\RRR{{\mathscr R}}
\def\SSS{{\mathscr S}}

\def\VVV{{\mathscr V}}
\def\WWW{{\mathscr W}}

\def\ZZZ{{\mathscr Z}}

\def\BBBB{{\mathfrak B}}
\def\CCCC{{\mathfrak C}}

\def\FFFF{{\mathfrak F}}

\def\HHHH{{\mathfrak H}}

\def\LLLL{{\mathfrak L}}
\def\MMMM{{\mathfrak M}}
\def\NNNN{{\mathfrak N}}

\def\WWWW{{\mathfrak W}}

\def\ffff{{\mathfrak f}}

\def\MMMM{{\mathfrak M}}

\def\Lloc{L_{\rm\tiny loc}}

\newcommand{\wto}{\rightharpoonup}

\def\eps{{\varepsilon}}

\newcommand{\la}{\langle}
\newcommand{\ra}{\rangle}
\newcommand{\rra}{\rangle\!\rangle}
\newcommand{\lla}{\langle\!\langle}
\newcommand{\lv}{\lvert}
\newcommand{\rv}{\lvert}
\newcommand{\lvv}{\lVert}
\newcommand{\rvv}{\rVert }
\newcommand{\lvvv}{\lvert\!\lvert\!\lvert}
\newcommand{\rvvv}{\rvert\!\rvert\!\rvert}

\newcommand{\grad}{\nabla}

\DeclareMathOperator{\Div}{div}

\newcommand\Ind{{\mathbf 1}}

\newtheorem{theo}{Theorem}[section]
\newtheorem{prop}[theo]{Proposition}
\newtheorem{lem}[theo]{Lemma}

\newtheorem*{thm*}{Theorem}
\theoremstyle{remark}
\newtheorem{rem}[theo]{Remark}

\newtheorem*{ex*}{Example}
\theoremstyle{definition}

\numberwithin{equation}{section}
\newcommand{\be}{\begin{equation}}
\newcommand{\ee}{\end{equation}}
\newcommand{\ba}{\begin{aligned}}
\newcommand{\ea}{\end{aligned}}
\newcommand{\beqn}{\begin{equation*}}
\newcommand{\eeqn}{\end{equation*}}
\newcommand{\bear}{\begin{eqnarray}}
\newcommand{\eear}{\end{eqnarray}}
\newcommand{\bean}{\begin{eqnarray*}}
\newcommand{\eean}{\end{eqnarray*}}
\newcommand{\bal}{\begin{aligned}}
\newcommand{\eal}{\end{aligned}}

\title[Non-equilibrium steady states of a weakly non-linear kinetic Fokker-Planck]{Existence and stability of non-equilibrium steady states of a weakly non-linear kinetic Fokker-Planck equation in a domain}
\author[J. Evans]{J. EVANS}
\author[R. Medina]{R. MEDINA}

\address[J.~Evans]{Warwick Mathematics Institute, 
Zeeman Building, University of Warwick, CV4 7AL}
\email{Josephine.Evans@warwick.ac.uk}

\address[R.~Medina]{Centre de Recherche en Math\'ematiques de
 la D\'ecision (CEREMADE, CNRS UMR 7534),
  Universit\'e Paris Dauphine - PSL, Place de Lattre de
  Tassigny, 75775 Paris 16, France}
\email{richard.medina-rodriguez@dauphine.psl.eu}

\date{\today}

\subjclass[2020]{35Q82, 35Q84, 35G31, 35B40, 82C05}

\keywords{Kinetic Fokker-Planck equation, BGK thermostat, Maxwell boundary conditions, Non-equilibrium steady states, long-time asymptotic behavior}

\begin{document}

\begin{abstract}
We study a weakly non-linear Fokker-Planck equation with BGK heat thermostats in a spatially bounded domain with  conservative Maxwell boundary conditions, presenting a space-dependent accommodation coefficient and a space-dependent temperature on the spatial boundary. 
The model is based from a problem introduced in [E. A. Carlen, R. Esposito, J. L. Lebowitz, R. Marra, and C. Mouhot.\textit{ Approach to the steady state in kinetic models with thermal
  reservoirs at different temperatures.} {\emph J. Stat. Phys.}, 172(2):522--543, 2018]
where the authors studied the properties of the non-equilibrium steady states for non-linear kinetic Fokker-Planck equations with BGK thermostats in the torus. 
We generalize those results for bounded domains using the recent results presented in [K. Carrapatoso, P. Gabriel, R. Medina, and S. Mischler. \textit{Constructive Krein-Rutman result for kinetic Fokker-Planck equations in a domain}, 2024] for the study of general kinetic Fokker-Planck equations with Maxwell boundary conditions.
More precisely, in a weakly non-linear regime, we obtain the existence of a non-equilibrium steady state and its stability in the perturbative regime. 
\end{abstract}

\maketitle

\tableofcontents

\section{Introduction}\label{sec:Intro}
The study of non-equilibrium steady states (NESS) remains one of the central problems in statistical mechanics.  There are only a few models where fundamental questions, such as whether they exist, whether they are unique, and whether they are stable, can be answered, even if partially (see Subsection \ref{ssec:Motivation} for a detailed review of known results). This paper aims to contribute to the study of NESS within the context of kinetic theory.

This field of mathematics study equations coming from statistical physics modeling multi-particle systems by taking a statistical viewpoint. The unknown quantity is $f(t,x,v)$ which is the probability density function of a single \textquotedblleft typical" particle which at any time $t$ is located at position $x$ and is moving with velocity $v$. In such a setting the system is described as being in \emph{(local) thermodynamic equilibrium} when the particles velocities have a Maxwellian (Gaussian outside the world of kinetic theory) distribution of the form
\[\rho(x) (2 \pi T(x))^{-d/2} \exp \left(- \frac{|v-u(x)|^2}{2T(x)} \right),  \]
for some functions $\rho, u, T$ depending on the spatial variable $x$. 
Most of the time kinetic equations model systems which are outisde thermodynamic equilibrium, however as these systems approach a global steady state, as $t \rightarrow \infty$, they move towards an equilibrium state. In this paper however, we are interested in situations where the steady state does not have a Gaussian velocity distribution.

Physically, non-equilibrium steady states occur in open systems where there exist flows of macroscopic variables such as heat, see for instance \cite{tomé_2006, MR3223169}. In particular, in the context of gas dynamics this can arise when gasses are in contact with thermostats, which is the physical situation we study in this paper. Furthermore it is worth remarking that the study of non-equilibrium steady states in kinetic theory is fundamental as it can help us understand how non-equilibrium behaviors involving the flow of macroscopic quantities can emerge from particle systems.

More precisely, in this paper we study a nonlinear kinetic Fokker-Planck (KFP) equation with several thermostat terms. 
The specific nonlinearity in the KFP equation we take comes from \cite[Equation (1.27)]{MR3824949} where they study NESSs for such an equation on the torus. In contrast, our work is on a bounded domain and contains a different and more challenging set of \textquotedblleft thermostat" terms which describe how heat flows in and out of the system. 

We are able to attack this more complicated situation using tools from  \cite{CGMM24}, where they develop powerful techniques to study the existence, uniqueness, and stability of steady states of a linear KFP equation in a bounded domain, presenting Maxwell boundary conditions with a space-dependent temperature. As we rely strongly on these linear tools we work in a regime where the nonlinear term is small.

%
%

Finally, it is important to note that, even in a weakly non-linear regime, boundary thermostats make the study of the NESS more challenging. For instance, and in contrast with \cite{MR3824949}, we lose the access to an explicit formula for the NESS, we no longer have reasons to believe that the steady states are spatially uniform, and we cannot rule out the existence of steady states with infinite energy.


\subsection{Framework}\label{ssec:Framework}
We consider $\alpha \in (0,1/2)$, dimension $d\geq 3$, and we study the following non-linear equation
\be\label{eq:NonLKFP}
\partial_t f =  -v\cdot \grad_x f + \left( \alpha \EE_f + (1-\alpha)\tau\right) \Delta_v f + \Div_v(vf)  + \GGG f \quad \text{ in } \UU := (0,\infty)\times \Omega \times \R^d,
\ee
describing the evolution of the density function $f:=f(t,x,v)$, depending on the variables associated to time $t\in (0,\infty)$, space $x\in \Omega\subset \R^d$ and velocity $v\in \R^d$.
We have considered the function $\tau = \tau(x):\Omega\to \R$ such that 
$$
\tau_0\leq \tau(x) \leq \tau_1 \qquad \forall x\in \Omega,
$$
for some constants $\tau_0,\tau_1>0$, and we have defined the \emph{total energy} functional
$$
\EE = \EE_f := \frac 1{ d} \int_{\Omega \times \R^d} \lv v\rv^2 f\dx \dv,
$$
and the BGK \emph{heat thermostat} 
$$
\GGG f = \sum_{n=1}^\NN \eta_n \GGG_n f \quad \text{ with } \quad \GGG_n f =  \Ind_{\Omega_n} \left( \varrho_f \MM_{T_n} - f \right),
$$ 
for some $\NN \in \N$, some parameters $\eta_n \geq 0$, $T_n>0$, the subsets $\Omega_n \subset \Omega$, and where we have defined
$$
\varrho_f = \int_{\R^d} fdv,  \quad \text{ and } \quad \MM_{\TT} =\MM_\TT (v):= {1\over (2\pi \TT)^{d/2}} \exp\left( -{\lvert v\rvert^2 \over 2\TT}\right) \quad \text{ for } \TT:\Omega \to \R_+^*. 
$$
We present to the reader a discussion on the physical interpretation of the different operators involved in Equation \eqref{eq:NonLKFP} in Subsection \ref{ssec:Motivation} (see also \cite{MR3824949} for further details on the modeling).
\\

We take $\Omega$ to be a bounded domain and, without loss of generality, we impose $\lvert \Omega\rvert =1$. 
Moreover we assume $\Omega$ to be a $C^1$ domain, more precisely we assume that 
$$
\Omega := \{ x \in \R^d; \,  \delta(x) > 0 \}
$$  
for a function $\delta\in W^{2,\infty}(\R^d)$ such that $\delta (x) := \hbox{\rm dist}(x,\partial\Omega)$ on a neighborhood of  the boundary set $\partial\Omega$, thus $n_x = n(x) :=   - \nabla \delta(x)$  
coincides with the  unit normal outward vector field on $\partial\Omega$.
We define the boundary set $\Sigma = \partial \Omega\times \R^d$ and we differentiate between the sets of \emph{outgoing} velocities ($\Sigma_+$) and \text{incoming} velocities ($\Sigma_-$) on the boundary by
\beqn
\Sigma_\pm = \{(x,v)\in \Sigma,\, \pm n_x\cdot v>0 \}.
\eeqn
Furthermore we set $\Gamma = (0,\infty)\times\Sigma$ and accordingly $\Gamma_{\pm} = (0,\infty)\times \Sigma_\pm$. We define $\gamma f$ as the trace function associated with $f$ over $\Gamma$ and $\gamma_\pm f= \Ind_{\Gamma_{\pm}} \gamma f$. We then consider the \emph{accommodation coefficient} $\iota \in C(\partial\Omega, [0,1])$ and we complement Equation \eqref{eq:NonLKFP} with the \emph{Maxwell boundary condition}
\be\label{eq:BoundaryConditions}
\gamma_- f =\RRR \gamma_+ f  := (1-\iota) \SSS  \gamma_+ f +\iota  \DDD \gamma_+ f,\qquad \quad \text{ on } \Gamma_{-},
\ee
where we have defined the \emph{specular reflection} operator 
\beqn
\SSS g(t,x,v) := g(t,x,\VV_x v),\qquad\text{ where } \quad  \VV_x v := v-2n_x(n_x\cdot v),
\eeqn
and the \emph{diffussive reflection} operator 
\beqn
\DDD g(t,x,v) = \MMM_{\Theta}(v) \int_{\R^d} g(t,x,u) (n_x\cdot u)_+ du  ,\qquad\text{ where } \quad \MMM_\Theta := \sqrt{2\pi/ \Theta}\MM_\Theta,
\eeqn
and we have introduced the boundary temperature $\Theta\in W^{1,\infty}(\overline\Omega, \R)$ such that 
\be\label{eq:Assum-Theta}
\Theta_* \leq \Theta(x) \leq  \Theta^*, \quad  \text{ for some constants } 0< \Theta_* \leq \Theta^* <\infty .
\ee

Finally we complement the nonlinear Equation \eqref{eq:NonLKFP}-\eqref{eq:BoundaryConditions} with the initial condition 
\be\label{eq:InitialDatum}
f (t=0, x,v) = f_0(x,v) \quad\text{ on } \OO,
\ee 
and we assume, without loss of generality, that $\lla f_0\rra_\OO:=\displaystyle \int_\OO f_0 \, \dv\, \dx = 1$.

\subsection{Main results} \label{ssec:Results}
In order to state our main results we introduce the class of the so-called {\it  admissible weight functions} $\omega : \R^d \to (0,\infty)$ defined by
\beqn\label{eq:omega}
\omega = \la v \ra^k \exp(\zeta\la v\ra^s),  
\eeqn
and such that either
\beqn
s=0, \, \zeta\geq 0, \text{ and }  k >   k_*  \text{ with } k_* = d+1,  \quad \text{ or } \quad s \in (0,1] \text{ with } \zeta \in (0,\infty), \text{ and any } k\geq 0.
\eeqn
We refer the reader towards \cite{CGMM24} for an explanation on how this choice of admissible weight functions provides, in particular, a control on the behavior of local Kolmogorov equations. Moreover, we emphasize that due to the fact that $s\in [0,1]$ we will also have the control
\be
\label{eq:cond-Lpomega}
{\grad_v \, \omega /  \omega} \in L^\infty(\R^d),
\ee 
which will be necessary in order to appropriately control the non-linearity of the equation. \\

We also need to introduce some functional spaces. 
For a given measure space $(Z,\ZZZ,\mu)$, a weight function $\sigma : Z \to (0,\infty)$ and an exponent $p \in [1,\infty]$, we define the weighted Lebesgue spaces $L^p_\sigma(Z)$ 
associated to the norm 
\be\label{def:Lebesgue_weighted_spaces}
\| g \|_{L^p_\sigma(Z)} = \| \sigma g \|_{L^p(Z)}. 
\ee
Furthermore, we define the space of continuous functions in $Y$ as $C(Z)$. We then have the following results.\\

We first present a well-posedness and stability theorem in the linear framework when $\alpha =0$. 

\begin{theo}\label{theo:SteadySolutionL}
We assume $\alpha =0$. There exists $\FFFF^0\in L^2(\Omega, H^1(\R^d)) \cap L^\infty(\OO)$ unique steady solution to the linear Equation \eqref{eq:NonLKFP}-\eqref{eq:BoundaryConditions}-\eqref{eq:InitialDatum}. Moreover, $\lla \FFFF^0 \rra_\OO=1$ and for any admissible weight function $\varsigma$ there holds
\be\label{eq:SteadySolutionLIneq}
\lvv \grad_v \FFFF^0 \rvv_{L^2_\varsigma(\OO)} <\infty \qquad \text{ and } \qquad \FFFF^0 (x,v) \lesssim (\varsigma(v))^{-1}.
\ee
Furthermore, let $\omega$ be an admissible weight function, for any initial data $f_0\in L^2_\omega(\OO)$ there is a unique global renormalized solution $f \in C(\R_+, L^2_\omega(\OO))$ to the linear Equation \eqref{eq:NonLKFP}-\eqref{eq:BoundaryConditions}-\eqref{eq:InitialDatum} and there is $\lambda >0$ such that 
\be\label{eq:SteadySolutionL}
\lvv f_t -   \FFFF^0\rvv_{L^2_\omega(\OO)} \lesssim e^{-\lambda t} \lvv f_0 -\FFFF^0\rvv_{L^2_\omega(\OO)} \qquad \forall t\geq 0.
\ee
\end{theo}

The precise sense of the global solutions provided in Theorem \ref{theo:SteadySolutionL} is given by Theorem \ref{theo:WellPosednessL2} with the choice of $\Lambda=\tau$. We also remark that Theorem \ref{theo:WellPosednessL2} is but a direct application of \cite[Theorem 2.11]{CM_Landau_domain} and the trace theory from \cite[Theorem 2.8]{CM_Landau_domain}. 

The existence and uniqueness of a stationary solution for the linear problem, which we remark is also in the sense of Theorem \ref{theo:WellPosednessL2}, as well as its stability are obtained as a direct application of the Krein-Rutmann-Doblin-Harris theory developed in \cite[Theorem 6.1]{sánchez2024voltageconductancekineticequation} but we also refer towards \cite[Theorem 7.1]{CGMM24} for a similar result in a non-conservative setting and to \cite{sanchez:hal-04093201} for the study of a general Krein-Rutmann-Doblin-Harris result in a general theoretical framework. 

We futher note that Theorem \ref{theo:SteadySolutionL} is a slight generalization of \cite[Theorems 1.1 and 1.2 ]{CGMM24}.
\\

In the non-linear framework we have first the existence of a steady state for $\alpha>0$ sufficiently small.

\begin{theo}\label{theo:SteadySolutionNonL}
There exists $\alpha^\star\in (0,1/2)$ such that for every $\alpha \in (0, \alpha^\star)$,
there is a positive function $\FFFF^\alpha \in L^2(\Omega, H^1(\R^d)) \cap L^\infty(\OO)$, steady solution of Equation \eqref{eq:NonLKFP}-\eqref{eq:BoundaryConditions}-\eqref{eq:InitialDatum}. Moreover, $\lla \FFFF^\alpha \rra_\OO=1$ and  for every admissible weight function $\omega$ there holds 
\be\label{eq:PropsSS}
\lvv \grad_v \FFFF^\alpha \rvv_{L^2_\omega(\OO)} <\infty, \quad \FFFF^\alpha (x,v) \lesssim (\omega(v))^{-1}, \quad \text{ and } \quad  \EE_{\FFFF^\alpha}  \leq 2 \EE_{\FFFF^0},
\ee
uniformly in $\alpha$, and where $\FFFF^0$ is given by Theorem \ref{theo:SteadySolutionL}.
\end{theo}

The main consequence of Theorem \ref{theo:SteadySolutionNonL} is the existence of a NESS for the non-linear  Equation \eqref{eq:NonLKFP}-\eqref{eq:BoundaryConditions}, as well as some of its qualitative properties regarding regularity and decay tail in velocity.  
We remark that the proof of Theorem \ref{theo:SteadySolutionNonL} is based in the application of a fixed point argument in the spirit of the proof of \cite[Theorem 1]{MR4265362}. Additionally $\FFFF^\alpha$ is a stationary state in the sense of Theorem \ref{theo:WellPosednessL2} by taking $\Lambda = \alpha \EE_{\FFFF^\alpha} + (1-\alpha) \tau$.

We remark that Theorem \ref{theo:SteadySolutionNonL} generalizes \cite[Theorem 1.2]{MR3824949} and that, in contrast with this work, we observe major differences on the behavior and properties of the NESS in the absence of periodic boundary conditions: we have no reasons to believe that the NESS will be independent of the spatial variable $x$, the bounds on the energy functional $\EE$ cannot be obtained as during \cite[Lemma 1.1]{MR3824949} (see Subsection \ref{ssec:First_Observ}), we lack the information to rule out the existence of steady states with unbounded total energy, and we don't have access to an explicit representation of the NESS.
\\

Finally, we state the following stability result for the previous NESS.

\begin{theo}\label{theo:GlobalSolution}
We consider an admissible weight function $\omega$. There are $\alpha^{\star\star}\in (0,\alpha^\star)$ and $\delta>0$, where $\alpha^\star$ is given by Theorem \ref{theo:SteadySolutionNonL}, such that for every $\alpha \in (0, \alpha^{\star\star})$ and for any initial datum $f_0\in L^2_\omega(\OO)$ satisfying
$$
\lvv f_0 - \FFFF^\alpha \rvv_{L^2_\omega(\OO)} \leq \delta,
$$
there is $f \in L^2_\omega(\UU)$ global weak solution of Equation \eqref{eq:NonLKFP}-\eqref{eq:BoundaryConditions}-\eqref{eq:InitialDatum}. Moreover, there is $\eta>0$ for which there holds
\be\label{eq:DecaySmallness}
\lvv f_t -  \FFFF^\alpha\rvv_{L^2_\omega(\OO)} \lesssim e^{-\eta t} \lvv f_0 -  \FFFF^\alpha\rvv_{L^2_\omega(\OO)} \qquad \forall t\geq 0.
\ee
\end{theo}

The global solutions provided by Theorem \ref{theo:GlobalSolution} are constructed in the sense that the function $h := f - \lla f_0\rra_\OO\,  \FFFF^\alpha$ satisfies Equation \eqref{eq:KFPstabilityNL}, in the sense of Proposition \ref{prop:KFPL2Perturbed}. 

It is worth remarking that Proposition \ref{prop:KFPL2Perturbed} is mainly an application of the Lion's variant of the Lax-Milgram theorem \cite[Chap~III,  \textsection 1]{MR0153974} as used in \cite{CM_Landau_domain}, see also \cite{CGMM24, CM_Landau_domain, sanchez:hal-04093201, CM-KFP**} for similar arguments on the existence of solutions of kinetic equations. 
The trace theory was taken mainly from \cite{CM_Landau_domain, CGMM24} but we also refer to \cite{MR1776840,MR2721875,sanchez:hal-04093201} for further references on the trace theory for kinetic equations. 

Furthermore, we emphasize that to obtain the a priori estimates leading to the proof of Proposition \ref{prop:KFPL2Perturbed} we have used the modified weight functions from \cite{CM-KFP**, sanchez:hal-04093201, CGMM24} to control the Maxwell boundary condition. 

In addition, the decay estimate was obtained by defining a new norm in the spirit of  \cite[Proposition 3.6]{CM_Landau},  \cite[Proposition 3.2]{Carrapatoso_2016} and \cite[Proposition 4.1]{MR3591133}. It is worth remarking that we are not able to construct of a \emph{hypocoercivity} theory in the spirit of \cite{MR2562709, MR4581432, DMS, CM-KFP**} due to the lack of extra information on the steady state, namely positivity bounds and regularity.


We note that Theorem \ref{theo:GlobalSolution} generalizes \cite[Theorem 1.3]{MR3824949} and we remark that the techniques used for the obtention of these results are different from those developed during the proof of the main theorems from \cite{MR3824949}. In particular, we do not need to study the underlying ergodic process associated with the linearized operator to obtain our results.

\subsection{Physical motivation and state of the art}\label{ssec:Motivation} 
%
%

The study of non-equilibrium systems within statistical physics has been motivated by the investigation of many physical and biological models (see \cite{MR3763144, MR4628001, BIANCA2012359}). We refer the reader towards \cite{Jepps_2010, MR3223169} and \cite[Chapter 9]{MR1707309} for an exposition on non-equilibrium statistical mechanics and towards \cite{tomé_2006} for a characterization of the entropy for NESS, including the case of the KFP equation.
\\

We expose now why Equation \eqref{eq:NonLKFP}-\eqref{eq:BoundaryConditions}-\eqref{eq:InitialDatum} represents a system which is not in \emph{thermal equilibrium} and we explain the physical process that each term models. 

We define the Fokker-Planck type operators
\be\label{eq:KFP_operator}
\CC_\TT f := \TT \Div_v\left( \MM_\TT \grad_v\left(\frac{f}{\MM_\TT}\right)\right) = \TT\Delta_v f + \Div_v(v f),
\ee
associated to the temperature $\TT:\Omega \to \R_+^*$. We observe that the operator described in \eqref{eq:KFP_operator} can be interpreted as having a \emph{thermostat} effect, as it represents the gasses interaction with another gas which is present throughout the domain and has density $\mathcal{M}_{\TT}$. We refer towards \cite{MR1942465, MR1787105, MR749386, MR8130} and the references therein for more on the modeling and properties of Fokker-Planck type operators.

On the other hand, for each $n \in \llbracket 1, \NN\rrbracket$, the BGK thermostat $\GGG_n $ model the gasses interaction with other gasses with density $\MM_{T_n}$, which are only present in some parts of the domain $\Omega_n$. We refer towards \cite{MR1942465, MR4265362, MR3824949} and the references therein for more on the modeling and properties of BGK operators.

We can then rewrite the right hand side of Equation \eqref{eq:NonLKFP} as
$$
 -v \cdot \nabla_x f + \alpha \CC_{\EE_f} f + (1-\alpha) \CC_\tau f + \sum_{n=1}^\NN \eta_n  \GGG_n f, 
$$
which highlights the different forces acting simultaneously upon the particles. On the one hand, we have the effect of the transport operator, followed by a non-linear interaction where the gasses interact with themselves via a Fokker-Planck operator whose temperature is the gasses total energy, and acting with intensity $\alpha$. On the other hand, we observe the presence of several energy exchange mechanisms:

 \smallskip\noindent
$\bullet$  a linear Fokker-Planck thermostat $\CC_\tau$ whose associated temperature is the function $\tau$ defined in Subsection \ref{ssec:Framework}, and acting with intensity $1-\alpha$,
  
\smallskip\noindent
$\bullet$  $\NN$ BGK thermostats as described above, each with an associated intensity $\eta_n$, with $n \in \llbracket 1, \NN \rrbracket$,
  
  \smallskip\noindent
$\bullet$  and the boundary thermostat at temperature $\Theta$, given by the diffusive boundary condition.
 %
%
%
%
%
%
%
%
\\

From the mathematical point of view, the study of NESS poses several challenges: they are generally not explicit, and it is often difficult to prove that they satisfy Poincar\'e or logarithmic Sobolev inequalities which in turn makes it challenging to apply the standard machinery of \emph{hypocoercivity} (see \cite{MR2562709, MR4581432, DMS}) to investigate their stability. 
Nonetheless, we highlight some results regarding non-equilibrium systems for various kinetic models.

\smallskip\noindent
$\bullet$ For a collisionless transport equation in a bounded domain with a non-isothermal boundary condition, we cite \cite{MR4179249}, where A. Bernou shows the existence and stability of a steady state.

\smallskip\noindent
$\bullet$ The study of NESS for BGK equations has received relatively more attention over the last years. We can cite for instance \cite{MR3824949, MR4146892} where E. Carlen, R. Esposito, J. Lebowitz, R. Marra and C. Mouhot study a non-linear BGK equation in the torus with the presence of BGK thermostats at different temperatures. In each of this articles the authors construct a NESS and they prove their exponential stability. Moreover, they also prove uniqueness of such steady state in \cite{MR4146892} whereas in \cite{MR3824949} they obtain local uniqueness.

We also mention the works of J. Evans and A. Menegaki on BGK equations. In \cite{evans2023propertiesnonequilibriumsteadystates} the authors study a non-linear BGK equation in a 1-dimensional torus and obtain existence, local uniqueness and stability of a NESS. It is also worth remaking that the decay estimates were obtained by the used of a hypocoercivity technique in the spirit of \cite{DMS}. Furthermore, in \cite{MR4265362} these authors study a non-linear BGK equation in a real interval presenting diffusive boundary conditions with different temperatures at each side. In this model they prove existence of a NESS under weak conditions on the boundary temperatures.

Finally, within the framework of BGK equation we have \cite{bernou_hal_04176707}, where A. Bernou study a non-isothermal problem, for which he constructs a NESS, and provides its uniqueness and stability under suitable conditions in dimensions 2 and 3.

\smallskip\noindent
$\bullet$ We take a glance now at KFP equation. Regarding linear problems we cite two main papers: on the one hand, we have \cite{MR4278431} where C. Cao studied a problem in the whole $\R^d$ as spatial domain, and proved the existence and uniqueness of a steady solution (non necessarily Maxwellian) as well as its stability. On the other hand, for linear KFP equations with bounded domains we cite \cite{CGMM24} where K. Carrapatoso, P. Gabriel, R. Medina and S. Mischler obtained the existence, uniqueness and stability of a NESS. It is worth remarking that the equation studied in \cite{CGMM24} was non-conservative and it exhibited Maxwell boundary conditions with a space-dependent temperature. 
  
 Furthermore, we cite again \cite{MR3824949} where E. Carlen, R. Esposito, J. Lebowitz, R. Marra and C. Mouhot also studied a non-linear KFP equations with BGK thermostats in the torus and they obtain the existence, local uniqueness, stability, and an explicit formula of a NESS.

\smallskip\noindent
$\bullet$ For the Boltzmann equation with hard spheres near the hydrodynamic limit, we have the papers \cite{MR3740632,GuoZhou_BE, MR3085665, MR2997586, MR1842024} exploring the effects of a heat thermostat on the boundary.
 
Whiting the previously mentioned papers, we highlight \cite{MR3085665} where R. Esposito, Y. Guo, C. Kim and R. Marra construct a NESS and provide its uniqueness and stability in a perturbative regime. Such perturbative regime consisted, besides the fact of being close to the hydrodynamical limit which is a perturbative regime in itself, on the imposition of boundary temperatures that don't fluctuate too much, some smallness condition on the initial data and a small \emph{Knudsen} number. 

Similarly, in \cite{MR2997586} L. Arkeryd, R. Esposito, R. Marra and A. Nouri  construct a NESS which is locally unique, and stable, under the conditions of a small \emph{Knudsen} number and some smallness on the initial condition.  

In the study of NESS for Boltzmann equations not necessarily near its hydrodynamic limit we have \cite{MR1842024, MR3336871, MR1636525}. More precisely, in \cite{MR1842024} the authors construct $L^1$ solutions to the stationary Boltzmann equation in a 1d slab. In \cite{MR3336871}, E. Carlen, J. Lebowitz and C. Mouhot investigate a space homogeneous Boltzmann equation with pseudo-Maxwellian molecules and the non-equilibrium effects come from a Boltzmann-type thermostat. In this setting, the authors obtain existence, uniqueness and stability of a NESS. Furthermore, in \cite{MR1636525}, E. Carlen, R. Esposito, J. Lebowitz, R. Marra and A. Rokhlenko studied a homogeneous Boltzmann equation with KFP and BGK thermostats.

\subsection{First mathematical observations and strategy of the proof of the main results}\label{ssec:First_Observ}
We first remark that Equation \eqref{eq:NonLKFP}-\eqref{eq:BoundaryConditions} is mass conservative. Indeed, we observe that at least at a formal level, there holds
\be\label{eq:MassConservation}
 \frac d{dt} \int_\OO f_t = \int_\OO -v\cdot \grad_x f +  \left( \alpha \EE_f + (1-\alpha)\tau\right) \Delta_v f + \Div_v(vf)  + \GGG f = 0,
\ee
where we have used the fact that the Maxwell boundary condition satisfies
$$
\RRR:L^1 (\Sigma_+, \, d\xi_{1} ^1) \to L^1 (\Sigma_-, \, d\xi_{1} ^1) , \quad \lvv \RRR\rvv_{L^1(\Sigma_+, \, d\xi_{1} ^1)} \leq 1.
$$
This together with our assumptions on $f_0$ imply in particular that $\lla f_t\rra_\OO = 1$ for every $t\geq 0$.

Furthermore, due to the structure of Equation \eqref{eq:NonLKFP} it would be of interest to establish a priori bounds on the behavior of the energy operator $\EE$ in time, however, and in difference with \cite{MR3824949}, this is challenging due to the presence of a diffusion condition at the boundary. Indeed at a formal level, using the conservation of mass, we observe that
$$
\frac d{dt} \EE_{f_t}   \leq 2d (1-\alpha) \left(1 - \frac {\EE_{f_t}} {\tau_1} \right)  + \sum_{n=1}^\NN \eta_n T_n  + \int_{\Sigma_+} \iota \, \lv v\rv^2 \, \gamma_+ f \, \left(  -\lv v\rv^2 + \frac {d+1}2\sqrt{2\over \pi} \Theta^{3/2}   \right) \,  (n_x\cdot v)_+ ,
$$
and using the Grönwall lemma we further deduce that 
\begin{multline}\label{eq:EnergyBoundaryNoControl}
\EE_{f_t} \leq  2d (1-\alpha) \, e^{- {1\over \tau_1} t} \EE_{f_0} + \tau_1\left(2d(1-\alpha) + \sum_{n=1}^\NN \eta_n T_n  \right)\left( 1- e^{- {1\over \tau_1} t}\right) \\+ \int_0^t e^{- {1\over \tau_1} (t-s)} \int_{\Sigma_+} \iota(x) \lv v\rv^2 \, \gamma_+ f_s \, \left(  -\lv v\rv^2 + \frac {d+1}2\sqrt{2\over \pi} \Theta^{3/2}   \right) \,  (n_x\cdot v)_+ .
\end{multline}
The previous formula has several consequences, on the one hand if $\iota \equiv 0$, i.e there is only specular reflection at the boundary, we can immediately deduce the existence of a ball enclosing the functional $\EE$ for all $t\geq 0$.

On the other hand however, when $\iota \not\equiv 0$ and there is a heat source anywhere at the boundary, we cannot guarantee that the energy functional won't explode with time. 
Nonetheless, it is worth remarking that the term
$$
-\lv v\rv^2 + \frac {d+1}2\sqrt{2\over \pi} \Theta^{3/2} ,
$$ 
in \eqref{eq:EnergyBoundaryNoControl} helps in ensuring that the velocity of the particles doesn't grow \emph{too much} at the boundary, and this makes us believe that there should be a mechanism in which this fact helps in bounding the total energy for all time. 
\\

We are unable however, to exploit the previous idea and our strategy then
%
%
%
%
consists on studying first the linear equation 
\be\label{eq:KFPtau}
\left\{\begin{array}{rcll} 
 \partial_t f &=& \LL f  & \text{ in } \UU ,   \\
\displaystyle \gamma_- f &=& \RRR \gamma_+ f &\text{ on } \Gamma_{-} ,\\
\displaystyle  f_{t=0} &=& f_0 &\text{ in } \OO.
 \end{array}\right.
 \ee
 where $\LL f :=  -v\cdot \grad_x f + \CC_\Lambda f +\GGG f $,
for a function $\Lambda : \Omega \to \R$ such that 
$$
\Lambda_0 \leq \Lambda (x) \leq  \Lambda_1 \quad \text{ for every } x\in \Omega,
$$
for some constants $\Lambda_0, \Lambda_1>0$,
and we recall that $\CC_\Lambda$ is defined in \eqref{eq:KFP_operator}. We then proceed as follows:

\smallskip\noindent
\ding{172} \emph{Well-posedness of Equation \eqref{eq:KFPtau}.} By using the well-posedness and trace theory from \cite{CM_Landau_domain} we deduce that there is a strongly continuous semigroup $S_\LL$ associated with the solutions of Equation \eqref{eq:KFPtau}. We will also study extensively the properties of the semigroup $S_\LL$ which will be obtained by using the results from \cite{CGMM24} on the local part of the operator $\LL$ and extending them to the full operator by using the Duhamel formula.

\smallskip\noindent
\ding{173} \emph{Krein-Rutmann-Doeblin-Harris theory.} By using \cite[Theorem 6.1]{sánchez2024voltageconductancekineticequation} we obtain the existence, uniqueness and stability of a steady solution for Equation \eqref{eq:KFPtau}. In particular, we will deduce that such a steady state has finite total energy. 

\smallskip\noindent
\ding{174} \emph{Proof of Theorem \ref{theo:SteadySolutionL}.} We observe that by taking $\Lambda=\tau$, Equation \eqref{eq:KFPtau} coincides with Equation \eqref{eq:NonLKFP}-\eqref{eq:BoundaryConditions}-\eqref{eq:InitialDatum} with $\alpha =0$. Therefore the previous results immediately imply Theorem \ref{theo:SteadySolutionL}.

\smallskip\noindent
\ding{175} \emph{Existence of the NESS for the non-linear problem.} Using the existence of a steady state of Equation \eqref{eq:KFPtau} with bounded total energy we deduce, by the use of a fixed point argument in the spirit of the proof of \cite[Theorem 1]{MR4265362}, that for small values of $\alpha>0$ we can construct a NESS for Equation \eqref{eq:NonLKFP}-\eqref{eq:BoundaryConditions} satisfying the energy bound \eqref{eq:PropsSS}.
\\

After having proved Theorem \ref{theo:SteadySolutionNonL} we perturb Equation \eqref{eq:NonLKFP}-\eqref{eq:BoundaryConditions}-\eqref{eq:InitialDatum} around the NESS: we take $h= f- \FFFF^\alpha $ and we study the resulting linearized perturbed problem 
\be\label{eq:KFPstabilityLin1}
\left\{\begin{array}{rcll} 
 \partial_t h &=& -v\cdot \grad_x h  +   \CC_{\Lambda^\star} h + \GGG h  + \alpha \EE_g \Delta_v h   + \alpha\EE_{h} \Delta_v \FFFF^\alpha & \text{ in } \UU  \\
\displaystyle \gamma_- h &=& \RRR \gamma_+ h &\text{ on } \Gamma_{-} \\
\displaystyle  h_{t=0} &=& f_0 - \FFFF^\alpha &\text{ in } \OO,
 \end{array}\right.
 \ee 
where we have taken $g:\UU\to \R$ and we have defined $\Lambda^\star := \alpha \EE_{\FFFF^\alpha }  + (1-\alpha) \tau$.
Then to prove Theorem \ref{theo:GlobalSolution} we proceed as follows:

\smallskip\noindent
\ding{176} \emph{Well-posedness of Equation \eqref{eq:KFPstabilityLin1}.} We use again the theory developed in \cite{CM_Landau_domain} to obtain the well-posedness of Equation \eqref{eq:KFPstabilityLin1} under a smallness condition on $g$. It is worth remarking that the arguments leading to this result are more delicate than the proof of the well-posedness for Equation \eqref{eq:KFPtau}, due to the presence of the bilinear term $\EE_g\Delta_v h$ and of the $H^{-1}$ term $\EE_{h} \Delta_v \FFFF^\alpha$.

\smallskip\noindent
\ding{177} \emph{Hypodissipativity.} We remark now that we have access to a decay estimate in a weighted $L^2$ space for the solutions of the following equation
\be\label{eq:PPP_before}
\partial_t h = -v\cdot \grad_x h + \CC_{\Lambda^\star} h + \GGG h , \qquad  \gamma_- h = \RRR \gamma_+ h, \qquad  h_{t=0} = f_0-\FFFF^\alpha.
\ee
This is nothing but a consequence of the fact that it coincides with Equation \eqref{eq:KFPtau} by taking $\Lambda = \Lambda^\star$. We then prove that, for $\alpha$ and $g$ \emph{small} we can construct a new norm $\lvvv \cdot \rvvv$, equivalent to the aforementioned weighted $L^2$ norm, in which we can extend the dissipativity properties of Equation \eqref{eq:PPP_before} to the solutions of Equation \eqref{eq:KFPstabilityLin1} by treating the term $\alpha \EE_g \Delta_v h   + \alpha\EE_{h} \Delta_v \FFFF^\alpha$ as a \emph{small} perturbation in some sense.

\smallskip\noindent
\ding{178} \emph{Fixed point argument on the perturbed setting.} Finally, using the previous informations we will prove that the map that to $g$ associates $h$ a solution of Equation \eqref{eq:KFPstabilityLin1}, leaves invariant a ball in a weighted $L^2$ space, and that it is continuous for the weak $L^2$ topology. This, together with the hypodissipativity result above, lead to the proof of Theorem \ref{theo:GlobalSolution} by using the Schauder fixed point theorem.

\subsection{Structure of the paper} 
The paper is organized as follows.

\smallskip\noindent
In Section \ref{sec:Toolbox} we introduce some elementary lemmas to control the non-local terms from the BGK operators, and we present some powerful results developed in \cite{CGMM24} for the control of local KFP equations.
During Section \ref{sec:Study_S_LLL} we develop a priori estimates and the well-posedness of Equation \eqref{eq:KFPtau}. We derive the existence of the semigroup $S_\LLL$ and we provide its ultracontractive properties. We also study the existence and properties of the backwards equation dual to Equation \eqref{eq:KFPtau}.
In Section \ref{sec:ProofTheo1} we present the Krein-Rutman-Doblin-Harris theorem from \cite[Theorem 6.1]{sánchez2024voltageconductancekineticequation} and we use it to prove Theorem \ref{theo:SteadySolutionL}.
We devote Section \ref{sec:Proof_Theo2} to prove Theorem \ref{theo:SteadySolutionNonL} by using the results from the previous sections and arguing by using a fixed point argument in the spirit of the proof of \cite[Theorem 1]{MR4265362}. 
In Section \ref{sec:PerturbEquilib} we study the perturbed Equation \eqref{eq:KFPstabilityLin} and we obtain some a priori estimates as well as its well-posedness under suitable assumptions. 
Additionally, during Section \ref{sec:HypoDissipativity} we prove a hypodissipativity result for the solutions of Equation \eqref{eq:KFPstabilityLin}. 
Finally, in Section \ref{sec:ProoTheo3} we obtain an equivalent version of Theorem \ref{theo:GlobalSolution} by proving the existence of solutions of the non-linear perturbed problem. The proof is based on the application of the Schauder fixed point theorem and by using the results obtained from Sections \ref{sec:PerturbEquilib} and \ref{sec:HypoDissipativity}.

\subsection{Notation} We state now some of the notations we will be using during this paper. 

\smallskip\noindent
$\bullet$ Given two admissible weight functions $\omega, \varsigma$ we say that $\omega \prec \varsigma$ when $ \omega \varsigma^{-1} \in L^1(\R^d)\cap L^\infty(\R^d)$.

\smallskip\noindent
$\bullet$
Consider a measure space $(Z,\ZZZ,\mu)$ and a weight function $\sigma : Z \to (0,\infty)$, we observe that $L^2_\sigma(Z)$ is a Hilbert space with the scalar product 
$$
\la \phi, \psi\ra_{L^2_\sigma(Z)} :=\int_Z \phi\, \psi \, \sigma^2. 
$$
Furthermore, we also define the weighted Sobolev space $H^1_\sigma(Z)$ as the functions $\psi\in H^1(Z)$ such that the norm
$$
\lvv \psi\rvv_{H^1_\sigma(Z)} :=\lvv \psi\rvv_{L^2_\sigma(Z)} + \lvv \grad \psi \rvv_{L^2_\sigma(Z)} <\infty.
$$ 

\smallskip\noindent
$\bullet$ For two Banach spaces $Z_1$, $Z_2$ we define $\BBB(Z_1, Z_2)$ as the space of the linear bounded operators from $Z_1$ to $Z_2$. In particular, we will denote it only as $\BBB(Z_1)$ when $Z_1 = Z_2$.

\section{Toolbox} \label{sec:Toolbox}
 We define the local ultraparabolic operator
 \be\label{eq:DefLLL-}
\BB f := -v\cdot \grad_x f +  \Lambda(x) \Delta_v f + v\cdot \grad_v f + \left( d- \sum_{n=1}^\NN \eta_n \Ind_{\Omega_n}  \right) f.
\ee
and the non-local operator
\be\label{def:HHH+def}
\AA f := \varrho_f  \sum_{n=1}^\NN \eta_n  \Ind_{\Omega_n} \MM_{T_n}  .
\ee 
and we observe that $\LL = \BB + \AA$, where we recall that $\LL$ is given in Subsection \ref{ssec:First_Observ}. 
This section is devoted to provide the necessary tools in order to study Equation \eqref{eq:KFPtau}. More precisely, in Subsection \ref{ssec:GGG+Properties} we prove the boundedness properties of the operator $\AA$ in every suitable Lebesgue space followed by Subsection \ref{ssec:UltraparabolicKolmogorov&Dual} where we summarize the results from \cite{CGMM24} on KFP equations with Maxwell boundary conditions.

\subsection{Properties of the non-local operator $\AA$} \label{ssec:GGG+Properties} We have the following proposition on the properties of the non-local term of the BGK thermostat.

\begin{prop}\label{prop:HHH+Lp-Lp}
For any two admissible weight functions $\omega$ and $\omega_\star$, and for any $p \in [1, \infty]$, there is a constant $C> 0$ such that 
\be\label{eq:HHH+regularization}
\lvv \AA f \rvv_{L^p_{\omega}(\OO)}  \leq C \sum_{n=1}^\NN \lvv f\rvv_{L^p_{\omega_\star}(\OO_n)} \lesssim \lvv f\rvv_{L^p_{\omega_\star}(\OO)}.
\ee
where we have defined $\OO_n = \Omega_n \times \R^d$.
\end{prop}

\begin{proof}
We observe that due to the very definition of admissible weight functions there holds $\omega \prec \MM_{T_n}^{-1}$, for every $n\in \llbracket 1, \NN \rrbracket$. We remark that for the case $p=\infty$ the result is obvious, furthermore if $p\in [1,\infty)$ we have that
\beqn
\sum_{n=1}^\NN \int_{\OO_n} \left( \varrho_f \, \eta_n  \MM_{T_n}\right)^p \omega^p \lesssim \sum_{n=1}^\NN \int_{\Omega_n} \left( \int_{\R^d} f \right)^p \lesssim   \sum_{n=1}^\NN \int_{\OO_n} f^p ,\\
\eeqn
where we have used the Hölder inequality to obtain the second inequality. 
\end{proof}

\subsection{Well-posedness and properties of the local ultraparabolic part} \label{ssec:UltraparabolicKolmogorov&Dual}
We look now at the local KFP equation 
\be\label{eq:KFPtauLocal}
\left\{\begin{array}{rcll} 
 \partial_t f &=& \BB f   & \text{ in } \UU , \\
\displaystyle \gamma_- f &=& \RRR \gamma_+ f &\text{ on } \Gamma_{-} ,\\
\displaystyle  f_{t=0} &=& f_0 &\text{ in } \OO.
 \end{array}\right.
 \ee
 where we recall that $\BB$ is given by \eqref{eq:DefLLL-}. 
Furthermore, we also present during this subsection the first eigenproblem associated to the previous equation, i.e the existence and properties of the eigentriplet  $(\lambda_1, f_1, \phi_1)$ satisfying
 \beqn\label{eq:1stEVP}
\lambda_1 \in \R, \quad \BB f_1 = \lambda_1 f_1, \quad \gamma_{_-} f_1 = \RRR \gamma_{_+} f_1, 
\quad \BB^* \phi_1 = \lambda_1 \phi_1, \quad\gamma_{_+} \phi_1 = \RRR^*  \gamma_{_-} \phi_1. 
\eeqn 
 We observe then that, by taking $b(x, v)=v$, $c(x, v)=d - \sum_{n=1}^N  \eta_n \Ind_{\Omega_n}(x)$ and since we have that $\Lambda(\Omega)\subset [\Lambda_0, \Lambda_1]$, this equation fits the framework developed in \cite{CGMM24} (see also \cite{CM_Landau_domain}) with $\gamma = 2$, $b_0 = b_1 =1$ and $k_p = d$ for all $p\in [1,\infty]$, thus by repeating its arguments we have the following theorem summarizing the results from  \cite[Theorems 1.1, 1.2, 5.2, 6.1 and Proposition 4.10]{CGMM24}.
 
 \begin{theo}\label{theo:LocalKFP}
 Let $\omega$ be an admissible weight function, then for any $f_0\in L^p_\omega(\OO)$ with $p\in [1,\infty]$, there exists $f\in C(\R_+ , L^p_\omega(\OO))$, unique global weak solution to the local KFP Equation \eqref{eq:KFPtauLocal} in the sense of distributions (see also \cite[Proposition 3.3]{CGMM24}). 
 Moreover, there are constants $\kappa \geq 0$ and $C>0$ such that for all $p\in [1,\infty]$ there holds
 \be
 \lvv f_t \rvv_{L^p_\omega(\OO)} \leq Ce^{\kappa t} \lvv f_0\rvv_{L^p_\omega(\OO)}\qquad \forall t\geq 0.  \label{eq:GrowthKFPLp}
 \ee
Additionally, the following statements hold.

\smallskip\noindent
 {\sl (1)} Let $\omega^\star$ be an admissible weight function such that either $s>0$ or $k > K+k^*$ if $s=0$ with $K :=  4(3d+1)(2d+3)$. Define $\omega^\star_\infty := (\omega^\star)^{1/2}$ if $s > 0$ or $\omega^\star_\infty := \omega^\star \langle v \rangle^{-K}$ if $s= 0$.
There exist $\kappa, \eta > 0$ such that  any solution $f$ to the local KFP Equation \eqref{eq:KFPtauLocal} satisfies 
\be\label{eq:EstimL1LrWeak_bis}
 \|  f(T,\cdot)    \|_{L^\infty_{\omega^\star_\infty}(\OO)}   
\lesssim e^{\kappa T} T^{-\eta}  \| f_0    \|_{L^1_{\omega^\star}(\OO)} .
\ee

\smallskip\noindent
 {\sl (2)}  Consider a  weak solution $0 \le f \in L^2((0,T) \times \OO) \cap L^2((0,T) \times \Omega; H^1(\R^d))$  
to the local KFP Equation \eqref{eq:KFPtauLocal}. For any $0 < T_0 < T_1 < T$ and $\eps > 0$, there holds 
\be\label{eq:Harnack}
\sup_{\OO_\eps} f_{T_0} \le C  \inf_{\OO_{\eps}} f_{T_1}, 
\ee
for some constant $C = C(T_0,T_1,\eps) > 0$ and where we have defined
\be\label{def:OmegaEpsOeps}
\OO_{\eps} := \Omega_{\eps}\times B_{\eps^{-1}},
\quad \Omega_{\eps} :=\{x\in \Omega  , \, \delta(x)   >\eps\} . 
\ee

\smallskip\noindent
 {\sl (3)} There exist two weight functions $\omega_1,m_1$ and an exponent $r > 2$  with $L^r_{\omega_1} \subset (L^2_{m_1})' $ 
such that there exists a unique eigentriplet $(\lambda_1^\BB, f_1^\BB,\phi_1^\BB) \in \R \times L^r_{\omega_1} \times L^2_{m_1}$ satisfying the first eigenproblem  \eqref{eq:1stEVP}.
These eigenfunctions are continuous functions and they also satisfy
\be\label{eq:theoKR-strictpo&Linftybound}
  0 < f_1^\BB \lesssim \varsigma^{-1} \quad \text{ and } \quad 
0 < \phi_1^\BB \lesssim \varsigma   \quad\hbox{on}\quad \OO, 
\ee
for any admissible weight function $\varsigma$.
 \end{theo}

\begin{rem}
Theorem \ref{theo:LocalKFP} implies, in particular, that the family of mappings $S_{\BB}(t) : L^p_\omega \to L^p_\omega$, 
defined by $S_{\BB}(t) f_0 := f(t,\cdot)$ for $t \ge 0$, $f_0 \in L^p_\omega$ and $f_t$ given by Theorem \ref{theo:LocalKFP},  is a positive semigroup of linear and bounded operators. 
\end{rem}

\section{Study of the semigroup $S_{\LL}$}\label{sec:Study_S_LLL}
In this section we study the properties of the semigroup generated by the linear operator $\LL$, namely, well-posedness of the associated evolution equation, control on the growth on the semigroup in weighted Lebesgue spaces, ultracontractivity and existence and properties of the associated dual semigroup.

\subsection{Growth estimates on a weighted $L^2$ framework} 
We define 
\beqn\label{eq:DefBB_0}
\BB_0 f := \CC_\Lambda f -  f\sum_{n=1}^N \eta_n \Ind_{\Omega_n} = \Lambda(x) \Delta_v f +  v\cdot \grad_v f   +  f\left( d- \sum_{n=1}^\NN \eta_n \Ind_{\Omega_n}\right),
\eeqn
and by following the ideas presented in \cite[Subsection 2.1]{CGMM24} for the study of local Kolmogorov equations we first observe that for two functions $h,\omega : \R^d \to \R_+$ and any $p\in [1,\infty)$, we have 
\be\label{eq:identCCCffp}
\int_{\R^d} (\BB_0 h) \, h^{p-1} \omega^p =
  - \frac{4 (p-1)}{ p^2}  \int_{\R^d} |\nabla_v (h\omega )^{p/2} |^2  + \int_{\R^d} h^p \omega^p \varpi^{\BB_0}_{\omega,p},
\ee
where 
\be\label{def:varpi} 
\varpi^{\BB_0}_{\omega, p}(x,v) := 2 \Lambda \left( 1-\frac1p  \right) \frac{|\nabla_v \omega|^2 }{ \omega^2}
+\left( \frac 2p -1\right) \Lambda {\Delta_v \, \omega\over \omega}
- v  \cdot \frac{\nabla_v \omega}{\omega}
+  \left( d- \sum_{n=1}^\NN \eta_n \Ind_{\Omega_n} \right)  - \frac{d }{ p} .
\ee
By choosing $\omega$ to be an admissible weight function, our assumptions during Subsection \ref{ssec:Results} imply that 
$(\varpi_{\omega, p}^{\BB_0})_+ \in L^\infty(\OO)$ and moreover there holds
\beqn\label{eq:varpi-asymptotic}
\limsup _{|v| \to \infty} \bigl( \sup_\Omega \varpi^{\BB_0}_{\omega, p} - \varpi^\sharp_{\omega, p}) \le 0, 
  \quad \text{ where } \quad \varpi^\sharp_{\omega, p} := - b_p^\sharp \langle v \rangle^{s}, 
\eeqn
with  $b^\sharp_p > 0$ given by 
\be\label{eq:b0sharp}
\begin{array}{ll}
 b_p^\sharp :=  k  -  d\left( 1-1/p\right) + \sum_{n=1}^\NN \eta_n \Ind_{\Omega_n}  & \text{ if }  s=0, \\
 b_p^\sharp :=  s \zeta  & \text{ if }  s \in (0,1].
 \end{array}
\ee
In a more quantitative way, for any $\vartheta \in (0,1)$, there exists $\kappa',R' > 0$ such that 
\be\label{eq:varpiC-varpisharp}
 \vartheta \chi_{R'}^c \varpi^\sharp_{\omega, p} \leq \sup_\Omega \varpi^{\BB_0}_{\omega, p} \le \kappa' \chi_{R'} + \vartheta \chi_{R'}^c  \varpi^\sharp_{\omega, p},
\ee
where $\chi_R(v) := \chi(|v|/R)$, $\chi \in C^2(\R_+)$, $\mathbf{1}_{[0,1]} \le \chi \le \mathbf{1}_{[0,2]}$, and $\chi^c_R : 1 - \chi_R$.
Then we have the following lemma.

\begin{lem}\label{lem:GrowthLpPrimal}
Consider an admissible weight function $\omega$. For every $p\in \{1, 2\}$ there are constants $\kappa \geq 0$ and $C\geq 1$ such that for every solution $f(t,x,v)\geq 0$ to Equation \eqref{eq:KFPtau} there holds 
\be\label{eq:GrowthL1Primal}
\lvv f_t \rvv_{L^p_\omega(\OO)} \leq C\, e^{\kappa t} \lvv f_0\rvv_{L^p_\omega(\OO)} \qquad \forall t\geq 0,
\ee
together with the energy estimate on the gradient
\be\label{eq:bddGrad}
 \int_0^t \| \grad_v f_{s} \|^2_{L^2_{\omega}(\OO)} ds \lesssim_C  \| f_0 \|^2_{L^2_{\omega}(\OO)} + 
 \int_0^t \| f_s \|^2_{L^2_{\omega}(\OO)} ds  \qquad \forall t>0.
\ee
\end{lem}

\begin{rem}
The proof consists on the use of a modified weight function, introduced in \cite[Lemma~2.3]{CM-KFP**}, in order to control the terms coming from the Maxwell boundary conditions. We will proceed as during the proof of \cite[Lemma 2.1]{CGMM24} to control the local the operator $\BB$ and we will use Proposition \ref{prop:HHH+Lp-Lp} to control the non-local part $\AA$. 
\end{rem}

\begin{proof}[Proof of Lemma \ref{lem:GrowthLpPrimal}]
We introduce, as during the proof of \cite[Lemma 2.1]{CGMM24}, the modified weight functions $\omega_A$ and $\widetilde \omega$ defined by 
\be\label{def:omegaA}
 \omega_A^p := \MMM_{\Theta}^{1-p} \chi_A + \omega^p \chi_A^c,
 \qquad  \widetilde \omega^p :=   \left( 1 + \frac{ 1  }{ 2}  \frac{ n_x \cdot   v}{\la v \ra^4} \right) \omega_A^p,
\ee
with $A \ge 1$ to be chosen later, $\hat v := v/\langle v \rangle$, $\widetilde{v} :=  \hat v/\langle v \rangle $, and we recall that $n_x$ is the normal vector on $\partial \Omega$ defined in Subsection \ref{ssec:Framework}. It is worth emphasizing that 
\be\label{eq:omega&omegatilde}
c_A^{-1} \omega \le \tfrac12\omega_A \le \widetilde \omega  \le \tfrac32 \omega_A \le c_A \omega, 
\ee
for some constant $c_A  \in (0,\infty)$ depending only on $A$. 
We recall that $\LL f= -v\cdot \grad_x f + \BB_0 f+\AA f$, thus we have that 
\begin{multline}\label{eq1:SLestimL1}
\frac1p\frac{d}{dt} \int_{\OO} f^p \, \widetilde \omega^p 
= \int_{\OO} (\BB_0 f)  f^{p-1} \widetilde \omega^p  
+   \frac1p  \int_{\OO} f^p \, (v\cdot \grad_x \, \widetilde \omega^p) 
- \frac1p  \int_\Sigma (\gamma f)^p \, \widetilde \omega^p   (n_x \cdot v) \\
 + \sum_{n=1}^\NN \eta_n \int_{\Omega_n} \varrho_f \int_{\R^d}  f^{p-1} \widetilde \omega^p \MM_{T_n}.
\end{multline}
We divide then the proof into 6 Steps and we emphasize that Steps 1, 2 and 3 are a repetition of Steps 1, 2 and 3 of the proof of \cite[Lemma 2.1]{CGMM24} thus we only sketch them.

\medskip\noindent
\textit{Step 1.} 
By repeating exactly the arguments from the Step 1 of the proof of \cite[Lemma 2.1]{CGMM24} we immediately have that there is $A>0$ large enough, such that 
$$
- \int_\Sigma (\gamma f)^p \, \widetilde \omega^p   (n_x \cdot v) \le 0.
$$

\medskip\noindent
\textit{Step 2.} 
We now deal with the first term at the right-hand side of \eqref{eq1:SLestimL1}. 
On the one hand from \eqref{eq:identCCCffp}, we have that
$$
\int_{\R^d} (\BB_0 f)  f^{p-1} \widetilde\omega^p   =
 -\frac{4(p-1)}{p^2}  \int_{\R^d} \Lambda(x) |\nabla_v (f^{p/2} \widetilde\omega^{p/2} ) |^2  + \int_{\R^d} f^p \widetilde\omega^p  \varpi^{\BB_0}_{\tilde \omega,p},
$$
with 
\bean
\varpi^{\BB_0}_{\tilde \omega,p}
:=  2 \Lambda \left(1-\frac{1}{p}\right) \frac{|\nabla_v \widetilde \omega|^2 }{ \widetilde\omega^2}  
+  \Lambda \left(\frac 2p -1\right)\frac{ \Delta_v  \widetilde \omega }{\widetilde\omega}  
- v\cdot  \frac{\nabla_v \widetilde\omega}{\widetilde\omega}
+  \left(d-\sum_{n=1}^\NN\eta_n \Ind_{\Omega_n}\right)  - \frac{d}{p}  .
\eean
Following exactly the computations from the Step 2 of the proof of \cite[Lemma 2.1]{CGMM24} we have that
$$
\varpi^{\BB_0}_{\widetilde \omega, p} =  \varpi^{\BB_0}_{ \omega, p} + \WWWW_p \qquad \text{ with } \qquad \WWWW_p = o(\varpi^\sharp_{\omega, p}).
$$
Combining the above estimates together with \eqref{eq:varpiC-varpisharp}, we deduce that for any $\vartheta \in (0,1)$, there exists $\tilde\kappa,\tilde R > 0$ such that 
 \beqn\label{eq:varpiCtilde-varpisharp}
\sup_\Omega \varpi^{\BB_0}_{\tilde\omega,p} \le \tilde\kappa  \chi_{\tilde R} +\vartheta  \chi_{\tilde R}^c  \varpi^\sharp_{\omega,p}. 
\eeqn

\smallskip\noindent
\textit{Step 3.}  
We control then the second term at the right-hand side of \eqref{eq1:SLestimL1}. We first compute
$$
v \cdot \nabla_x (\widetilde \omega^p)
= \frac12 ( \hat v \cdot \grad_x( n_x \cdot \hat v)) \frac{\omega_A^p}{\la v \ra^2}
+ \left( 1 + \frac12 \frac{n_x \cdot v}{\la v \ra^4} \right) v \cdot \nabla_x (\omega_A^p),
$$
and since 
$$
\nabla_x (\omega_A^p) 
=  (p-1) \chi_A \MMM_\Theta^{1-p} \left[ \frac{(d-1)}{2} \frac{\nabla_x \Theta}{\Theta} - \frac{|v|^2}{2} \frac{\nabla_x \Theta}{\Theta^2} \right],
$$
assumption \eqref{eq:Assum-Theta} together with the fact that $\chi_A$ is compactly supported and the regularity assumption on $\Omega$ imply that
$$
v \cdot \nabla_x (\widetilde \omega^p) \lesssim \frac1{ \langle v \rangle^2} \widetilde \omega^p
\lesssim \frac{|\varpi^\sharp_{\omega,p}| }{ \langle v \rangle^{s+2}} \widetilde \omega^p.
$$

\medskip\noindent
\textit{Step 4.} 
We now control the terms coming from the non-local terms. We compute by using the Hölder inequality
\bean
 \int_{\R^d} f &\leq&  \left( \int_{\R^d} f^p \widetilde \omega^p \right)^{1/p}  \left( \int_{\R^d}  \widetilde \omega^{-p'} \right)^{1/p'} \leq c_A^{-1}  \lvv \omega^{-1} \rvv_{L^{p'}(\R^d)}  \left( \int_{\R^d} f^p \widetilde \omega^p \right)^{1/p} \\
 \int_{\R^d} f^{p-1} \widetilde \omega^p \MM_{T_n}  &\leq &\left( \int_{\R^d} f^{p} \widetilde \omega^p \right)^{1-1/p} \left(\int_{\R^d} \widetilde \omega^p \MM_{T_n}^p\right)^{1/p} \leq c_A \lvv \omega \MM_{T_n}\rvv_{L^p(\R^d)}  \left( \int_{\R^d} f^{p} \widetilde \omega^p \right)^{1-1/p} 
\eean 
where $p' = p/(p-1)$, with the convention $1/0 =\infty$, is the conjugate of $p$, and we have used \eqref{eq:omega&omegatilde} to obtain the previous estimates. We then define the constant 
$$
\varpi_{\GGG, p}=  \lvv \omega^{-1} \rvv_{L^{p'}(\R^d)} \sum_{n=1}^\NN \eta_n \lvv  \omega \MM_{T_n}\rvv_{L^p(\R^d)} <\infty,
$$
and we have that
\bean
\sum_{n=1}^\NN \eta_n  \int_{\Omega_n} \varrho_f \left( \int_{\R^d} f^{p-1} \widetilde\omega^p \MM_{T_n} \right)   &\leq&  \sum_{n=1}^\NN \eta_n  \int_\Omega \left( \int_{\R^d} f\right)  \left( \int_{\R^d} f^{p-1} \widetilde \omega^p \MM_{T_n} \right) \\
&\leq& \varpi_{\GGG, p} \int_\OO  f^p \widetilde \omega^p,
\eean
where we have used the above estimates to deduce the second line.

\medskip\noindent
\textit{Step 5.} 
Coming back to \eqref{eq1:SLestimL1} and using Steps 1, 2, 3 and 4, we deduce that 
\be\label{eq:disssipSL-LpPrimal}
\frac1p\frac{d}{dt} \int_\OO f^p \widetilde \omega^p \le  -\frac{4(p-1)}{p^2}  \Lambda_0  \int_\OO |\nabla_v (f\widetilde \omega)^{p/2}|^2 + \int_\OO f^p \widetilde \omega^p \varpi^{\LL}_{\tilde\omega,p}
\ee
with 
\beqn\label{eq:disssipSL-LpBIS}
\varpi^\LL_{\tilde\omega,p} := \varpi^{\BB_0}_{\tilde\omega,p} +  \frac 1p \frac{1 }{ \widetilde \omega^p} v \cdot \nabla_x \widetilde \omega^p + \varpi_{\GGG, p}. 
\eeqn
Gathering the estimate from \eqref{eq:varpiC-varpisharp} and those established in Step~2 and Step~3, we deduce that for any $\vartheta \in (0,1)$, there are $\kappa,R > 0$ such that 
\be\label{eq:varpiL-varpisharp}
  \varpi^{\LL}_{\tilde\omega,p} \le \kappa \chi_{R} +\vartheta \chi_{R}^c  \varpi^\sharp_{\omega,p}.
\ee
In particular, choosing $\vartheta = 1/2$ we have that $  \varpi^{\LL}_{\tilde\omega,p}  \le \kappa$ and we immediately conclude \eqref{eq:GrowthL1Primal}, thanks to Gr\"onwall's lemma and the equivalence between $\omega$ and $\widetilde\omega$ given by \eqref{eq:omega&omegatilde}. 

\medskip\noindent
\textit{Step 6.} 
Coming back now to \eqref{eq:disssipSL-LpPrimal} in the case $p=2$ and integrating in the time interval $(0, t)$ we have that
\be\label{eq:AprioriGrad_1}
\frac12 \int_\OO f^2_t \widetilde \omega^2 \dv\dx  +  \Lambda_0  \int_0^t \int_\OO |\nabla_v (f_s\widetilde \omega)|^2 \dv\dx\d s\leq \frac12 \int_\OO f^2_0 \widetilde \omega^2 \dv\dx + \kappa   \int_0^t   \int_\OO f^2_s \, \widetilde \omega^2 \dv\dx\d s, 
\ee
where we recall that $\kappa>0$ is given by \eqref{eq:varpiL-varpisharp}. Using the triangular inequality we observe now that 
\be\label{eq:AprioriGrad_2}
\lv \grad_v(f\widetilde \omega)\rv^2 \geq \lv \grad_v f\rv^2 \widetilde \omega^2 - \left\lv {\grad_v \widetilde \omega \over \widetilde \omega} \right\rv^2 \, f^2 \widetilde \omega^2. 
\ee
Putting together \eqref{eq:AprioriGrad_1} and \eqref{eq:AprioriGrad_2} we deduce that
\be\label{eq:disssipSL-L2}
  \Lambda_0  \int_0^t \int_\OO |\nabla_v f_s  |^2 \widetilde \omega^2 \leq \frac12 \int_\OO f^2_0 \widetilde \omega^2  + \left( \kappa + \Lambda_0 \left\lvv {\grad_v \widetilde \omega\over \widetilde \omega} \right\rvv_{L^\infty(\OO)}^2 \right)   \int_0^t   \int_\OO f^2_s \, \widetilde \omega^2 .
\ee
Defining, $\wp^2 := 1 +  (n_x \cdot v)/(2\la v \ra^4)$ and $\wp_A^2 := 1 + \chi_A (\MMM_\Theta^{-1} \omega^{-2} - 1)$ so that $\widetilde\omega = \wp \omega_A$ and $\omega_A =   \wp_A \omega$,
we observe that 
\be\label{eq:ControlGradTilde}
\left\lv \grad_v \widetilde \omega \over \widetilde \omega \right\rv \leq \left\lv \grad_v  \omega \over \omega \right\rv + \left\lv \grad_v  \wp \over \wp \right\rv  + \left\lv \grad_v  \wp_A \over \wp_A \right\rv \lesssim 1
\ee
where we have used \eqref{eq:cond-Lpomega} and the fact that $\chi_A$ has compact support, see for instance the proof of \cite[Lemma 2.1]{CGMM24}. 
Putting together the previous computations with \eqref{eq:disssipSL-L2} and using again \eqref{eq:omega&omegatilde} we conclude \eqref{eq:bddGrad}. 
\end{proof}

\subsection{Well-posedness of the kinetic Fokker-Planck equation with BGK thermostats}\label{ssec:KolmogorovPrimal} 
We obtain in this subsection the well-posedness of Equation \eqref{eq:KFPtau}. 
It is worth remarking that existence results for Kolmogorov type equations presenting Maxwell boundary conditions in the context of kinetic equations have been deeply studied in recent years. In particular we construct our solutions using the existence results from \cite[Section 2]{CM_Landau_domain}, but we also refer to \cite[Section 11]{sanchez:hal-04093201}, \cite[Section 3]{CGMM24}, \cite[Section 4.1]{MR2721875} and \cite{MR2072842} for further references on the subject.  

\smallskip\noindent
We denote the boundary measures
\be\label{def:WeightedBoundaryMeasures}
\d\xi_\omega^1 := \omega^2 \lv n_x\cdot v\rv \d\sigma_x\dv , \quad \d\xi_\omega^2 := \omega^2 \la v\ra^{-2} ( n_x\cdot v)^2 \d\sigma_x\dv ,
\ee
where $\d\sigma_x$ represents the Lebesgue measure on the boundary set $\partial \Omega$, we define $\BBBB$ as the set of renormalizing functions $\beta \in W^{2, \infty}_{loc}(\R)$ such that $\beta''\in L^\infty(\R)$, and we consider the Hilbert space $\HHH_\omega$ associated to the Hilbert norm $\lvv \cdot \rvv_{\HHH_\omega}$ defined by
\be\label{def:ExistenceHilbert}
\lvv g\rvv_{\HHH_\omega}^2 := \lvv g\rvv_{L^2_\omega}^2 + \lvv g\rvv_{H^{1,\dagger}_\omega}^2,  \quad \text{ with }\quad \lvv g\rvv_{H^{1,\dagger}_{ \omega}}^2 := \int \left\{ \Lambda \lv \grad_v (f \omega) \rv^2 + \la \varpi^\sharp_{\omega, 2}\ra \, \omega^2 g^2  \right\}.
\ee
Then we have the following well-posedness result.

\begin{theo}\label{theo:WellPosednessL2}
For any admissible weight function $\omega$ and for any $f_0\in L^2_\omega(\OO)$, there exists a unique global renormalized solution $f\in C(\R_+, L^2_\omega(\OO))\cap \HHH_\omega(\UU)$ to Equation \eqref{eq:KFPtau} associated to the initial datum $f_0$. More precisely, $f$ satisifies \eqref{eq:KFPtau} in the renormalized sense, i.e there holds
\begin{multline}\label{eq:paraLp_RenormalizationFormula}
\int_\OO \beta(f_t) \, \varphi (t, \cdot) \,\d v\d x + \int_0^t \int_\OO \beta(f) \left[ -\partial_t \varphi - \BB^* \varphi\right] -\beta'(f) \, \varphi \, \GGG f  \, \d v \d x \d t \\
+\int_0^t \int_\OO   \Lambda \beta''(f) \lv \grad_v f \rv^2 \varphi  \, \d v \d x \d t + \int_0^t \int_\Sigma \beta(\gamma f) \, \varphi (n_x \cdot v) \, \d v\d\sigma_x \d t =\int_\OO \beta(f_0) \, \varphi (0, \cdot) \,\d v\d x,
\end{multline}
for any $t >0 $, any test function $\varphi \in \DD(\bar\UU)$, any $\beta \in \BBBB$ and where we have defined the formal dual operator $\BB^* \varphi := v\cdot \grad_x \varphi + \Lambda \Delta_v \varphi -v\cdot \grad_v \varphi$. 
Moreover, we emphasize that $\gamma f$ is given by \cite[Theorem 2.8]{CM_Landau_domain} and it satisfies $\gamma f\in L^2(\Gamma, \d\xi_\omega^2 \d t)$, as well as the Maxwell boundary condition \eqref{eq:BoundaryConditions} point-wisely. Likewise, it is worth remarking that $f(0,\cdot )=f_0$ also holds point-wisely. 
Furthermore, $f$ satisfies the conclusions of Lemma \ref{lem:GrowthLpPrimal}.
\end{theo}

\begin{rem}\label{rem:WellPosednessL2}
Theorem \ref{theo:WellPosednessL2} in particular implies that we may associate a strongly continuous in time semigroup, that we denote during the sequel by $S_\LL :L^2_\omega(\OO) \to L^2_\omega(\OO)$, to the solutions of the Equation \eqref{eq:KFPtau}.
\end{rem}

\begin{proof}[Proof of Theorem \ref{theo:WellPosednessL2}]
We split the proof into three steps. 

\medskip\noindent
\emph{Step 1. (Modified weight function)} We take $\widetilde \omega$ as defined in \eqref{def:omegaA} and we remark that $\widetilde \omega = \theta \omega$ with   
$$
\theta (x,v) :=  \left( 1 + \frac{ 1  }{ 2}  \frac{ n_x \cdot   v}{\la v \ra^4} \right) \left( \MMM_{\Theta}^{-1} \omega^{-2} \chi_A +  (1-\chi_A) \right),
$$
and we readily observe that there holds 
\be\label{eq:BoundsVartheta}
\frac 12 \leq \left(1-\frac 12\right) \left( \left[ \MMM_{\Theta}^{-1} \omega^{-2}  -1 \right]\chi_A + 1\right) \leq \theta \leq \frac 32 \left( \MMM_{\Theta}^{-1}(A) +1\right).
\ee
where we have used that $\MMM_{\Theta}^{-1} \omega^{-2} \geq 1$ due to our hypothesis on $\omega$. 
Moreover, recalling that $\omega(v)= \la v\ra^ke^{\zeta\la v\ra^s}$ we compute
\bean
\grad_x \theta&=& \frac 12 {\grad_x (n_x\cdot v)\over \la v\ra^4} \left( \MMM_{\Theta}^{-1} \omega^{-2} \chi_A +  (1-\chi_A)\right) \\
&&+  \left( 1 + \frac{ 1  }{ 2}  \frac{ n_x \cdot   v}{\la v \ra^4} \right)  \omega^{-2} \chi_A \MMM_{\Theta}^{-1}  \left( -\frac{(1+d)}2 {\grad_x\Theta\over \Theta} - {\lv v\rv^2 \grad_x\Theta\over 2\Theta^2}\right), \\
\grad_v \theta &=& \frac 12 {n_x \la v \ra^2 -4 v \, (n_x\cdot v) \over \la v\ra^6} \left( \MMM_{\Theta}^{-1} \omega^{-2} \chi_A +  (1-\chi_A)\right) \\
&&+ \left( 1 + \frac{ 1  }{ 2}  \frac{ n_x \cdot   v}{\la v \ra^4} \right) \left[ \omega^{-2}  \MMM_{\Theta}^{-1} \left( \chi_A {v\over 2 \Theta} + \grad_v \chi_A -2\chi_A \left( {k\over \la v\ra^2} + \zeta s \la v\ra^{s-2} \right)  \right) -\grad_v \chi_A \right].
\eean
From the previous computations and recalling that $\chi_A$ has compact support in the ball of radius $2A$, we deduce that
\bean
\lv\grad_x \theta\rv &\leq& \la v\ra^{-1} \left( \frac12 \lvv \delta\rvv_{W^{2,\infty}}  \left(  \MMM_{\Theta}^{-1}(2A) +1\right)  + \frac 34 \la 2A\ra^3  \MMM_{\Theta}^{-1}(2A) \lvv\Theta\rvv_{W^{1,\infty}} \left( \frac{(1+d)} { \Theta_*} + {1 \over \Theta_*^2}\right)  \right)\\
\lv\grad_x \theta\rv &\leq& \la v\ra^{-1} \left( \frac52  \left(  \MMM_{\Theta}^{-1}(2A) +1\right)  \right.\\
&&\left. + \frac {3\la 2A\ra} 2 \left(\MMM_{\Theta}^{-1}(2A) \lvv \chi_A\rvv_{W^{1,\infty}}  \left( {\la A\ra\over 2\Theta_*} + 1 +2\left( k+\zeta s  \right)\right) +1\right) \right)
\eean
where we have used that $\delta\in W^{2,\infty}$ as defined during Subsection \ref{ssec:Framework}. Finally, from the above computations together with \eqref{eq:BoundsVartheta} and our assumptions on the regularity of $\Omega$, we deduce that
\be\label{eq:Condition_Omega_Grad}
\lv \grad_x \theta\rv + \lv \grad_v \theta\rv \lesssim \theta \la v\ra^{-1}.
\ee
 This concludes Step 1.

\medskip\noindent
\emph{Step 2. (Well-posedness)} We recall now that from the Step 2 of the proof of Lemma \ref{lem:GrowthLpPrimal} we have that the function $\varpi^{\BB_0}_{\widetilde \omega, 2}$, defined in \eqref{def:varpi}, satisfies 
\beqn\label{eq:HptOmegaTildeVarpi}
\varpi^\sharp_{\omega, 2} \leq \varpi^{\BB_0}_{\widetilde \omega, 2} \leq \kappa + \varpi^\sharp_{\omega, 2},
\eeqn
for some constant $\kappa >0$. Furthermore, Proposition \ref{prop:HHH+Lp-Lp} implies that
\beqn\label{eq:CondHL2}
\underset{ (0,T)\times \Omega}{\sup} \lvv \AA\rvv_{\BBB(L^1_v(\omega))} <\infty, \qquad \underset{ (0,T)\times \Omega}{\sup} \lvv \AA\rvv_{\BBB(L^2_v(\omega))} <\infty, 
\eeqn
and from the Step 1 of the proof of Lemma \ref{lem:GrowthLpPrimal} we further have that the boundary collision operator $\RRR$ satisfies the bound
\be\label{eq:ConditionL2RRR}
\RRR:L^2 (\Sigma_+, \, \d\xi_{\widetilde \omega} ^1) \to L^2 (\Sigma_-, \, d\xi_{\widetilde \omega} ^1) , \quad \lvv \RRR\rvv_{L^2(\Sigma_+, \, \d\xi_{\widetilde \omega} ^1)} \leq 1.
\ee
We remark that, from its very definition there holds $\varpi^{\BB_0}_{1, 1} \leq 0$, and also due to the hypothesis on admissible weight functions we have that
\beqn\label{eq:WeightCompatibility}
\left( \Lambda + \lv v\rv + 2d +\sum_{n=1}^\NN \eta_n \Ind_{\Omega_n}\right) \omega^{-1} \in L^2(\UU).
\eeqn
The above informations together with \eqref{eq:BoundsVartheta} and \eqref{eq:Condition_Omega_Grad} imply that we may use \cite[Theorem 2.11]{CM_Landau_domain} and \cite[Theorem 2.8]{CM_Landau_domain}, and we deduce that there exists $f\in C(\R_+, L^2_\omega(\OO))\cap \HHH_\omega (\UU)$ unique renormalized solution to Equation \eqref{eq:KFPtau} with an associated trace function $\gamma f\in L^2(\Gamma, \d\xi^2_\omega \d t)$ satisfying \eqref{eq:paraLp_RenormalizationFormula}.

\medskip\noindent
\emph{Step 3. (Energy estimates)} We obtain the validity of \eqref{eq:GrowthL1Primal} with $p=2$ and \eqref{eq:bddGrad} as a consequence of \eqref{eq:paraLp_RenormalizationFormula} with $\beta(s) = s^2$ and $\varphi = \widetilde \omega^2\chi_R$, for any $R>0$, repeating the computations performed during the proof of Lemma \ref{lem:GrowthLpPrimal}, passing $R\to \infty$ and using the integral version of the Grönwall lemma instead. Moreover, \eqref{eq:GrowthL1Primal} for $p=1$ is obtained in a similar fashion by taking instead $\beta(s) = s$ and $\varphi = \widetilde \omega \chi_R$. 
\end{proof}

\subsection{Ultracontractivity}
We establish now the ultracontractive properties of the semigroup $S_\LL$, which will be obtained from the interplay, via the Duhamel formulation, between the ultracontractive properties of the semigroup $S_{\BB}$ and the boundedness of the BGK non-local term $\AA$.

\begin{prop}\label{prop:Ultra}
 Let $\omega$ be an admissible weight function such that either $s > 0$ or $s=0$ and $k > K + k^*$ for $K :=  4(3d+1)(2d+3)$. Define $\omega_\infty$ as during Theorem \ref{theo:LocalKFP}-(1), i.e  $\omega_\infty := (\omega)^{1/2}$ if $s > 0$ or $\omega_\infty := \omega \langle v \rangle^{-K}$ if $s= 0$. 
There are constants $\kappa, \eta > 0$ such that for any solution $f\geq 0$ to Equation \eqref{eq:KFPtau} there holds 
\be\label{eq:Ultra}
 \|  f(T,\cdot)    \|_{L^\infty_{\omega_\infty}(\OO)}   
\lesssim e^{\kappa T} T^{-\eta}  \| f_0    \|_{L^1_{\omega}(\OO)} \qquad \forall T>0 .
\ee
\end{prop}

\begin{proof}
\medskip\noindent
We recall the splitting $\LL = \BB+ \AA$ and using the Duhamel formula we have
 $$
S_\LL = S_{\BB}+ (S_{\BB} \AA) * S_\LL = S_{\BB}+  S_\LL* (\AA S_\BB)  .
$$
Iterating this formula we further deduce that there holds
$$
S_\LL = \VV + \WW*S_\LL *(\AA S_\BB), 
$$
with 
$$
\VV := S_{\BB} + \dots + (S_{\BB} \AA)^{*(N-1)} * S_{\BB}, \quad 
\WW := (S_\BB \AA)^{*N},
$$
and where we define recursively $U^{*k} = U^{*(k-1)}*U$, with the convention $U^{*1} = U$, and some $N\in \N$ to be fixed later. We proceed now in three steps.

\medskip\noindent
\emph{Step 1.} 
We set $\widetilde S_\BB := S_\BB \AA$ and from \eqref{eq:GrowthKFPLp} and Proposition \ref{prop:HHH+Lp-Lp} we have that there are constants $\kappa_1>0$ and $C_1\geq 1$ such that for every $p\in \{1,\infty\}$ there holds
\be\label{eq:UltraStilde1}
\lvv \widetilde S_\BB (t) f_0 \rvv_{L^p_{\omega_\infty}(\OO)} \leq C_1e^{\kappa_1 t} \lvv f_0 \rvv_{L^p_{\omega_\infty}(\OO)} \qquad \forall t\geq 0.
\ee
Furthermore, Theorem \ref{theo:LocalKFP}-(1) implies that there are constants $\kappa_2, \eta ,C_2^0 >0$ for which there holds
\be\label{eq:UltraStilde2}
\lvv \widetilde S_\BB (t) f_0 \rvv_{L^\infty_{\omega_\infty}(\OO)} \leq C_2^0 t^{-\eta} e^{\kappa_2 t} \lvv  \AA f_0 \rvv_{L^1_{\omega}(\OO)}  \leq C_2 t^{-\eta} e^{\kappa_2 t} \lvv  f_0 \rvv_{L^1_{\omega_\infty}(\OO)} \qquad \forall t > 0,
\ee
for some constant $C_2>0$ and where we remark that we have used Proposition \ref{prop:HHH+Lp-Lp} to obtain the second inequality.
Remarking now that $\WW = (\widetilde  S_\BB )^{*N}$, and using \eqref{eq:UltraStilde1} and \eqref{eq:UltraStilde2} we may apply \cite[Proposition 2.5]{MR3465438} and we have that we can set $N\in \N$ such that 
\be\label{eq:UltraStildeWW}
\lvv (\widetilde S_\BB)^{*N} (t) f_0 \rvv_{L^\infty_{\omega_\infty}(\OO)} \leq C_3 e^{\kappa_3 t} \lvv  f_0 \rvv_{L^1_{\omega_\infty}(\OO)} \qquad \forall t\geq 0,
\ee
for some constant $C_3>0$ and any $\kappa_3 >\max(\kappa_1, \kappa_2)$. Using now \eqref{eq:UltraStildeWW} and once again Proposition \ref{prop:HHH+Lp-Lp} we further deduce that
\bean
\lvv \WW*S_\LL *(\AA S_\BB) (t) f_0 \rvv_{L^\infty_{\omega_\infty}(\OO)} &\leq& \int_0^t \int_0^s \lvv \WW \rvv_{\BBB(L^1_{\omega_\infty}, L^\infty_{\omega_\infty})} (t-s) \lvv S_\LL \rvv_{\BBB(L^1_{\omega_\infty})} (s-r)\\
&&  \lvv \AA \rvv_{\BBB(L^1_{\omega}, L^1_{\omega_\infty})} \lvv S_\BB \rvv_{\BBB(L^1_{\omega})} (r) \lvv f_0\rvv_{L^1_\omega} drds  \\
&\lesssim & e^{\kappa_3 t} \lvv f_0\rvv_{L^1_\omega},
\eean
where we have successively used \eqref{eq:UltraStildeWW}, Lemma \ref{lem:GrowthLpPrimal}, Proposition \ref{prop:HHH+Lp-Lp} and Theorem \ref{theo:LocalKFP}, and we have chosen $\kappa_3 = \max (2\kappa_1, 2\kappa_2, \kappa_4)$ where $\kappa_4$ is given by Lemma \ref{lem:GrowthLpPrimal}.

\medskip\noindent
\emph{Step 2.} 
Now we analyze the rest of the terms, on the one hand Theorem \ref{theo:LocalKFP}-(1) immediately gives that there is a constant $C_4>0$ for which there holds
\be\label{eq:UltraSBB1}
\lvv  S_\BB (t) f_0 \rvv_{L^\infty_{\omega_\infty}(\OO)} \leq C_4 t^{-\eta} e^{\kappa_2 t} \lvv  f \rvv_{L^1_{\omega}(\OO)} \qquad \forall t > 0.
\ee
On the other hand we set $\VV_j := (\widetilde S_{\BB} )^{*(j-1)} * S_{\BB}$ for $j\in \llbracket 2, N\rrbracket$ and, from \cite[proof of Proposition 2.5, Equation (2.9)]{MR3465438} we have that for every $p\in \{ 1,\infty\}$ there holds
\be\label{eq:UltraInductionHypo}
\lvv  \VV_j (t) f \rvv_{L^p_{\omega_\infty}(\OO)} \leq   {C^j t^{j-1}\over (j-1)! } \, e^{\kappa_3 t} \lvv f \rvv_{L^p_{\omega_\infty}(\OO)} \qquad \forall t\geq 0, 
\ee
where we have defined $C=\max(C_1, C_2, C_4, C_5)$ with $C_5>0$ given by \eqref{eq:GrowthKFPLp} in Theorem \ref{theo:LocalKFP}. \\

Furthermore, following then the same ideas as in the proof of \cite[Proposition 2.5]{MR3465438}, we will prove now by induction that 
\be\label{eq:UltraInductionHypo}
\lvv  \VV_j (t) f \rvv_{L^\infty_{\omega_\infty}(\OO)} \leq  p_j (t) \, t^{ -\eta} e^{\kappa_2 t} \lvv f \rvv_{L^1_{\omega}(\OO)} \qquad \forall t\geq 0, 
\ee
for some polynomial function $p_j$.

\medskip\noindent
\emph{Step 2.1 (Base case).} 
In particular, for $j=2$ we have that $\VV_2 = \widetilde S_{\BB}*S_{\BB}$ and we compute
\bean
\lvv \widetilde S_{\BB}*S_{\BB} (t)  f_0 \rvv_{L^\infty_{\omega_\infty}(\OO)} &\leq&  \int_0^{t/2} \lvv \widetilde  S_{\BB} (t-s) S_{\BB}(s)  f_0 \rvv_{L^\infty_{\omega_\infty}(\OO)} + \int_{t/2}^t \lvv \widetilde S_{\BB} (t-s) S_{\BB}(s)  f_0 \rvv_{L^\infty_{\omega_\infty}(\OO)} \\
&&=: I_1 + I_2,
\eean
and we bound then each integral separately. We compute for the first one
\bean
I_1 &\leq & C_2 \int_0^{t/2}  (t-s)^{-\eta} e^{\kappa_2 (t-s)} \lvv  S_{\BB}(s)  f_0   \rvv_{L^1_{\omega}(\OO)} \d s \leq (C_1C_2) e^{\kappa_3 t}  \lvv f_0 \rvv_{L^1_{\omega}(\OO)}   \int_0^{t/2}  (t-s)^{-\eta}  \d s\\
& \leq&   {C^2 \over 2^{1-\eta} (1-\eta)} \, t^{1-\eta} e^{\kappa_3 t}  \lvv f_0 \rvv_{L^1_{\omega}(\OO)},
\eean
where we have successively used \eqref{eq:UltraStilde1}, Lemma \ref{lem:GrowthLpPrimal} and the very definitions of $\kappa_3$ and $C$. Moreover we similarly compute for the second integral as follows
\bean
I_2 &\leq &  C_1  \int_{t/2}^t  e^{\kappa_1 (t-s)} \lvv  S_{\BB}(s)  f_0   \rvv_{L^\infty_{\omega_\infty}(\OO)} \d s \leq  C^2 e^{\kappa_3 t}  \lvv f_0 \rvv_{L^1_{\omega}(\OO)}   \int_{t/2}^t  s^{-\eta}  \d s \\
&\leq&    {C^2 \over 2^{1-\eta} (1-\eta)} \, t^{1-\eta} e^{\kappa_3 t}  \lvv f_0 \rvv_{L^1_{\omega}(\OO)},
\eean
where we have successively used \eqref{eq:UltraStilde2}, \eqref{eq:UltraSBB1} and the very definitions of $\kappa_3$ and $C$. Altogether this implies that 
$$
\lvv \widetilde  \VV_2 f_0 \rvv_{L^\infty_{\omega_\infty}(\OO)} \leq  {C^2 \over 2^{-\eta} (1-\eta)} \, t^{1-\eta} e^{\kappa_3 t}  \lvv f_0 \rvv_{L^1_{\omega}(\OO)}. 
$$
and we observe in particular that $p_2 (t) =  t\, C^2 / (2^{-\eta}(1-\eta))$. 

\medskip\noindent
\emph{Step 2.2 (Induction step).} 
We then assume \eqref{eq:UltraInductionHypo} holds for some $j \in \llbracket 2, N-2\rrbracket$ and we will prove it for $j+1$ following an argument similar as above. 
We observe that $\VV_{j+1} = \widetilde S_\BB * \VV_j $, we compute 
\bean
\lvv \widetilde S_\BB * \VV_j \rvv_{L^\infty_{\omega_\infty}(\OO)} &\leq&  \int_0^{t/2} \lvv \widetilde  S_{\BB} (t-s) \VV_j (s)  f_0 \rvv_{L^\infty_{\omega_\infty}(\OO)} + \int_{t/2}^t \lvv \widetilde S_{\BB} (t-s) \VV_j (s)  f_0 \rvv_{L^\infty_{\omega_\infty}(\OO)} \\
&&=: I_1^j + I_2^j,
\eean
and we bound then each integral separately. We compute for the first one
\bean
I_1^j &\leq & C_2 \int_0^{t/2}  (t-s)^{-\eta} e^{\kappa_2 (t-s)} \lvv  \VV_j (s)  f_0   \rvv_{L^1_{\omega}(\OO)} \d s \leq {C^{j+1} \over (j-1) !} e^{\kappa_3 t}  \lvv f_0 \rvv_{L^1_{\omega}(\OO)}   \int_0^{t/2}  (t-s)^{-\eta} s^{j-1} \d s\\
& \leq&   {C^{j+1} t^{j-1} \over 2^{1-\eta} (j-1)! (1-\eta)} \, t^{1-\eta} e^{\kappa_3 t}  \lvv f_0 \rvv_{L^1_{\omega}(\OO)},
\eean
where we have successively used \eqref{eq:UltraStilde1}, Lemma \ref{lem:GrowthLpPrimal} and the very definitions of $\kappa_3$ and $C$. Moreover we similarly compute for the second integral as follows
\bean
I_2^j &\leq &  C_1  \int_{t/2}^t  e^{\kappa_1 (t-s)} \lvv  \VV_j (s)  f_0   \rvv_{L^\infty_{\omega_\infty}(\OO)} \d s \leq p_n(t) e^{\kappa_3 t}  \lvv f_0 \rvv_{L^1_{\omega}(\OO)}   \int_{t/2}^t  s^{ -\eta}  \d s \\
&\leq&  {p_j(t) \over 2^{1-\eta} (1-\eta)}  \, t^{1 -\eta} e^{\kappa_3 t}  \lvv f_0 \rvv_{L^1_{\omega}(\OO)},
\eean
and we have that $p_j$ is defined recurrently by the formula
$$
p_{j+1} (t) = {C^{j+1} \over 2^{1-\eta} (j-1)! (1-\eta)} \, t^j +   {p_j(t) \over 2^{1-\eta} (1-\eta)}  \, t , \quad \text{ for every $j\geq 2$, with $p_2(t)$ as defined above.} 
$$
This concludes the induction argument and the validity of \eqref{eq:UltraInductionHypo}. 

\medskip\noindent
\emph{Step 3. (Conclussion)} We conclude by putting together the results from Steps 1 and 2 and choosing any $\kappa >\kappa_3$.  
\end{proof}

\subsection{Weak Maximum Principle} 
Combining the positivity of the semigroup $S_{\BB}$ given by Theorem \ref{theo:LocalKFP}, and the fact that the non-local operator $\AA$ is also positive we will prove the following proposition.

\begin{prop}[Weak maximum principle] \label{prop:WeakMaxP}
Let $\omega$ be an admissible weight function and let $f \in L^\infty((0,\infty); L^2_\omega ( \OO)) \cap \HHH_\omega(\UU)$ be a solution to Equation \eqref{eq:KFPtau} associated to the initial data $0\leq f_0 \in L^2_\omega(\OO)$. There holds $f_t\geq 0$, for every $t>0$.
\end{prop}
\begin{proof}
We consider $f^0=0$ and then we define recurrently $f^k\in L^\infty((0,\infty); L^2_\omega(\OO))\cap \HHH_\omega(\UU)$ as the solution of the following equation
\be\label{eq:WMP_k}
\left\{\begin{array}{rcll} 
 \partial_t f^{k+1} &=& \BB f^{k+1} + \AA f^k& \text{ in } \UU  \\
\displaystyle \gamma_- f^{k+1} &=& \RRR \gamma_+ f^{k+1} &\text{ on } \Gamma_{-} \\
\displaystyle  f^{k+1}_{t=0} &=& f_0 &\text{ in } \OO,
 \end{array}\right.
 \ee
 which is given by Theorem \ref{theo:WellPosednessL2}.
We assume then that $f^k\geq 0$ and by using the Duhamel formula together with Theorem \ref{theo:LocalKFP} -(2) and Proposition \ref{prop:HHH+Lp-Lp}, we immediately deduce that 
$$
f^{k+1}_t = S_{\BB} (t) f_0 + \int_0^t S_{\BB} (t-s) \AA f^k_s ds \geq 0,
$$
due to the fact that $\AA$ is a positive operator. 
Moreover, from the linearity of the operators $\AA$ and $\BB$ we deduce that, in the weak sense, there holds
\be
\left\{\begin{array}{rcll} 
 \partial_t (f^{k+1} -f^k)&=& \BB(f^{k+1}-f^k) + \AA(f^k-f^{k-1})& \text{ in } \UU \\
\displaystyle \gamma_- (f^{k+1}-f^k) &=& \RRR \gamma_+ (f^{k+1}-f^k) &\text{ on } \Gamma_{-} \\
\displaystyle  (f^{k+1}-f^k)_{t=0} &=& 0 &\text{ in } \OO.
 \end{array}\right.
 \ee
We define $\psi^{k+1}:= f^{k+1} -f^k \in L^\infty((0,\infty); L^2_\omega ( \OO)) \cap \HHH_\omega(\UU)$ and, at the level of a priori estimates, by arguing then as during the proof of Lemma \ref{lem:GrowthLpPrimal} we have that 
\begin{multline}\label{eq1:SLestimL1k}
\frac12\frac{d}{dt} \int_{\OO} (\psi^{k+1})^2 \, \widetilde \omega^2
= \int_{\OO} (\BB_0 \psi^{k+1})  \psi^{k+1} \widetilde \omega^2  
+   \frac12  \int_{\OO} (\psi^{k+1})^2 \, (v\cdot \grad_x \, \widetilde \omega^2)  \\
- \frac12  \int_\Sigma (\gamma \psi^{k+1})^2 \, \widetilde \omega^2   (n_x \cdot v) 
 - \int_{\OO} \AA \psi^{k}\,   \psi^{k+1} \widetilde \omega^2.
\end{multline}
where we recall that $\widetilde \omega$ is defined in \eqref{def:omegaA}.
Using the Cauchy-Schwartz inequality and the Young inequality we deduce that 
\be\label{eq:AprioriWMP_k}
\int_{\OO} \AA \psi^{k}\,   \psi^{k+1} \widetilde \omega^2 \lesssim \lvv \AA \psi^k\rvv_{L^2_{\widetilde \omega}(\OO)}^2 + \lvv  \psi^{k+1}\rvv_{L^2_{\widetilde \omega}(\OO)}^2  \lesssim \lvv \psi^k\rvv_{L^2_\omega(\OO)}^2 + \lvv  \psi^{k+1}\rvv_{L^2_\omega(\OO)}^2,
\ee
where we have used \eqref{eq:omega&omegatilde} and \eqref{eq:HHH+regularization} to obtain the second inequality. Coming back now to \eqref{eq1:SLestimL1k}, arguing similarly as during the proof of Lemma \ref{lem:GrowthLpPrimal} and using \eqref{eq:AprioriWMP_k} we have that there is a constant $\kappa >0$ such that 
$$
\lvv \psi^{k+1}_t\rvv_{ L^2_\omega ( \OO)}^2 + \lvv  \psi^{k+1} \rvv_{ H^{1, \dagger}_\omega ( \UU)}^2 \lesssim  \int_0^t e^{\kappa (t-s)} \lvv \psi^k_s\rvv_{L^2_\omega (\OO)}^2  \lesssim e^{\kappa t} \lvv \psi^k_s\rvv_{L^\infty((0, t); L^2_\omega (\OO))}^2 \qquad \forall t\geq 0.
$$
We remark now that the previous estimate is valid for weak solutions either by arguing as during the proof of \cite[Proposition 3.3]{CGMM24} or by using the weak formulation provided by Theorem \ref{theo:WellPosednessL2} taking $\beta (s) = s^2$ and $\varphi = \widetilde \omega^2\chi_R$, repeating the previous analysis, letting $R\to \infty$, and using the integral version of the Grönwall lemma. 

Moreover, by arguing as during the Step 2 of the proof of \cite[Proposition 3.3]{CGMM24} or Step 3 of the proof of \cite[Theorem 2.11]{CM_Landau_domain} (see also Step 3 of the proof of Proposition \ref{prop:KFPL2Perturbed}) we have that
$$
\lvv \gamma \psi^{k+1}\rvv_{L^2(\Gamma, \d\xi^2_\omega\dt)}^2 \lesssim e^{\kappa t} \lvv \psi^k_s\rvv_{L^\infty((0, t); L^2_\omega (\OO))}^2 \qquad \forall t\geq 0
$$

Choosing then $T>0$ small enough we deduce that $f^k, \gamma f^k$ are a Cauchy sequences in the Banach spaces $L^\infty((0,T); L^2_\omega( \OO))\cap \HHH_\omega(\UU)$ and $L^2(\Gamma, \d\xi^2_\omega\dt)$ respectively, thus there are functions $\ffff \in L^\infty((0,T; L^2(\OO))\cap \HHH_\omega(\UU)$ and $\gamma \ffff \in L^2(\Gamma, \d\xi^2_\omega\dt)$ such that 
$$
f^k \to \ffff  \quad \text{strongly in} \quad L^\infty((0,T); L^2_\omega( \OO)) \cap \HHH_\omega(\UU) \quad \text{as} \quad k\to \infty.
$$ 
and 
$$
\gamma f^k \to\gamma \ffff  \quad \text{strongly in} \quad L^2(\Gamma, \d\xi^2_\omega\dt)\quad \text{as} \quad k\to \infty.
$$ 
Furthermore, we remark that as the limit of positive functions we have that $\ffff\geq 0$, and by using the previous convergences we may pass to the limit in the weak formulation associated to Equation \eqref{eq:WMP_k} and we deduce that $\ffff$ is a weak solution to Equation \eqref{eq:KFPtau}. Finally \cite[Theorem 2.8-(3)]{CM_Landau_domain} implies that $\ffff$ is a renormalized solution of Equation \eqref{eq:KFPtau}, thus a solution in the sense of Theorem \ref{theo:WellPosednessL2}. 
By using the uniqueness provided by Theorem \ref{theo:WellPosednessL2} we deduce that $\ffff = f$ a.e and repeating this argument in the time intervals $[jT, (j+1)T]$ for $j\in \N$, we deduce that $S_\LL$ is a positive semigroup.
\end{proof}

\subsection{Study of the dual backwards KFP equation with BGK thermostats in a weighted $L^2$ framework}\label{ssec:KolmogorovDual}
For any finite $T>0$ and a final datum $g_T$, we study in this subsection the dual backwards equation associated to the linear Equation \ref{eq:KFPtau}, 
\be\label{eq:KFPtau*}
\left\{\begin{array}{rcll} 
 -\partial_t g &=& \LL^* g  & \text{ in } \UU_T := (0,T)\times \OO  \\
\displaystyle \gamma_+ g &=& \RRR^* \gamma_- g &\text{ on } \Gamma_{T, +}:= (0,T)\times \Sigma_+ \\
\displaystyle  g_{t=T} &=& g_T &\text{ in } \OO,
 \end{array}\right.
 \ee 
where we have defined the formal dual operator $\LL^* g = v\cdot \grad_x g + \BB^*_0 g + \AA^* g$ with
\beqn\label{eq:DefBB_0*AA*}
\BB^*_0 g:=\Lambda(x) \Delta_v g - v\cdot \grad_v g -  \sum_{n=1}^\NN \eta_n \Ind_{\Omega_n} g, \quad \text{ and } \quad  \AA^* g =  \sum_{n=1}^\NN \eta_n \Ind_{\Omega_n} \int_{\R^d} g\MM_{T_n}\dv .
\eeqn
Moreover, we have defined the dual Maxwell boundary condition operator as follows
\be\label{eq:DefDualMaxwellBC}
\RRR^* g (x,v) = (1-\iota(x))  \SSS g (x,v) + \iota(x) \DDD^* g (x), 
\ee
with
$$ 
 \DDD^* g (x) 
 := \int_{\R^d} g(x,w) \MMM_\Theta(w) (n_x \cdot w)_{-} \, dw, 
$$ 
for any function $g$ with support on $\Sigma_-$. We observe that the dual equation is defined in such a way that two solutions $f$ to the forward Cauchy problem \eqref{eq:KFPtau} and $g$ to the  dual problem  \eqref{eq:KFPtau*} satisfy, at least formally, the usual dual identity
\begin{equation}\label{eq:identite-dualite}
\int_\OO f(T) g_T   
= \int_\OO f_0 g(0).
\end{equation}

\subsubsection{Properties of the dual non-local operator} We now have a result analogue to Proposition \ref{prop:HHH+Lp-Lp} for the dual operator $\AA^*$.

\begin{prop}\label{prop:HHH+Lp-Lp*}
We consider two admissible weight functions $\omega$, $\omega_\star$ and we define $m=\omega^{-1}$ and $m_\star = (\omega_\star)^{-1}$. For any $q \in [1, \infty]$ there holds
\be\label{eq:HHH+regularization*}
\lvv \AA^*g \rvv_{L^q_{m}(\OO)}  \leq C \sum_{n=1}^\NN \lvv g \rvv_{L^p_{m_\star}(\OO_n)} \lesssim \lvv g\rvv_{L^q_{m_\star}(\OO)}.
\ee
for some constant $C\geq 0$.
\end{prop}

\begin{proof} The proof follows similar as that of Proposition \ref{prop:HHH+Lp-Lp}, thus we skip it. 
\end{proof}

\subsubsection{A priori estimates for the dual kinetic Fokker-Planck equation with BGK thermostats}
\begin{lem}\label{lem:GrowthLpDual}
For any admissible weight function $\omega$ there are constants $\kappa>0$ and $C\geq 1$ such that for every $T>0$ and any initial final datum $g_T \in L^2_m (\OO)$ with $m=\omega^{-1}$, the associated solution $g$ to the backwards dual Equation \eqref{eq:KFPtau*} satisfies
\be\label{eq:GrowthL1Dual}
\lvv g(0) \rvv_{L^2_m(\OO)} \leq C\, e^{\kappa T} \lvv g_T \rvv_{L^2_m(\OO)} ,
\ee
together with the following energy estimate on the gradient
\be\label{eq:bddGrad*}
 \int_0^T \| \grad_v g_{s} \|^2_{L^2_{m}(\OO)} ds \lesssim_C  \| g_T \|^2_{L^2_{m}(\OO)} + 
 \int_0^T \| g_s \|^2_{L^2_{m}(\OO)} ds.
\ee
\end{lem}

\begin{rem}
As for Lemma \ref{lem:GrowthLpPrimal}, we emphasize that the proof of Lemma \ref{lem:GrowthLpDual} consists on the introduction of a modified weight function from \cite[Lemma 2.2]{CGMM24}, in order to control the terms coming from the Maxwell boundary conditions. As for the rest of the terms, we repeat the computations from \cite[Lemma 2.2]{CGMM24} to control the terms coming from the local Kolmogorov part of the equation and we use Proposition \ref{prop:HHH+Lp-Lp*} to control the non-local terms.
\end{rem}

\begin{proof}[Proof of Lemma \ref{lem:GrowthLpDual}]
Without loss of generality we may suppose that $m \ge \MMM_\Theta$, otherwise we replace $m$ by $c m$ where $c>0$ is such that $ m \ge c^{-1}  \MMM_\Theta$.
For $A\geq 1$, we introduce the modified weight functions 
\be\label{def:mA}
m_A^2 := \chi_A \MMM_\Theta + \chi_A^c  m^2,
\quad
\widetilde m^2   :=  \left( 1 -  \frac12 \frac{n_x \cdot  v}{\la v \ra^4}  \right) m_A^2,  
\ee
and we emphasize that 
\be\label{eq:m&mtilde} 
\MMM_\Theta \le m_A \le m \quad\hbox{and}\quad
c_A^{-1} m \le \frac{1}{ 2}m_A \le \widetilde m  \le \frac{3}{ 2} m_A \le \frac{3}{ 2}  m,
\ee
for some constant $c_A \in (0,\infty)$ depending on $A$.
We then compute
\begin{multline}\label{eq:dualGrowthL12}
- \frac{1}{ 2}\frac{d}{dt} \int_\OO g^2 \, \widetilde m^2 = \int_{\OO}  g  \, (\BB^*_0 g)\, \widetilde m^2  - \frac{1}{ 2}\int_\OO g^2\, \left( v\cdot \grad_x \widetilde m^2\right) +  \int_{\OO}  g  \, (\AA^* g)\, \widetilde m^2  \\
+  \frac{1}{ 2}\int_\Sigma (\gamma g)^2  \, \widetilde m^2 \, (n_x \cdot v),
\end{multline}
and we proceed to control each of this terms.

\medskip\noindent
\emph{Step 1.} On the one hand, arguing exactly as during the Step 1 of the proof of \cite[Lemma 2.2]{CGMM24} we immediately have that there is $A>0$ large enough for which there holds
$$
\frac{1}{ 2}\int_\Sigma (\gamma g)^2  \, \widetilde m^2 \, (n_x \cdot v) \leq 0.
$$
\medskip\noindent
\textit{Step 2.}
On the other hand, from \eqref{eq:identCCCffp}-\eqref{def:varpi}, we have
$$
\int_{\R^d} (\BB_0^* g)  g \, \widetilde m^2   =
 - \Lambda(x)  \int_{\R^d}  |\nabla_v (g \widetilde m ) |^2  + \int g^2 \  \widetilde m^2  \varpi^{\BB_0^*}_{\tilde m, 2},
$$
with 
$$
 \varpi^{\BB_0^*}_{\widetilde m,2}
=   \Lambda  \frac{|\nabla_v \varphi|^2 }{ \varphi^2}  
+ v  \cdot \frac{\nabla_v \widetilde  m}{\widetilde m}
 -\sum_{n=1}^\NN \eta_n \Ind_{\Omega_n} + \frac d2.
$$
Arguing exactly as in Step~2 of the proof of \cite[Lemma 2.2]{CGMM24}, we can write
$$
 \varpi^{\BB_0^*}_{\tilde m, 2} = \varpi^{\BB_0^*}_{ m, 2} + \MMMM, \quad \text{ with }\quad  \varpi^{\BB_0^*}_{m, 2}  = \varpi^{\BB_0}_{\omega, 2}, \quad \text{ and } \quad  \MMMM = o (|\varpi^{\sharp}_{\omega, 2}|).
 $$

\medskip\noindent
\textit{Step 3.}  We then deal with the third term on the right side of \eqref{eq:dualGrowthL12}. 
We compute by using the Cauchy-Schwartz inequality
\bean
 \int_{\R^d} g \,  \widetilde m^2 &\leq&   \lvv \widetilde m\rvv_{L^2(\R^d)} \lvv g\rvv_{L^2_{\widetilde m}(\R^d)} \leq  \frac 32 \lvv  m\rvv_{L^2(\R^d)}\lvv g\rvv_{L^2_{\widetilde m}(\R^d)}  \\
 \int_{\R^d} g \, \MM_{T_n}  &\leq &  \lvv \MM_{T_n} \widetilde m^{-1}\rvv_{L^{2}(\R^d)} \lvv g\rvv_{L^2_{\widetilde m}(\R^d)} \leq c_A^{-1}  \lvv \MM_{T_n}  m^{-1}\rvv_{L^{2}(\R^d)} \lvv g\rvv_{L^2_{\widetilde m}(\R^d)}
\eean 
where we have used \eqref{eq:m&mtilde} to obtain the second inequalities in each estimate and we recall that $c_A$ is given by \eqref{eq:m&mtilde}. We define the constant 
$$
\varpi_{\GGG}^*=  \frac 32 c_A^{-1} \lvv  m\rvv_{L^2(\R^d)}   \lvv \MM_{T_n}  m^{-1}\rvv_{L^{2}(\R^d)}  ,
$$
and using the above estimates we then obtain
\beqn
 \int_{\OO}  g  \, (\AA^* g)\, \widetilde m^2 = \sum_{n=1}^\NN  \eta_n  \int_{\Omega_n} \left(\int_{\R^d} g\, \widetilde m^2 \right) \left( \int_{\R^d} g \,  \MM_{T_n} \right)  \leq  \varpi_{\GGG}^*  \lvv g\rvv_{L^2_{\widetilde m}(\R^d)}^2.
\eeqn

\medskip\noindent
\textit{Step 4.}
Coming back to \eqref{eq:dualGrowthL12} and using Steps 1, 2 and 3, we deduce that 
\be\label{eq:disssipSL-LpDual}
-\frac12 \frac{d}{dt} \int_\OO g^2 \widetilde m^2 \le  -\Lambda_0  \int_\OO |\nabla_v (g\, \widetilde m)|^2 + \int_\OO g^2 \widetilde m^2 \varpi^{\LL^*}_{\tilde m, 2}
\ee
with 
\beqn\label{eq:disssipSL-LpBIS}
\varpi^{\LL^*}_{\tilde m, 2} := \varpi^{\BB_0^*}_{\tilde m, 2} +  \frac{1 }{ \widetilde m^2} v \cdot \nabla_x \widetilde m^2 + \varpi_{\GGG}^*. 
\eeqn
Gathering the estimates \eqref{eq:varpiC-varpisharp} and those established in Step~2, 3 and Step~4, we deduce that for any $\vartheta \in (0,1)$, there are $\kappa,R > 0$ such that 
\be\label{eq:varpiL-varpisharpDual}
  \varpi^{\LL^*}_{\tilde m, 2} \le \kappa \chi_{R} + \chi_{R}^c \vartheta \varpi^\sharp_{\omega, 2}.
\ee
In particular, $  \varpi^{\LL^*}_{\tilde m, 2}  \le \kappa$ and we obtain \eqref{eq:GrowthL1Dual} by using the Grönwall lemma. We conclude by integrating \eqref{eq:disssipSL-LpDual} in the time interval $(0, T)$ to obtain \eqref{eq:bddGrad*} as during the Step 6 of the proof of Lemma \ref{lem:GrowthLpPrimal}. 
\end{proof}

\subsubsection{Well-posedness of the dual backwards KFP equation with BGK thermostats}
We provide now a well-posedness result for the backwards dual Equation \eqref{eq:KFPtau*}.

\begin{prop}\label{prop:WellPosednessL2*}
We consider an admissible weight function $\omega$, a finite time $T>0$, and a final datum $g_T\in L^2_m$ with $m=\omega^{-1}$. There is $g\in C([0,T], L^2_m(\OO))\cap \HHH_m(\UU)$ unique global renormalized solution to the dual backwards Equation \eqref{eq:KFPtau*} associated to the final datum $g_T$. Furthermore, there is a trace function $\gamma g\in L^2(\Gamma_T; \d\xi_m^2\d t)$, where $\Gamma_T := (0,T)\times \Sigma$, associated to $g$ which is given by \cite[Theorem 2.8]{CM_Landau_domain}. 
Furthermore $g$ satisfies the conclusions of Lemma \ref{lem:GrowthLpDual}.
\end{prop}

\begin{rem}\label{rem:WellPosednessL2*}
In particular we may associate a strongly continuous in time semigroup denoted by $S_{\LL^*} :L^2_m(\OO) \to L^2_m(\OO)$ to the solutions of the dual Equation \eqref{eq:KFPtau*}. 
\end{rem}

\begin{proof}[Proof of Proposition \ref{prop:WellPosednessL2*}] The proof for the well-posedness follows similar lines as in \cite[Proposition 3.4]{CGMM24} and \cite[Theorem 2.11]{CM_Landau_domain}
(see also \cite{MR872231}, \cite{MR2072842}, \cite[Sec. 8 \& Sec. 11]{sanchez:hal-04093201}) and it is thus only sketched. 

\medskip\noindent
\textit{Step 1.} Given $\mathfrak g \in L^2(\Gamma_{T,-}; \d\xi^1_m\d t)$, where $\Gamma_{T,-}:=(0,T)\times \Sigma_-$ and we recall that $\d\xi^1_m$ is defined in \eqref{def:WeightedBoundaryMeasures}, we consider the backwards inflow problem
\begin{equation}\label{eq:linear_gGinflow}
\left\{\begin{array}{rcll}
 -\partial_t g &=& v \cdot \nabla_x g + \Lambda \Delta_v  g - v\cdot \grad_v g  + \GGG^* g   &\text{ in } \UU_T \\
 \gamma_{+} g  &=& \mathfrak g  &\text{ on }  \Gamma_{T, +}  \\
g_{t=T} &=& g_T   &\text{ in }   \OO,
\end{array}
\right.
\end{equation}
where we have defined 
\be\label{eq:Def_GGG_Dual}
\GGG^* g :=  \sum_{n=1}^\NN \eta_n \Ind_{\Omega_n} \left(\int_{\R^d} \MM_{T_n} \, g\right) - g \sum_{n=1}^\NN \eta_n \Ind_{\Omega_n}  ,
\ee
the formal dual operator of $\GGG$.

We introduce the bilinear form
$\EEE : \HHH_m (\UU_T) \times C_c^1((0,T] \times \OO \cup \Gamma_+) \to \R$, defined by \bean
\EEE(g,\varphi) 
&:=& \int_{\UU_T} g ( \partial_t   + v \cdot \nabla_x   ) (\varphi\widetilde m^2) + \int_{\UU_T} g \left( \lambda  - \BB_0 - \AA \right) (\varphi \widetilde m^2)   , 
\eean
which is coercive for $\lambda$ large enough thanks to Lemma~\ref{lem:GrowthLpDual} (and more precisely \eqref{eq:disssipSL-LpDual}-\eqref{eq:varpiL-varpisharpDual}). 
Using Lions' variant of the Lax-Milgram theorem \cite[Chap~III,  \textsection 1]{MR0153974}, we have that there is a variational solution $g \in \HHH_m(\UU_T)$ to \eqref{eq:linear_gGinflow}, and more precisely $g$ satisfies
$$
\EEE(g,\varphi) = \int_{\Gamma_{T,+}} \mathfrak{g}  \varphi \widetilde m^2 \, \d\xi_1^1  + \int_\OO g_T\,  \varphi(T,\cdot) \widetilde m^2  \, \d v \, \d x, \quad \forall \, \varphi \in C_c^1((0,T] \times \OO \cup \Gamma_+).
$$
We then remark that $\GGG^* g \in L^2$, and applying the trace theorem \cite[Theorem 2.8]{CM_Landau_domain} we deduce that $g\in C([0,T];L^2_m(\OO)) \cap  \HHH_m(\UU_T)$ and that $g$ is a renormalized solution of Equation \eqref{eq:linear_gGinflow}.

\medskip\noindent
\textit{Step 2.} For a sequence $\alpha_k \in (0,1)$, such that $\alpha_k \nearrow 1$, we consider a sequence $(g_k)$ of solutions to the modified Maxwell boundary condition problem
\begin{equation}\label{eq:linearDual_gak}
\left\{\begin{array}{rcll}
 -\partial_t g_k &=& v \cdot \nabla_x g_k + \Lambda \Delta_v g_k + v\cdot \grad_v g_k + \GGG^* g_k &\text{ in } \UU_T \\
 \gamma_{+} g_k   &=&  \alpha_k \RRR^* \gamma_{-} g_{k}   &\text{ on } \Gamma_{T, +} \\
g_k(t=T, \cdot) &=& g_T  &\text{ in }   \OO.
\end{array}
\right.
\end{equation}
We observe that the Step~1 of the proof of Lemma~\ref{lem:GrowthLpDual} implies that 
$$
\RRR^* : L^2 (\Sigma_{-};\d\xi^1_{\tilde m}) \to L^2 (\Sigma_{+}; \d\xi^1_{\tilde m}), \quad \lvv \RRR^*\rvv_{L^2(\Sigma_-, \, \d\xi_{\widetilde m} ^1)} \leq 1,
$$ 
and using the Step 1 and a Banach fixed point argument as during the Step 2 of the proof of \cite[Theorem 2.11]{CM_Landau_domain} (see also the Step 2 of the proof of Proposition \ref{prop:KFPL2Perturbed}) we deduce the existence of a function $g_k\in  C([0,T], L^2_m(\OO)) \cap \HHH_m(\UU_T)$, unique weak solution of Equation \eqref{eq:linearDual_gak} with an associated trace $\gamma g\in L^2(\Gamma_T, \d\xi^2_m\d t)$ given by \cite[Theorem 2.8]{CM_Landau_domain}. Furthermore, arguing as during the Step 1, we may apply \cite[Theorem 2.8]{CM_Landau_domain} and we further have that $g_k$ is a renormalized solution of Equation \eqref{eq:linearDual_gak}. 

\smallskip\noindent
By taking then the renormalized formulation associated to Equation \eqref{eq:linearDual_gak} with $\beta(s) = s^2$ and $\varphi = \widetilde m^2\chi_R$ for any $R>0$, we may repeat the computations performed during the proof of Lemma~\ref{lem:GrowthLpDual} and we obtain, taking $R\to \infty$ and using the integral version of the Grönwall lemma that there is $\kappa_0 >0$ such that this sequence satisfies  
\be\label{eq:EnergyEstimatek}
 \| g_{k0} \|^2_{L^2_{\tilde m}(\OO)} + \int_0^T \left\{ \| \gamma g_{ks} \|^2_{L^2(\Sigma_+; \d\xi^2_{\widetilde m})}
+   
 \| g_{ks} \|^2_{H^{1,\dagger}_{\tilde m} (\OO)} 
\right\} \,  \d s 
\le
e^{\kappa_0 T}  \| g_T \|^2_{L^2_{\tilde m}(\OO)}  ,
\ee
and the gradient energy estimate 
\be\label{eq:bddGrad*k}
 \int_0^T \| \grad_v g_{ks} \|^2_{L^2_{m}(\OO)} ds \lesssim_C  \| g_T \|^2_{L^2_{m}(\OO)} + 
 \int_0^T \| g_{ks} \|^2_{L^2_{m}(\OO)} ds, 
\ee
for any $k \ge 1$ and a constant $C>0$ independent of $k$.
We may then extract converging subsequences $(g_{k'})$ and $(\gamma g_{k'})$ with associated limits $g$ and $\bar\gamma$ satisfying \eqref{eq:EnergyEstimatek} and \eqref{eq:bddGrad*k}, and passing to the limit in the weak formulation associated to Equation \eqref{eq:linearDual_gak}, with the help of the stability result \cite[Proposition 2.9]{CM_Landau_domain}, we deduce that $\bar\gamma = \gamma g$ and that $g$ is a weak solution to Equation \eqref{eq:KFPtau*}. Additionally we emphasize that this implies the validity of Lemma~\ref{lem:GrowthLpDual}.

\medskip\noindent
\textit{Step 3.} 
We consider now two solutions $g_1$ and $g_2 \in  C([0,T];L^2_m(\OO)) \cap \HHH_m(\UU_T)$ to the backwards dual Equation \eqref{eq:KFPtau*} associated to the same final datum $g_T$. We remark then, by the linearity of the problem, that the function $G := g_2 - g_1 \in  C([0,T]; L^2_m(\OO)) \cap \HHH_m(\UU_T)$ is  a solution to the backwards Equation \eqref{eq:KFPtau*} associated to the initial datum $G(T) = 0$, thus \eqref{eq:GrowthL1Dual}, which was obtained as part of Step 2, implies that $G\equiv 0$.
\end{proof}

\subsubsection{Rigourous dual relation between the semigroups $S_\LLL$ and $S_{\LLL^*}$} We now prove that the semigroups $S_\LLL$ and $S_{\LLL^*}$ are dual toward one another. 

\begin{prop}\label{prop:DualityRigorous}
The semigroups $ S_{\LLL}$ and $ S_{\LLL^*}$ are dual one toward the other. In other words, for any admissible weight function $\omega$ and any finite $T>0$, there holds \eqref{eq:identite-dualite} for any $f_0 \in L^2_\omega(\OO)$ and $g_T \in L^2_m(\OO)$, where $m=\omega^{-1}$.
\end{prop}

\begin{proof}
The proof follows by repeating the ideas of the Step 3 of the proof of \cite[Theorem 3.5]{CGMM24} thus we only sketch it. 

We denote $f_n$ the solution to Equation \eqref{eq:KFPtau} by replacing the boundary condition by the modified Maxwell boundary conditions
$$
\gamma_- f_n = \RRR_n \gamma_+ f := \left(1-\frac 1n \right)(1-\iota) \SSS \gamma_+ f_n + \left(1-\frac 1n \right) \iota \DDD \gamma_+ f_n \qquad \text{ on } \Gamma_-,
$$
we remark that $f_n$ is given by Theorem~\ref{theo:WellPosednessL2} (see also Step 2 of Proposition \ref{prop:KFPL2Perturbed} or \cite[Step 2 of Theorem 2.11]{CM_Landau_domain}). Moreover, arguing as during the Step 2 of Proposition \ref{prop:KFPL2Perturbed} (see also \cite[Step 2 of Theorem 2.11]{CM_Landau_domain}) we further have that $\gamma f_n \in L^2(\Gamma;\d\xi_\omega^1\d t)$. 
We also define $g_n$ as the solution of the backwards dual equation
\beqn
\left\{\begin{array}{rcll} 
 -\partial_t g_n &=& \LL^* g_n  & \text{ in } \UU_T  \\
\displaystyle \gamma_+ g_n &=& \RRR^*_n \gamma_- g_n := \left(1-\frac 1n \right) (1-\iota) \SSS \gamma_- g_n + \left(1-\frac 1n \right) \iota \DDD^* \gamma_- g_n &\text{ on } \Gamma_{T, +} \\
\displaystyle  g_{t=T} &=& g_T &\text{ in } \OO,
 \end{array}\right.
 \eeqn
which is given by Proposition \ref{prop:WellPosednessL2*}. Furthermore, from the Step 2 of the proof of Proposition \ref{prop:WellPosednessL2*}, we have that $\gamma g_n\in L^2(\Gamma_T; \d\xi^1_m\d t)$. 

Using then the extra boundary estimates and a density argument we have that 
\begin{multline}
\partial_t(f_n g_n) = -v\cdot \grad_x (f_n g_n) + \Lambda\Delta_v(f_n) g_n - \Lambda \Delta (g_n) f_n+ \Div_v(vf_n) g_n \\
+ v\cdot \grad_v (g_n ) f_n + \GGG(f_n ) g_n -\GGG^*(g_n) f_n 
\qquad \text{ in } \DD'(\UU),
\end{multline}
where we recall that $\GGG^*$ is given by \eqref{eq:Def_GGG_Dual}. 
Integrating against $\chi_R$ for any $R>0$, and taking $R\to \infty$, we obtain that
\begin{equation}\label{eq:identite-dualite-fngn}
\int_\OO f_n(T) g_T   
= \int_\OO f_0 g_n(0), \quad \forall \, n \ge 1.
\end{equation}
Since $(f_n)$ is bounded in $L^\infty((0,T);L^2_\omega(\OO)) \cap W^{1,\infty}((0,T);\DD'(\OO))$, we deduce that 
$$
f_n(T) \wto f (T) := S_\LL(T) f_0 \quad \text{ weakly in } L^2_\omega(\OO).
$$
Similarly, we have that
$g_n(0) \wto g(0) := S_{\LLL^*}(T) g_T$ weakly in $L^2_m(\OO)$. We may thus pass to the limit $n\to\infty$ in \eqref{eq:identite-dualite-fngn} and we deduce that \eqref{eq:identite-dualite} holds, which 
exactly means that $(S_{\LLL})^* = S_{\LLL^*}$. 
\end{proof}

\section{Proof of Theorem \ref{theo:SteadySolutionL}} \label{sec:ProofTheo1}

 \subsection{Krein-Rutmann-Doblin-Harris Theorem in a Banach lattice}  
\label{ssec:DHtheorem}
 We now present a general Krein-Rutman-Doblin-Harris result taken from \cite[Section 6]{sánchez2024voltageconductancekineticequation} obtained for a conservative setting, but we also refer to  \cite{CGMM24, MR2857021, sanchez:hal-04093201, MR4534707} and the references therein for more general results and their applications.

\smallskip\noindent
We consider a  Banach lattice $X$, which means that $X$ is a Banach space endowed with a  closed positive cone  $X_+$ (we write $f \ge 0$ if $f \in X_+$ and we recall that $f = f_+ - f_-$ with $f_\pm \in X_+$ for any $f \in X$. We also denote $|f| := f_+ + f_-$). We assume that $X$  is in duality with another Banach lattice $Y$, with  closed positive cone  $Y_+$,  so that the bracket $\langle \phi,f \rangle$ is well defined for any $f \in X$, $\phi \in Y$, and 
 that $f \in X_+$ (resp. $\phi \ge 0$) iff $\langle \psi , f\rangle \ge 0$ for any $\psi \in Y_+$ (resp. iff  $\langle \phi, g \rangle \ge 0$ for any $g \in X_+$), typically $X = Y'$ or $Y = X'$. 
  We write  
 $\psi \in Y_{++}$ if $\psi \in Y$ satisfies $\langle \psi, f \rangle > 0$ for any $f \in X_+ \backslash \{ 0 \}$. 
 
 \smallskip\noindent
 We consider a  positive and  conservative (or stochastic) semigroup $S = (S_t) = (S(t))$ on $X$, that means that  $S_t$ is a bounded linear mapping on $X$ such that 
$$S_t : X_+ \to X_+ \quad \text{ for any } t \ge 0,$$
and there exist  $\phi_1 \in Y_{++}$, $\| \phi_1 \| = 1$,   and a dual semigroup $S^*= S^*_t = S^*(t)$ on $Y$ such that  $S_t^* \phi_1 =  \phi_1$ for any $t \ge 0$. 
  More precisely, we assume that $S^*_t$ is a bounded linear mapping on $Y$ such that $\langle S(t) f, \phi \rangle = \langle  f, S^*(t)\phi \rangle$, for any $f \in X$, $\phi \in Y$ and $t \ge 0$, 
  and in particular $S^*_t : Y_+ \to Y_+$ for any $t \ge 0$.

\smallskip\noindent
We denote by $\LLLL$ the generator of $S$ with domain $D(\LLLL)$. 
 For $\psi \in Y_+$, we define the seminorm
 $$
 [f]_\psi := \langle |f|, \psi \rangle, \quad  \forall \, f \in X,
  $$
and we introduce now some assumptions on the semigroup $S$.

\smallskip\noindent
$\bullet$  First we introduce the strong dissipativity condition 
 \bear
\label{eq:NEWHarris-Primal-LyapunovCond}
\|   S (t) f  \|
&\le& C_0 e^{\lambda t}   \|  f  \| +   C_1 \int_0^t e^{\lambda(t-s)} [  S(s) f ]_{\phi_1}  ds,  
 \eear
for any $f \in X$ and $t \ge 0$, where $\lambda < 0$ and $C_i \in (0,\infty)$, $i=0, 1$.
 
\smallskip\noindent
$\bullet$ Next, we make the slightly relaxed  Doblin-Harris  positivity  assumption 
   \bear
\label{eq:NEWDoblinHarris-primal}
&&  
S_T f \ge  \eta_{\eps,T} g_\eps [S_{T_0} f]_{\psi_\eps}, \quad \forall \, f \in X_+,  
 \eear
 for any  $T \ge T_1 >  T_0 \ge0$ and $\eps > 0$, where $ \eta_{\eps,T} > 0$, and $(g_\eps)$ and $(\psi_\eps)$ are bounded decreasing families of functions of $ X_+ \backslash \{ 0 \}$ and $Y_{+} \backslash \{ 0 \}$ respectively. 

\smallskip\noindent
$\bullet$ We finally assume the compatibility interpolation like condition  
  \bear
\label{eq:NEWHarris-LyapunovCondNpsieps}
&&  
[f]_{\phi_1} \le \xi_\eps  \| f \| + \Xi_\eps [f]_{\psi_\eps}, \ \forall \, f \in X,  \  \eps \in (0,1], 
 \eear
for two families of positive real numbers $(\xi_\eps)$ and $(\Xi_\eps)$ such that $\xi_\eps\searrow0$ as $\eps \searrow 0$.

 \smallskip\noindent
 We refer to \cite[Section 7]{CGMM24} for a detailed discussion about these assumptions. Then we have the following theorem from \cite[Theorem 6.1]{sánchez2024voltageconductancekineticequation}, see also \cite[Theorem 7.1]{CGMM24}.

\begin{theo}\label{theo:KRDoblinHarris}
Consider a semigroup $S$ on a Banach lattice $X$ which satisfies \eqref{eq:NEWHarris-Primal-LyapunovCond}-\eqref{eq:NEWDoblinHarris-primal}-\eqref{eq:NEWHarris-LyapunovCondNpsieps}. 
There exists a unique normalized positive stationary state  $\ffff_1 \in D(\LLLL)$, that is  
\be\label{eq:FirstEigenproblemDH}
\LLLL \ffff_1 = 0, \quad \ffff_1 \ge 0, \quad   \langle \phi_1, \ffff_1 \rangle = 1.
\ee
 Furthermore, there exist some constructive constants $C \ge 1$ and $\lambda_2 > 0$ such that 
\be\label{eq:SteadySolutionLLongTimeB}
\|S(t) f - \langle f,\phi_1 \rangle \ffff_1   \| \le C e^{-\lambda_2 t} 
\|  f - \langle f,\phi_1 \rangle \ffff_1 \|
\ee
for any $f \in X$ and $t \ge 0$.
 \end{theo}

 \subsection{Proof of Theorem \ref{theo:SteadySolutionL}}\label{ssec:KRDHTheo1}
We will now prove our first main result, Theorem \ref{theo:SteadySolutionL}. The proof will be done by verifying that the semigroup $S_\LL$ defined by Theorem \ref{theo:WellPosednessL2}, satisfies the necessary hypothesis to use Theorem \ref{theo:KRDoblinHarris}. It is worth remarking that the verification of the dissipativity and the interpolation conditions will be done similarly as in the Steps 1 and 2 of the proof of \cite[Theorem 1.2]{sánchez2024voltageconductancekineticequation} and the proof of \cite[Theorem 1.2]{CGMM24}.

\smallskip\noindent
During the rest of this subsection we consider any fixed admissible weight function $\omega$ and we define $X=L^2_\omega(\OO)$. Moreover, we remark that since Equation \eqref{eq:KFPtau} conserves mass then we immediately deduce that $\phi_1=1$. 

\subsubsection{Strong dissipativity condition} \label{sssec:Dissipativity_KRDH}
We consider two real constants $M,R>0$ to be specified later and we define the operators 
$$
\AAA f = \AA f + M\chi_R f \qquad \text{ and }\qquad \BBB f = \BB f - M\chi_R f,
$$
where we recall that $\chi_R$ is defined in Section \ref{sec:Study_S_LLL}. We study then the local equation
\be \label{eq:LocalKFPDH}
\partial_t f = \BBB f \qquad \text { in } \UU,
\ee
complemented with the Maxwell boundary conditions \eqref{eq:BoundaryConditions} and the initial datum \eqref{eq:InitialDatum}. We observe that Equation \eqref{eq:LocalKFPDH}-\eqref{eq:BoundaryConditions}-\eqref{eq:InitialDatum} fits the framework developed in \cite{CGMM24}, thus by repeating the arguments leading to the proof of  \cite[Theorem 1.1]{CGMM24} we deduce the existence of a strongly continuous semigroup $S_\BBB$ associated to the solutions of Equation \eqref{eq:LocalKFPDH}-\eqref{eq:BoundaryConditions}-\eqref{eq:InitialDatum} and with generator $\BBB$.

Moreover, by exactly repeating the arguments leading to the proof of Lemma \ref{lem:GrowthLpPrimal} (see also the proof of \cite[Lemma 2.1]{CGMM24}) we deduce that for any $p\in \{1, 2\}$ and for every $a>0$ there are $M, R>0$ such that
\be\label{eq:DHStrongDissipativityBB}
\lvv S_\BBB(t) f_0\lvv_{L^p_\omega(\OO)} \leq C_1 e^{-a t} \lvv f_0\rvv_{L^p_\omega(\OO)} \qquad \forall t\geq 0,
\ee
for some constant $C_1>0$. 
Indeed, by taking $\widetilde \omega$ as defined in \eqref{def:omegaA} we have by using integration by parts 
\bean
\frac 1p \frac d{dt} \int_\OO f^p \widetilde \omega^p &=& \int_{\OO} (\BB_0 f) \, f^{p-1}\,  \widetilde \omega^p  +  \frac 1p \int_{\OO} f^p \widetilde \omega^p \, \left({v\cdot \grad_x \, \widetilde \omega \over \widetilde \omega }\right) -  \frac 1p \int_\Sigma (\gamma f)^p \, \widetilde \omega^p   (n_x \cdot v)  -   \int_{\OO} M\chi_R f^p\\
&\leq &  \int_\OO f^p \, \widetilde \omega^p \, \varpi_{\widetilde \omega, p}^{\BBB} ,
\eean
where we have defined 
$$
 \varpi_{\widetilde \omega, p}^{\BBB}  =  \varpi_{\widetilde \omega, p}^{\BB_0} + {1\over \widetilde \omega} \, v\cdot \grad_x \widetilde \omega - M\chi_R.
$$
Then by proceeding as during the Steps 2 and 3 of the proof of Lemma \ref{lem:GrowthLpPrimal} we deduce that for every $\vartheta \in (0,1)$ there are $\kappa, A>0$ such that 
$$
 \varpi_{\widetilde \omega, p}^{\BBB} = \kappa \chi_A + \vartheta \chi^c_A   \varpi_{ \omega, p}^{\sharp} -M\chi_R,
$$
and we obtain \eqref{eq:DHStrongDissipativityBB} by choosing $\vartheta = 1/2$, $M=-\kappa -a$, $R=A$, and using the Gronwall lemma together with \eqref{eq:omega&omegatilde}. 

Furthermore, using \eqref{eq:DHStrongDissipativityBB} and following the arguments leading to the proof of \cite[Proposition 4.10]{CGMM24} we deduce that, up to possibly choosing $M, R>0$ larger, we also have that there are constants $\nu >0$ and $C_2\geq 1$ such that
\be
 \lvv S_\BBB (t) f_0 \rvv_{L^2_\omega(\OO)} \leq C_2^0  \| S_{\BBB}(t) f_0 \|_{L^\infty_{\omega_1}(\OO)} \le C_2^0 C_2 {e^{-a t} \over t^\nu} \| f_0 \|_{L^1_{ \omega_2}(\OO)}
\quad \forall \, t > 0  , \label{eq:UltraSBN12}
\ee
where we have defined $\omega_1 = (\omega)^2$ and $\omega_2 = (\omega_1)^2$ if $s>0$ or $\omega_2 = \omega_1 \la v\ra^K$ if $s=0$ with $K= 4(3d+1)(2d+3)$, and we have taken the constant $C_2^0= \lvv \omega (\omega_1)^{-1}\rvv_{L^2(\R^d)}$.

We define $\widetilde S_\BBB = S_\BBB \AAA$ and arguing as during \eqref{eq:UltraStilde2} we observe that \eqref{eq:DHStrongDissipativityBB} and \eqref{eq:UltraSBN12} together with Proposition~\ref{prop:HHH+Lp-Lp} imply that for every $p\in \{1,2\}$ there holds
\be
\lvv \widetilde S_{\BBB} (t) f_0 \rvv_{L^p_\omega(\OO)} \leq C_2 e^{ -a t} \lvv f_0 \rvv_{L^p_\omega(\OO)}  \quad \text{ and } \quad \lvv \widetilde S_\BBB (t) f_0 \rvv_{L^2_\omega(\OO)} \leq C_2^0 C_2 {e^{-a t} \over t^\nu} \| f_0 \|_{L^1_{ \omega}}. \label{eq:DHStrongDissipativityBBL1TILDE}
\ee
We then remark that \eqref{eq:DHStrongDissipativityBBL1TILDE} implies that we may use \cite[Proposition 2.5]{MR3465438} and we obtain that there is $N\in \N$ such that 
$$
\lvv (\widetilde S_\BBB)^{*N} (t) f_0 \rvv_{L^2_\omega(\OO)} \leq C_3 e^{-a' t} \| f_0 \|_{L^1_{ \omega}},
$$
for any $a'<a$ and some constant $C_3>0$. 
Using then the iterated Duhamel formula as during the proof of Proposition~\ref{prop:Ultra} we have that
$$
S_\LL = \VVV + \WWW*S_\LL , 
$$
with 
$$
\VVV := S_\BBB + \dots + (S_\BBB \AAA)^{*(N)} * S_\BBB, \quad \text{ and } \quad 
\WWW := (S_\BBB \AAA)^{*(N+1)}.
$$
On the one hand, using \eqref{eq:DHStrongDissipativityBB} and Proposition~\ref{prop:HHH+Lp-Lp} we immediately deduce that 
$$
\lvv \VVV (t) f_0 \rvv_{L^2_\omega(\OO)} \lesssim  e^{-a t} \lvv f_0\rvv_{L^2_\omega(\OO)}.
$$ 
On the other hand, we observe that 
\bean
\lvv \WWW (t) f\rvv_{L^2_\omega(\OO)} &\leq& \int_0^t \lvv (\widetilde S_\BBB \AAA)^{*N} (t-s)  S_\BBB (s) \AAA   f \rvv_{L^2_\omega(\OO)} \d s\\
&\lesssim & \int_0^t e^{-a' (t-s)} e^{-as} \lvv \AA f\rvv_{L^1_\omega(\OO)} \lesssim  e^{-a' t }  \lvv  f\rvv_{L^1(\OO)}
\eean
where we have successively used the very definition of $\WWW$, \eqref{eq:DHStrongDissipativityBBL1TILDE} and Proposition~\ref{prop:HHH+Lp-Lp}. Finally choosing $a=2$ and $a'=1$ and putting together the previous computations we obtain that
$$
\lvv S_\LL f_0\rvv_{L^2_\omega(\OO)} \lesssim e^{-t} \lvv f_0 \rvv_{L^2_\omega(\OO)} + \int_0^t e^{-(t-s)} \lvv S_\LL(s) f_0\rvv_{L^1(\OO)},
$$ 
which is nothing but the strong dissipativity condition \eqref{eq:NEWHarris-Primal-LyapunovCond}.

\subsubsection{Relaxed Doblin-Harris positivity condition.} \label{sssec:Positivity_KRDH}
We fix $0 \le f_0 \in L^2_\omega$ and denote $f_t := S_\LL(t) f_0$, and $\phi_1^{\BB}$ the steady solution of the dual backwards equation of Equation \eqref{eq:KFPtauLocal} given by Theorem \ref{theo:LocalKFP}-(3). 
For any $T>0$ we set $0<T_0<T_1 \leq T$, we fix $\eps>0$ and an admissible weight function $\varsigma$ such that $\omega \prec \varsigma$, and we compute
 \bean
 \varsigma f_{T_1} \geq \phi_1^{\BB} f_{T_1} &=& \phi_1^{\BB}\left( S_{\BB}(T_1-T_0) f_{T_0} + \int_0^{T_1-T_0} S_{\BB}(T_1-T_0-s) \AA f_{T_0+s} \, ds \right)\\
 &\geq& \phi_1^{\BB} S_{\BB}(T_1-T_0) f_{T_0} \, \geq  \left(\inf_{\OO_\eps} \phi_1^{\BB} \right) \left(\inf_{\OO_\eps} S_{\BB}(T_1-T_0) f_{T_0} \right) \mathbf{1}_{\OO_{\eps}} ,
 \eean
 where we have used \eqref{eq:theoKR-strictpo&Linftybound} and the Duhamel formula to obtain the first line and we have used the fact that $f_{T_0+s}\geq 0$ for every $s\in [0, T_1-T_0]$, due to the weak maximum principle established in Proposition \ref{prop:WeakMaxP}, the fact that $\AA$ is a positive operator, and the fact that $S_\BB$ is a positive semigroup (see Theorem \ref{theo:LocalKFP}-(2)) to obtain the second line. 
 
 \smallskip\noindent
 Then using the Harnack inequality from Theorem~\ref{theo:LocalKFP}-(2) and the fact that $0<\phi_1^{\BB}$ is a continuous function in $\OO$, we further deduce that there are constants $C_1 := C_1 (T_0, T_1, \eps)>0$ and $C_2>0$ for which there holds
 \bean
  \varsigma f_{T_1} 
  &\geq & \frac 1{C_2} \left(\sup_{\OO_\eps} \phi_1^{\BB} \right)  \frac1{C_1}\left(\sup_{\OO_\eps} S_{\BB}((T_1-T_0)/2) f_{T_0} \right) \mathbf{1}_{\OO_{\eps}} \\
  &\geq & \frac 1{C_1C_2} \mathbf{1}_{\OO_{\eps}} \sup_{\OO_\eps} \left( \phi_1^{\BB} S_{\BB}((T_1-T_0)/2) f_{T_0} \right)  \\
  &\geq & \frac 1{C_1C_2}  \mathbf{1}_{\OO_{\eps}} \frac{1 }{ |\OO_\eps|} \int_{\OO_\eps}  \phi_1^{\BB}\, S_{\BB}((T_1-T_0)/2) f_{T_0}  \\
 &\geq & \frac 1{C_1C_2}  \mathbf{1}_{\OO_{\eps}} \frac{1 }{ |\OO_\eps|} \int_{\OO_\eps} f_{T_0} \, S_{\BB^*}((T_1-T_0)/2) \phi_1^{\BB}  =  \frac {e^{-\lambda_1 (T_1-T_0)/2}}{C_1C_2}  \mathbf{1}_{\OO_{\eps}} \frac{1 }{ |\OO_\eps|}  \, \langle f_{T_0},  \phi_1^{\BB} \mathbf{1}_{\OO_\eps} \rangle  ,
 \eean
 where $\lambda_1$ is given by Theorem \ref{theo:LocalKFP}-(3), $S_{\BB^*}$ is the dual semigroup of $S_\BB$, which existence is given by \cite[Proposition 3.4]{CGMM24} and the duality relation is given by \cite[Theorem 3.5-(3)]{CGMM24}.
Moreover, to obtain the final equality we have used the fact that $g (t,\cdot) := e^{-\lambda_1 t} \phi_1^{\BB}(\cdot)$ is a solution to the backwards in time problem 
 \beqn\label{eq:KFPtauLocal*}
 -\partial_t g =  \BB^* g  \qquad \text{ in } (0,T)\times \OO, 
 \eeqn
complemented with the dual Maxwell boundary conditions \eqref{eq:DefDualMaxwellBC} and associated to the final datum $g_T = \phi_1^{\BB}$. 
Altogether, we have obtained that 
$$
f_{T_1} \geq \frac {e^{-\lambda_1 (T_1-T_0)/2}} {C_1C_2} \frac{1 }{ |\OO_\eps|}  \,\varsigma^{-1} \mathbf{1}_{\OO_{\eps}} \langle S_{T_0} f_0, \phi_1^{\BB} \mathbf{1}_{\OO_\eps} \rangle  
$$
which is \eqref{eq:NEWDoblinHarris-primal} 
with the constant $\eta_{\eps, T} = e^{-\lambda_1 (T_1-T_0)/2}/(C_1C_2 |\OO_\eps|)$, and the families of functions  $g_\eps = \varsigma^{-1}\mathbf{1}_{\OO_\eps}  \in L^2_\omega(\OO)$ and $ \psi_\eps = \phi_1^{\BB} \mathbf{1}_{\OO_\eps} \in L^2_m(\OO)$, with $m=\omega^{-1}$. 
 
\subsubsection{Interpolation condition.}\label{sssec:Interpolation_KRDH} For any $f \in L^2_\omega$, we have 
\beqn
\int_\OO |f| = \int_{\OO^c_\eps} |f| +  \int_{\OO_\eps} |f| 
\le \Bigl( \int_{\OO^c_\eps} \omega^{-2} \Bigr)^{1/2} \| f \|_{L^2_\omega} + \int |f| {\bf 1}_{\OO_\eps}, 
\eeqn
so that \eqref{eq:NEWHarris-LyapunovCondNpsieps} holds true with 
$\Xi_\eps := 1$, $\xi_\eps := \| {\bf 1}_{\OO^c_\eps} \|_{L^2_\omega} \to 0$ because $\omega^{-1} \in L^2(\OO)$.

\subsubsection{Krein-Rutmann-Doblin-Harris result for Equation \eqref{eq:KFPtau}} 
Due to the results in sub-sub\-sections \ref{sssec:Dissipativity_KRDH}, \ref{sssec:Positivity_KRDH} and \ref{sssec:Interpolation_KRDH} we deduce the following result. 

\begin{prop}\label{prop:SS_Lambda}
There exists a unique normalized positive stationary state $\CCCC$ to Equation \eqref{eq:KFPtau} such that $\lla \CCCC\rra_\OO=1$. Moreover, for any admissible weight function $\varsigma$ there holds
\be\label{eq:SteadySolutionLIneq_tau}
\lvv \grad_v \CCCC \rvv_{L^2_\varsigma(\OO)} <\infty \qquad \text{ and } \qquad \CCCC (x,v) \lesssim (\varsigma(v))^{-1}.
\ee
Furthermore, let $\omega$ be an admissible weight function, for any initial data $f_0\in L^2_\omega(\OO)$ there is a unique global renormalized solution $f \in C(\R_+, L^2_\omega(\OO))$ to Equation \eqref{eq:KFPtau} and there is $\lambda >0$ such that 
\be\label{eq:SteadySolutionLLongTimeB_tau}
\|S(t) f - \lla f_0 \rra_\OO\,  \CCCC   \|_{L^2_\omega(\OO)} \lesssim  e^{-\lambda t} 
\|  f - \lla f_0 \rra_\OO \, \CCCC \|_{L^2_\omega(\OO)}
\ee
for any $t \ge 0$.
\end{prop}

\begin{proof}
The existence of $\CCCC$ is a direct application of Theorem \ref{theo:KRDoblinHarris} to Equation \eqref{eq:KFPtau} by using the results from sub-sub\-sections \ref{sssec:Dissipativity_KRDH}, \ref{sssec:Positivity_KRDH} and \ref{sssec:Interpolation_KRDH}. The estimates \eqref{eq:SteadySolutionLIneq_tau} are a consequence of Lemma \ref{lem:GrowthLpPrimal} and Proposition \ref{prop:Ultra}. Furthermore, the global existence is given by Theorem \ref{theo:WellPosednessL2} and the decay estimate \eqref{eq:SteadySolutionLLongTimeB_tau} is a consequence of Theorem \ref{theo:KRDoblinHarris}.
\end{proof}

\subsubsection{Proof of Theorem \ref{theo:SteadySolutionL}}
We remark that since $\alpha = 0$, Equation \eqref{eq:NonLKFP} coincides with Equation \eqref{eq:KFPtau} by taking $\Lambda=\tau$, thus Proposition \ref{prop:SS_Lambda} implies the existence of $\FFFF^0 \in L^2_{\omega}(\UU)$ unique stationary solution to this linear equation. 
\qed

 \subsection{Decay for the dual semigroup $S_{\LL^*}$}

 In this subsection we deduce a decay estimate for the solutions of the dual Equation \eqref{eq:KFPtau*} in the spirit of Proposition \ref{prop:SS_Lambda} that we present in the following proposition.
 
 \begin{prop}\label{prop:SteadySolutionLDual}
We consider a finite $T>0$, an admissible weight function $\omega$ and we define $m = \omega^{-1}$. There are constants $\lambda, C >0$ such that for any $g$ solution to Equation \eqref{eq:KFPtau*} there holds
\be\label{eq:SteadySolutionLDual}
\lvv g_0 - \la g_T, \CCCC \ra_{L^2(\OO)} \rvv_{L^2_m(\OO)} \leq C e^{-\lambda T} \lvv g_T - \la g_T, \CCCC \ra_{L^2(\OO)} \rvv_{L^2_m(\OO)} \qquad \forall t\geq 0,
\ee
where $\CCCC$ is given by Proposition \ref{prop:SS_Lambda}. 
\end{prop}

\begin{proof}
We remark that 
$$
\la g(0), \CCCC\ra_{L^2(\OO)} = \la S_{\LL^*} g_T, \CCCC\ra_{L^2(\OO)} = \la g_T, S_{\LL} \CCCC\ra_{L^2(\OO)} = \la g_T, \CCCC\ra_{L^2(\OO)},
$$
where we have used the duality relation \eqref{eq:identite-dualite} together with the fact that 1 and $\CCCC$ are steady states of Equation \eqref{eq:KFPtau*} and Equation \eqref{eq:KFPtau} respectively.
 This implies in particular that the function $g(0) - \la g_T, \CCCC\ra_{L^2(\OO)}$ is also a solution of Equation \eqref{eq:KFPtau*} and 
 \beqn
\la g(0) - \la g_T, \CCCC\ra_{L^2(\OO)}, \CCCC\ra_{L^2(\OO)}= \la g(0), \CCCC\ra_{L^2(\OO)} + \lla \CCCC\rra_\OO \, \la g_T, \CCCC\ra_{L^2(\OO)}   = 0,
 \eeqn
 due to the fact that $\lla \CCCC\rra_\OO=1$ as provided by Proposition \ref{prop:SS_Lambda}. In particular, arguing exactly in the same way we also have that
 \be\label{eq:DualConservation}
\la g_T - \la g_T, \CCCC \ra_{L^2(\OO)}, \CCCC\ra_{L^2(\OO)}=  0.
 \ee
We then compute 
 \bean
 \| g(0) - \la g_T, \CCCC \ra \|_{L^2_m(\OO)} 
 &=& \sup_{f_0 \in L^2_\omega, \| f_0 \|_{L^{2}_\omega(\OO)} \le 1} \, \int_\OO f_0 \left( g(0)-  \la g_T, \CCCC\ra\right)
 \\
 &=&\sup_{f_0 \in L^2_\omega, \| f_0 \|_{L^{2}_\omega(\OO)} \le 1} \, \int_\OO f(T) \, \left(  g_T  - \la g_T, \CCCC\ra \right)
 \\
 &=&  \sup_{f_0 \in L^2_\omega, \| f_0 \|_{L^{2}_\omega(\OO)} \le 1} \, \int_\OO \left( f(T) - \lla f_0\rra \CCCC \right)\, \left(  g_T  - \la g_T, \CCCC\ra \right)
 \\
&\le&  \| g_T - \la g_T, \CCCC\ra \|_{L^2_m(\OO)}  \sup_{f_0 \in L^2_\omega, \| f_0 \|_{L^{2}_\omega(\OO)} \le 1} \,  \| f(T)  - \lla f_0 \rra \CCCC \|_{L^{2}_\omega(\OO)} 
 \\
&\lesssim&  \| g_T - \la g_T, \CCCC\ra \|_{L^2_m(\OO)} \sup_{f_0 \in L^2_\omega, \| f_0 \|_{L^{2}_\omega(\OO)} \le 1} \,e^{-\lambda T}   \| f_0  - \lla f_0 \rra \CCCC \|_{L^{2}_\omega(\OO)} \\
&\lesssim&  e^{-\lambda T}  \| g_T - \la g_T, \CCCC \ra \|_{L^2_m(\OO)}, 
\eean
where we have defined $f_T = S_\LL (T) f_0$, and we have successively used the Riesz representation theorem, the duality identity  \eqref{eq:identite-dualite}, the conservation of mass, \eqref{eq:DualConservation}, the Cauchy-Schwartz inequality, and the decay estimate \eqref{eq:SteadySolutionLLongTimeB_tau}.
\end{proof}

\section{Proof of Theorem \ref{theo:SteadySolutionNonL}}\label{sec:Proof_Theo2}
We now consider a constant $\nu \in [0,2 \EE_{\FFFF^0}]$, where $\FFFF^0$ is given by Theorem \ref{theo:SteadySolutionL}, and we study the problem
\be\label{eq:NESS}
\left\{\begin{array}{rcll} 
 \partial_t f &=& -v\cdot \grad_x f + \alpha\CC_{\nu}  f + (1-\alpha) \CC_\tau f + \GGG f  & \text{ on } \UU , \\
\displaystyle \gamma_- f &=& \RRR \gamma_+ f &\text{ on } \Gamma_{-} ,\\
\displaystyle  f_{t=0} &=& f_0 &\text{ on } \OO,
 \end{array}\right.
 \ee
where $\alpha \in (0,1/2)$ and we remark that
 $$
 \alpha\CC_{\nu}  f + (1-\alpha) \CC_\tau f = (\alpha \nu + (1-\alpha) \tau) \Delta_v f + \Div_v(vf) \quad \text{ with }\quad \alpha \nu + (1-\alpha) \tau\in [\tau_0/2, \tau_1 + 2\EEE_{\FFFF^0}].
 $$
We observe that Equation \eqref{eq:NESS} fits the framework developed in Sections \ref{sec:Study_S_LLL} and \ref{sec:ProofTheo1}, with the choice $\Lambda(x) = \alpha \nu + (1-\alpha) \tau(x)$ and we remark that, as as pointed out above, the bounds on $\Lambda$ are independent of $\alpha$ and $\nu$. We then define the map $\FFF: [0,2 \EE_{\FFFF^0}] \to \R$ by
\be\label{eq:DefFFF}
\FFF(\nu) = \EE_{\FFFF_{\nu}^\alpha},
\ee
where $\FFFF_{\nu}^\alpha$ is the steady solution of Equation \eqref{eq:NESS} given by Proposition \ref{prop:SS_Lambda}. 
We will then prove Theorem \ref{theo:SteadySolutionNonL} by proving that there is $\alpha^\star>0$ small enough such that for every $\alpha \in (0,\alpha^\star)$ the map $\FFF$ has a fixed point $\nu^\star$, which in particular will imply that $\FFFF^\alpha_{\nu^\star}$ is a steady solution of the non-linear Equation \eqref{eq:NonLKFP}-\eqref{eq:BoundaryConditions}.

\subsection{Proof of Theorem \ref{theo:SteadySolutionNonL}}
In this subsection we will proceed to verify the necessary hypothesis to deduce the existence of the fixed point for $\FFF$ by using a fixed point theorem.

\subsubsection{Continuity of the map $\FFF$}  We prove the continuity of the map defined on \eqref{eq:DefFFF} for which we will first prove the following lemma.

\begin{lem}\label{lem:ContinuityFFF}
We consider $f_1, f_2$ solutions of Equation \eqref{eq:NESS} associated with $\nu_1, \nu_2 \in [0,2 \EE_{\FFFF^0}]$ respectively. For any admissible weight function $\omega$ there are constants $\kappa\geq 0$ and $C\geq 1$ such that there holds 
\be\label{eq:ContinuityFFF}
\lvv f_{2,t}-f_{1, t}\rvv_{L^2_\omega} \leq   \alpha^2  \lv \nu_1-\nu_2\rv^2 C e^{\kappa t}  \lvv f_0  \rvv_{L^2_\omega} \qquad \forall t\geq 0.
\ee
\end{lem}

\begin{proof}
We take $F = f_1 - f_2$ and we observe that $F$ solves, in the weak sense, the following equation 
 \begin{equation}
\left\{\begin{array}{rcll} \label{eq:FFFcont}
 \partial_t F &=& \displaystyle -v\cdot \grad_x F +\BB_1 F + \AA F+ \alpha (\nu_1-\nu_2) \Delta_v f_2 & \text{ in } \UU  \\
\displaystyle \gamma_- F &=& \RRR \gamma_+ F &\text{ on } \Gamma_{-} \\
\displaystyle  F_{t=0} &=& 0 &\text{ in } \OO,
 \end{array}\right.
 \end{equation}
 where we have defined $\BB_1 f = \CC_{\Lambda_1} f - f \sum_{n=1}^N \eta_n \Ind_{\Omega_n}$
 with 
 $$
 \Lambda_1(x) := \alpha \nu_1 + \tau(x) \quad \text{ and we remark that } \quad  {\tau_0\over 2} \leq \Lambda (x) \leq   2 \EE_{\FFFF^0}+\tau_1 \quad \text{ for all } x\in \Omega. 
$$
At the level of a priori estimates we proceed as during the proof of Lemma \ref{lem:GrowthLpPrimal} and we have that
\beqn
\frac 12 \frac d{dt} \int_\OO F_t^2 \, \widetilde \omega^2 \leq   -{\tau_0\over 2}  \int_{\OO} |\nabla_v (F \widetilde\omega ) |^2  +  \int_\OO \varpi^{\LL_1}_{\widetilde \omega, 2} \, F^2 \widetilde \omega^2+ \alpha (\nu_1-\nu_2)  \int_\OO (\Delta_v f_2) F\,  \widetilde \omega^2 ,
\eeqn
where 
$$
 \varpi_{\widetilde \omega, 2}^{\LL_1}  =  \varpi_{\widetilde \omega, 2}^{\BB_1} +  \left( v\cdot \grad_x \widetilde \omega \over \widetilde \omega \right) +\varpi_\GGG.
$$
Again as during the proof of Lemma \ref{lem:GrowthLpPrimal} we deduce that $\varpi_{\widetilde \omega, 2}^{\LL_1} \leq  \kappa_1$, for some constant $\kappa_1>0$ independent of $\nu_1, \nu_2$ and $\alpha$. 
Using now integration by parts, we compute
\bean
\int_\OO (\Delta_v f_2) F\widetilde \omega^2 &=& -\int_\OO \grad f_2 \cdot \grad (F \widetilde \omega) \, \widetilde \omega - \int_\OO \grad f_2 \cdot \grad_v \widetilde \omega\, (F \, \widetilde \omega)  \\
&\leq& \lvv \grad_v f_2\lvv_{L^2_\omega(\OO)} \lvv \grad_v (F\widetilde \omega)\rvv_{L^2(\OO)} + \int_\OO f_2 \, \grad_v (F\widetilde \omega) \grad_v \widetilde \omega + \int_\OO f_2 \, F\, \widetilde \omega \, \Delta_v \widetilde \omega\\
&\leq &  \lvv \grad_v f_2\lvv_{L^2_\omega(\OO)} \lvv \grad_v (F\widetilde \omega)\rvv_{L^2(\OO)}  + \left\lvv {\grad_v \widetilde \omega \over \widetilde \omega} \right\rvv_{L^\infty(\OO)} \lvv \grad_v(F\widetilde \omega)\rvv_{L^2(\OO)} \lvv f_2 \lvv_{L^2_{\widetilde \omega}(\OO)}  \\
&& +  \left\lvv {\Delta_v \widetilde \omega \over \widetilde \omega} \right\rvv_{L^\infty(\OO)} \lvv F \rvv_{L^2_{\widetilde \omega}(\OO)} \lvv f_2 \lvv_{L^2_{\widetilde \omega}(\OO)} 
\eean
where we have used the Cauchy Schwartz inequality in the second and third line. We recall from the Step 6 of the proof of Lemma~\ref{lem:GrowthLpPrimal} that we may write $\widetilde\omega = \wp \omega_A$ and $\omega_A =   \wp_A \omega$  where $\wp^2 := 1 +  (n_x \cdot v)/(2\la v \ra^4)$ and $\wp_A^2 := 1 + \chi_A (\MMM_\Theta^{-1} \omega^{-2} - 1)$. 
We then have that
$$
\left\lv {\Delta_v \widetilde \omega \over \widetilde \omega} \right\rv \leq \left\lv {\Delta_v \wp \over  \wp}  \right\rv + \left\lv {\Delta_v \wp_A \over  \wp_A}  \right\rv + \left\lv {\Delta_v \omega \over  \omega}  \right\rv  + 2\left\lv {\grad_v \wp\over \wp} \cdot { \grad_v \wp_A\over \wp_4}\right\rv + 2\left\lv {\grad_v \wp\over \wp} \cdot { \grad_v \omega \over \omega}\right\rv + 2\left\lv {\grad_v \omega \over \omega} \cdot { \grad_v \wp_A\over \wp_4}\right\rv,
$$
thus using our hypothesis on $\omega$, the compact support of $\wp_A$, and the very definition of $\wp$ together with \eqref{eq:ControlGradTilde} we deduce that 
\be\label{eq:ControlGradDeltaTilde}
\left\lvv {\grad_v \widetilde \omega / \widetilde \omega} \right\rvv_{L^\infty(\OO)} +  \left\lvv {\Delta_v \widetilde \omega / \widetilde \omega} \right\rvv_{L^\infty(\OO)} \leq C_\omega <\infty,
\ee
for some constant $C_\omega>0$.
Using the previous informations together with the Young inequality we deduce 
\bean
\int_\OO (\Delta_v f_2) F\widetilde \omega^2 &\leq & \frac {a_1}2 \lvv \grad_v (F\widetilde \omega)\rvv_{L^2(\OO)}^2 + \frac 1{2a_1}   \lvv \grad_v f_2\lvv_{L^2_\omega(\OO)}^2  +  \frac{a_2C_\omega }2 \lvv \grad_v(F\widetilde \omega)\rvv_{L^2(\OO)}^2  \\
&& + \frac {C_\omega}{2a_2} \lvv f_2 \lvv_{L^2_{\widetilde \omega}(\OO)}^2   + \frac {a_3 C_\omega}2 \lvv F \rvv_{L^2_{\widetilde \omega}(\OO)}^2 + \frac{C_\omega}{2a_3}  \lvv f_2 \lvv_{L^2_{\widetilde \omega}(\OO)}^2
\eean
for any constants $a_j>0$, $j=1,2, 3$. Then we choose $a_1 = 1/(\alpha \lv \nu_1-\nu_2\rv)$, $a_2 = 1/(\alpha C_\omega\lv \nu_1-\nu_2\rv)$ and $a_3 = 1/(\alpha \lv \nu_1-\nu_2\rv)$ and putting together the previous informations we obtain 
\begin{multline*}
\frac 12 \frac d{dt}\int_\OO F_t^2 \, \widetilde \omega^2 \leq   -{\tau_0\over 2}  \int_{\R^d} |\nabla_v (F \widetilde\omega ) |^2  +  \int_\OO \left( \varpi^{\LL_1}_{\widetilde \omega, 2} +{C_\omega \over 2} \right) \, F^2 \widetilde \omega^2 \\
 + \frac{\alpha^2}2 \lv \nu_1-\nu_2\rv^2  \left[ \lvv \grad_v f_2\lvv_{L^2_{\widetilde \omega}(\OO)}^2  + ( C_\omega^2 + C_\omega) \lvv f_2 \lvv_{L^2_{\widetilde \omega}(\OO)}^2 \right].
\end{multline*}
Using the Grönwall lemma and \eqref{eq:omega&omegatilde} we deduce that
\begin{multline}\label{eq:Continuity_lemma}
\lvv F\rvv_{L^2_{ \omega}(\OO)}^2 \leq  2c_A^2 \frac{\alpha^2}2 \lv \nu_1-\nu_2\rv^2 \int_0^t e^{ \left( \kappa_1 +C_\omega/2 \right) (t-s)} \\
\times \left[ \lvv \grad_v f_{2,s}\lvv_{L^2_{ \omega}(\OO)}  + ( C_\omega^2 + C_\omega) \lvv f_{2,s} \lvv_{L^2_{ \omega}(\OO)} \right] \, \d s.
\end{multline}
We obtain \eqref{eq:ContinuityFFF} by using  \eqref{eq:GrowthL1Primal} with $p=2$ and \eqref{eq:bddGrad} from Lemma \ref{lem:GrowthLpPrimal}.
\\

We conclude by remarking that \eqref{eq:Continuity_lemma} is still valid for weak solutions by arguing as follows. Using \cite[Theorem 2.8]{CM_Landau_domain} (see also for instance the proof of Proposition \ref{prop:KFPL2Perturbed}) we deduce that $F$ is also a renormalized solution of Equation \eqref{eq:FFFcont}. Applying then the renormalized formulation associated to the previous equation with $\beta (s) = s^2$, $\varphi = \widetilde \omega^2\chi_R$ for any $R>0$ and where $\widetilde \omega$ is as defined in \eqref{def:omegaA}, we deduce that arguing as before, passing to the limit as $R\to \infty$, and using the integral version of the Grönwall lemma we obtain \eqref{eq:Continuity_lemma}.
\end{proof}

Then we have the tools to prove the following continuity result.
\begin{prop}\label{prop:SchauderContinuity}
The map $\FFF:\R\to \R$ is continuous. More precisely, for every $\eps>0$ there is $\delta >0$ such that if $\lv \nu_1-\nu_2\rv \leq \delta$ then 
$$
\lv \FFF(\nu_1) - \FFF(\nu_2) \rv \leq \eps.
$$
\end{prop}

\begin{proof}
We fix an admissible weight function $\omega$ and we define $C_0 = d^{-1}\lvv  \la v\ra^2 \omega^{-1}\rvv_{L^2}<\infty$ and $T>0$ to be specified later. There holds
\bean
\lv \FFF(\nu_1)- \FFF(\nu_2)\rv &\leq & \frac 1d \int_\OO \lv v\rv^2 \lv \FFFF^\alpha_{\nu_1} - \FFFF^\alpha_{\nu_2} \rv \quad \leq C_0   \lvv \FFFF^\alpha_{\nu_1} - \FFFF^\alpha_{\nu_2}\rvv_{L^2_\omega(\OO)}\\
&\leq &  C_0 \lvv \FFFF^\alpha_{\nu_1} - f_{1, T} \rvv_{L^2_\omega(\OO)} +  C_0 \lvv f_{1, T} - f_{2, T} \rvv_{L^2_\omega(\OO)} +  C_0 \lvv \FFFF^\alpha_{\nu_2} - f_{2, T} \rvv_{L^2_\omega(\OO)}\\
&\leq & C_0 C_1 e^{-\lambda T} \lvv \FFFF^\alpha_{\nu_1} - f_0 \rvv_{L^2{\omega}(\OO)} +  C_0 C_2\alpha^2  \lv \nu_1-\nu_2\rv^2  e^{\kappa T}  \lvv f_0  \rvv_{L^2_\omega(\OO)} \\
&&+   C_0C_1 e^{-\lambda T} \lvv \FFFF^\alpha_{\nu_2} - f_0 \rvv_{L^2_{\omega}(\OO)}
\eean
where we have used the Cauchy-Schwartz inequality in the first line, the triangular inequality to obtain the second, and Theorem \ref{theo:SteadySolutionL} and Lemma \ref{lem:ContinuityFFF} on the last line and we remark that $C_1, \lambda>0$ are given by Theorem \ref{theo:SteadySolutionL} and $\kappa, C_2>0$ are given by Lemma \ref{lem:ContinuityFFF}. 

\smallskip\noindent
Then for every fixed $\eps>0$ we choose $T$ large enough such that $C_0C_1e^{-\lambda T} \leq \eps/3$ and we choose $\delta = \sqrt{\eps/(3C_0C_2 e^{\kappa T})}$ so that there holds 
$$
\lv \FFF(\nu_1)- \FFF(\nu_2)\rv \leq \eps.
$$
and this completes the proof.
\end{proof}

\subsubsection{Preservation of lower and upper bounds} We dedicate this sub-subsection to prove the following proposition.

\begin{prop}\label{prop:SchauderBounds}
Let $\nu \in [0, 2\EE_{\FFFF^0}]$, there is $\alpha^\star \in (0,1/2)$ such that for every $\alpha \in (0,\alpha^\star)$ there holds $\FFF(\nu) \in [0, 2\EE_{\FFFF^0}]$. 
\end{prop}

\begin{proof}
We fix during the proof an admissible weight function $\omega$ and an arbitrary $0\leq f_0 \in L^2_\omega(\OO)$ with $\lla f_0\rra = 1$. On the one hand from the weak maximum principle provided by Proposition \ref{prop:WeakMaxP} we deduce that
$$
0\leq \EE_f \leq  C_\omega^0 \,  \lvv f\rvv_{L^2_\omega(\OO)},
$$
where $C^0_\omega = d^{-1}\lvv  \la v\ra^2  \omega^{-1}\rvv_{L^2(\R^d)}<\infty$ and we remark that we have used the Cauchy-Schwartz inequality to obtain the second inequality.
We observe that 
$$ 
\alpha \CC_{\nu} f + (1-\alpha) \CC_{\tau} f = \alpha (\nu-\tau)\Delta_v f + \CC_{\tau} f ,
$$ 
and we will compare a solution $f$ to Equation \eqref{eq:NESS} and a solution $\phi$ to the linear equation
\be\label{eq:UpperBoundG}
\left\{\begin{array}{rcll} 
 \partial_t \phi &=& -v\cdot \grad_x \phi + \CC_\tau \phi + \GGG \phi   & \text{ in } \UU , \\
\displaystyle \gamma_- \phi &=& \RRR \gamma_+ \phi &\text{ on } \Gamma_{-} ,\\
\displaystyle  \phi_{t=0} &=& f_0 &\text{ in } \OO.
 \end{array}\right.
 \ee 
 
 \medskip\noindent
 \emph{Step 1.} We first remark that Proposition \ref{prop:SS_Lambda} and Theorem \ref{theo:SteadySolutionL} provide the existence of $\FFFF^\alpha_{\nu}$ and $\FFFF^0$, respective stationary solutions to Equations \eqref{eq:NESS} and \eqref{eq:UpperBoundG}, furthermore we also have the existence of some constants $\lambda, C_1>0$ independent of $\alpha$ and $\nu$ such that 
\be\label{eq:DecayComparisonLUB}
\lvv f_t - \FFFF^\alpha_{\nu}  \rvv_{L^2_{\omega}(\OO)} \leq C_1 e^{-\lambda t} \lvv f_0 - \FFFF^\alpha_{\nu}  \rvv_{L^2_{\omega}(\OO)} \quad \text{ and } \quad  \lvv \phi_t - \FFFF^0  \rvv_{L^1_{\omega}(\OO)}  \leq C_1 e^{-\lambda t} \lvv f_0 - \FFFF^0  \rvv_{L^1_{\omega}(\OO)}, 
\ee
 for every $t\geq 0$. 
 We then set $\psi=f-\phi$ and we observe that $\psi$ is a weak solution of the following kinetic equation
 \be\label{eq:UpperBoundH}
\left\{\begin{array}{rcll} 
 \partial_t \psi &=& -v\cdot \grad_x \psi + \CC_\tau \psi + \GGG \psi + \alpha (\nu -\tau(x)) \Delta_v f  & \text{ in } \UU  \\
\displaystyle \gamma_- \psi &=& \RRR \gamma_+ \psi&\text{ on } \Gamma_{-} \\
\displaystyle  \psi_{t=0} &=& 0 &\text{ in } \OO.
 \end{array}\right.
 \ee

 At the level of a priori estimates, we introduce the modified weight function $\widetilde \omega$ as defined in \eqref{def:omegaA}, and arguing as during the proof of Lemma \ref{lem:GrowthLpPrimal} we have that
 \be\label{eq:EstimateGronwall1}
\frac d{dt} \int_\OO \psi^2_t \, \widetilde \omega^2 \leq   -{\tau_0\over 2} \int_{\OO} |\nabla_v (\psi \widetilde\omega ) |^2  +   \int_\OO \varpi^{\LL}_{\widetilde \omega, 2} \, \psi^2 \widetilde \omega^2 + \alpha  \int_\OO  (\nu -\tau(x))\,  (\Delta_v f)  \psi \, \widetilde \omega^2 ,
\ee
an there is a constant $\kappa_1$, independent of $\nu_1, \nu_2$ and $\alpha$,  such that $\varpi_{\widetilde \omega, 2}^{\LL} \leq  \kappa_1$. 
On the other hand we compute 
\bear
\left\lv \int_\OO (\Delta_v f)  \psi \, \widetilde \omega^2 \right\rv&\leq& \left\lv \int_\OO \grad_v f \cdot \grad_v(  \psi \widetilde \omega) \,  \widetilde \omega \right\rv + \left\lv \int_\OO \grad_v f  \cdot \grad_v \widetilde \omega \,  ( \psi \widetilde \omega) \right\rv \nonumber\\
&\leq & \frac 12 \left( 1 + C_\star\right) \lvv \grad_v f\rvv_{L^2_{\widetilde \omega}(\OO)}^2 +  \frac 12 \int_{\OO} |\nabla_v (\psi \widetilde\omega ) |^2 +    \frac {C_\star}2  \lvv \psi \lvv_{L^2_{\widetilde \omega}(\OO)}^2, \label{eq:EstimateGronwall2}
\eear
 where we have used integration by parts and the triangular inequality to obtain the first line and the Cauchy-Scwhartz inequality together with the Young inequality to obtain the second. Moreover we remark that we have set $C_\star = \lvv (\grad_v \widetilde \omega) / \widetilde \omega\rvv_{L^\infty(\OO)}$, which is finite due to the analysis leading to \eqref{eq:ControlGradTilde}. 
Putting together \eqref{eq:EstimateGronwall1} and \eqref{eq:EstimateGronwall2} we then obtain
\begin{multline*}
\frac d {dt} \int_\OO \psi^2_t \, \widetilde \omega^2 \leq   -\left( {\tau_0\over 2} - \alpha \left( \tau_1 + 2 \EE_{\FFFF^0} \right) \right)   \int_{\OO} |\nabla_v (\psi \widetilde\omega ) |^2  +  \left( \kappa_1 + \frac{C_\star} 2  \right) \int_\OO  \psi^2 \widetilde \omega^2 \\
 + \frac \alpha 2 \left( 1 + C_\star\right) \left( \tau_1 + 2 \EE_{\FFFF^0} \right) \int_0^t \int_\OO    \lv \grad_v f\rv^2 \, \widetilde \omega^2  .
\end{multline*}
We set then $\alpha^\star_1 = \tau_0/(4\left( \tau_1 + 2 \EE_{\FFFF^0} \right))$ and we observe that using \eqref{eq:omega&omegatilde} and the Grönwall lemma we deduce that
\bean
\lvv \psi_t \rvv_{L^2_{ \omega}(\OO)}^2 &\leq & c_A^2 \frac \alpha 2 \left( 1 + C_\star\right) \left( \tau_1 + 2 \EE_{\FFFF^0} \right) \int_0^t \int_\OO e^{\kappa_2(t-s)} \lv \grad_v f_s \rv^2\omega^2 \\
&\leq & c_A^2 \, C\,  \frac \alpha 2 \left( 1 + C_\star\right) \left( \tau_1 + 2 \EE_{\FFFF^0} \right) \, e^{\kappa_2 t}  \left( \lvv f_0 \rvv_{L^2_\omega(\OO)}^2 + \int_0^t \lvv f_s \rvv_{L^2_\omega(\OO)}^2  \right) \\
&\leq &  c_A^2 \, C\,  \frac \alpha 2 \left( 1 + C_\star\right) \left( \tau_1 + 2 \EE_{\FFFF^0} \right) \, e^{\kappa_2 t} \left( 1+ C^2e^{2\kappa t}\right) \lvv f_0 \rvv_{L^2_\omega(\OO)}^2
\eean
where we have used \eqref{eq:bddGrad} to obtain the second inequality, we have used \eqref{eq:GrowthL1Primal} with $p=2$ to obtain the third, and we remark that we have set $\kappa_2 = \kappa_1 + C_\star/2$, and the constants $\kappa, C>0$ are given by Lemma~\ref{lem:GrowthLpPrimal}.
Finally we define the constant 
$$
C_t^2  := c_A^2 \,  \frac C 2 \left( 1 + C_\star\right) \left( \tau_1 + 2 \EE_{\FFFF^0} \right) \, e^{\kappa_2 t} \left( 1+ C^2e^{2\kappa t}\right).
$$
 We conclude this step by arguing that the previous estimate holds for weak solutions of Equation \eqref{eq:UpperBoundG}. Using \cite[Theorem 2.8]{CM_Landau_domain} (see also for instance the proof of Proposition \ref{prop:KFPL2Perturbed}) we deduce that $h$ is also a renormalized solution of Equation \eqref{eq:UpperBoundH}. Applying then the renormalized formulation associated to the previous equation with $\beta (s) = s^2$, $\varphi = \widetilde \omega^2\chi_R$ for any $R>0$, we have that arguing as during this step, passing to the limit as $R\to \infty$, and using the integral version of the Grönwall lemma we obtain the same estimate.

\medskip\noindent
\emph{Step 2.} 
We take $T>0$ to be defined later and we compute by using the Cauchy-Schwartz inequality
 \bean
 \frac 1d \int_\OO \lv v\rv^2 \FFFF^\alpha_{\nu}\dx \dv &\leq&   \frac 1d \int_\OO \lv v\rv^2 \lv \FFFF^\alpha_{\nu} -f_{T} \rv\dx\dv  +  \frac 1d \int_\OO \lv v\rv^2 \lv f_{T} - \phi_{T}  \rv\dx \dv \\
 &&+  \frac 1d \int_\OO \lv v\rv^2 \lv \phi_{T} - \FFFF^0  \rv\dx \dv + \frac 1d \int_\OO \lv v\rv^2 \lv \FFFF^0 \rv\dx \dv \\
 &\leq & C^0_\omega \lvv \FFFF^\alpha_{\nu} -f_{T} \rvv_{L^2_\omega(\OO)} + C^0_\omega \lvv f_{T} -\phi_{T}   \rvv_{L^2_\omega(\OO)}  + C^0_\omega  \lvv \phi_{T} - \FFFF^0   \rvv_{L^2_\omega(\OO)} +  \EE_{\FFFF^0}  \\
 &\leq & C^0_\omega C_1 e^{-\lambda T} \lvv \FFFF^\alpha_{\nu} -f_{0} \rvv_{L^2_\omega(\OO)} +  \alpha^{1/2} C_T  \lvv f_0 \lvv_{L^2_{\omega}(\OO)} + C^0_\omega C_1 e^{-\lambda T}  \lvv f_{0} - \FFFF^0   \rvv_{L^2_\omega(\OO)} \\
 &&+  \EE_{\FFFF^0} 
  \eean
where we have used \eqref{eq:DecayComparisonLUB} to obtain the last inequality and we recall that $C^0_\omega>0$ is defined at the beginning of the proof. 
We set then $T>0$ such that
 $$
C^0_\omega C_1 e^{-\lambda T} \lvv f_0 - \FFFF^\alpha_{\nu}  \rvv_{L^2_{\omega}(\OO)} \leq  \frac {\EE_{\FFFF^0}} 3, \qquad 
C^0_\omega C_1 e^{-\lambda T} \lvv f_0 - \FFFF^0  \rvv_{L^2_{\omega}(\OO)} \leq  \frac {\EE_{\FFFF^0}} 3,
 $$
 we choose $\alpha^\star_2$ small enough such that 
 $$
(\alpha^\star_2)^{1/2} C_{T} \lvv f_0 \lvv_{L^2_{\omega}(\OO)} \leq \frac 13 \EE_{\FFFF^0},
 $$
 and we conclude by setting $\alpha^\star = \min(\alpha^\star_1, \alpha^\star_2)$.  
 \end{proof}

\subsubsection{Proof of Theorem \ref{theo:SteadySolutionNonL}}
We take $\alpha \in (0, \alpha^\star)$, where $\alpha^\star >0$ is given by Proposition \ref{prop:SchauderBounds} and we remark that Proposition \ref{prop:SchauderContinuity} implies in particular that the image set of the map $\FFF$ is a compact set contained in $[0, 2\EE_{\FFFF^0}]$. Using this together with Propositions \ref{prop:SchauderContinuity} and \ref{prop:SchauderBounds}, we may apply a fixed point theorem (for instance a real version of the Schauder fixed point theorem) and we conclude that there is $\nu^\star\in [0, 2\EE_{\FFFF^0}]$ such that $\FFF(\nu^\star) = \nu^\star$.
\qed

\section{Perturbation around the equilibrium. Growth estimate and well-posedness.} \label{sec:PerturbEquilib}

Throughout the sequel we take $\alpha \in (0,\alpha^\star)$ where $\alpha^\star>0$ is given by Theorem \ref{theo:SteadySolutionNonL}, we introduce a function $g:\UU\to \R$ such that $\lla g_t\rra_\OO = 0$ for all $t\geq 0$, and the initial datum 
\be\label{eq:PerturbedInitialDatum}
h_0:\OO\to \R \qquad \text{ such that }\qquad  \lla h_0\rra_\OO = 0,
\ee 
and we will study the following equation
\be\label{eq:KFPstabilityLin}
\left\{\begin{array}{rcll} 
 \partial_t h &=& -v\cdot \grad_x h  +   \CC_{\Lambda^\star} h +\GGG h   + \alpha \EE_{g}\Delta_v h  + \alpha\EE_{h} \Delta_v \FFFF^\alpha  & \text{ in } \UU  \\
\displaystyle \gamma_- h &=& \RRR \gamma_+ h &\text{ on } \Gamma_{-} \\
\displaystyle  h_{t=0} &=& h_0 &\text{ in } \OO,
 \end{array}\right.
 \ee
where $\Lambda^\star = \alpha \EE_{\FFFF^\alpha } + (1-\alpha) \tau$, as introduced in \eqref{eq:KFPstabilityLin1} during Subsection \ref{ssec:First_Observ}. In particular, we observe that, due to Theorem \ref{theo:SteadySolutionNonL}, there holds $\tau_0/2 \leq \Lambda^\star \leq 2\EE_{\FFFF^0} + \tau_1$,
where it is worth remarking that the previous upper and lower bounds of $\Lambda^\star$ are independent on $\alpha$.

 \begin{rem} 
 At a formal level we note that, by arguing as during \eqref{eq:MassConservation}, Equation \eqref{eq:KFPstabilityLin} conserves mass, therefore a solution $h$ to Equation \eqref{eq:KFPstabilityLin} will satisfy that $\lla h_t \rra_\OO=0$ for all $t\geq 0$. 
 \end{rem}
 
\begin{rem}\label{rem:NonLPerturbed}
Moreover, and still at a formal level, if $g=h$ then a solution $h$ of Equation \eqref{eq:KFPstabilityLin} satisfies that $f= \FFFF^\alpha + h$ is a solution to Equation \eqref{eq:NonLKFP}-\eqref{eq:BoundaryConditions}-\eqref{eq:InitialDatum}.
\end{rem}

\smallskip\noindent
We will dedicate the rest of this section to prove the well-posedness of Equation \eqref{eq:KFPstabilityLin} under suitable assumptions. 
During the sequel we will use the notations 
\be\label{eq:DefQQQg}
\QQ_g h = \PP h +\NNNN_g h, \quad \text{ and } \quad \QQQ_gh = \PPP h + \NNNN_g h, 
\ee
where
\be \label{eq:DefPPPg}
  \PPP h = -v\cdot \grad_x h  + \PP h, \quad  \PP h =  \CC_{\Lambda^\star} h +\GGG h, \quad \text{ and } \quad   \NNNN_g h = \alpha \EE_{g}\Delta_v h  + \alpha\EE_{h} \Delta_v \FFFF^\alpha. \\
 \ee

\subsection{A priori growth estimate}
We will prove first that under suitable assumptions for the function $g$ we can control the growth of the solution in time.
\begin{prop}\label{prop:GrowthLpPrimal+NNNN} 
Let $\omega$ be an admissible weight function, there are constants $\eps_1, \kappa, C>0$ such that if $\lvv g\rvv_{L^2_\omega(\OO)} \leq \eps$ for any $\eps\in (0,\eps_1)$, then there is a constant $C>0$ such that for any solution $h$ to Equation \eqref{eq:KFPstabilityLin} there holds 
\be\label{eq:bddGrowthPerturbed}
\lvv h_t \rvv_{L^2_\omega (\OO)} \leq Ce^{\kappa t} \lvv h_0\rvv_{L^2_\omega(\OO)} \qquad \forall t \geq 0,
\ee
together with the energy estimate on the gradient
\be\label{eq:bddGradPerturbed}
 \int_0^t \| \grad_v h_{s} \|^2_{L^2_{\omega}(\OO)} ds \lesssim_C  \| h_0 \|^2_{L^2_{\omega}(\OO)} + 
 \int_0^t \| h_s \|^2_{L^2_{\omega}(\OO)} ds  \qquad \forall t>0.
\ee
\end{prop}

\begin{proof}
The proof follows that of Lemma \ref{lem:GrowthLpPrimal}. We introduce the modified weight function $\widetilde \omega$ as defined in \eqref{def:omegaA} and we write 
\begin{multline}\label{eq1:SQQestimL2}
\frac12\frac{d}{dt} \int_{\OO} h^2 \, \widetilde \omega^2
= \la \PP h,   h\ra_{L^2_{\widetilde \omega}(\OO)}  
+   \frac12  \int_{\OO} h^2 \, (v\cdot \grad_x \, \widetilde \omega^2) 
-  \frac 12 \int_\Sigma (\gamma h)^p \, \widetilde \omega^2   (n_x \cdot v) \\
  + \la \NNNN_g h, h\ra_{L^2_{\widetilde\omega}(\OO)}.
\end{multline}
We split the proof into 3 Steps.

\medskip\noindent
\emph{Step 1. (Control of the terms coming from the operator $\PPP$)}
Repeating the arguments from the Steps 1, 2, 3, 4 and 5 of the proof of Lemma \ref{lem:GrowthLpPrimal} we know that there is $\kappa>0$ such that 
\beqn 
\la \PP h,   h\ra_{L^2_{\widetilde \omega}(\OO)} 
+   \frac12  \int_{\OO} h^2 \, (v\cdot \grad_x \, \widetilde \omega^2) 
-  \frac 12 \int_\Sigma (\gamma h)^p \, \widetilde \omega^2   (n_x \cdot v) 
\leq - \frac {\tau_0}2 \lvv \grad_v(h \, \widetilde \omega)\rvv_{L^2(\OO)} + \kappa \lvv h \rvv_{L^2_{\widetilde \omega}(\OO)}^2.
\eeqn

\medskip\noindent
\emph{Step 2. (Control of the term $\NNNN_g$)} We will prove during this step that for every $a^*>0$ there holds 
\begin{multline*}
\la \NNNN_g h, h\ra_{L^2_{\widetilde \omega}(\OO)} \leq \alpha C_\omega  \left( \lvv g\rvv_{L^2_{ \omega}(\OO)} + {a^* \over 2}  \lvv \grad_v \FFFF^\alpha \rvv_{L^2_{ \omega}(\OO)} \right) \lvv \grad_v (h{\widetilde \omega}) \rvv_{L^2(\OO)}^2 \\
+ \alpha C_\omega \left[\left( 1+ {1\over 2a^*} \right)  \lvv \grad_v \FFFF^\alpha \rvv_{L^2_{ \omega}(\OO)}+ \lvv g\rvv_{L^2_{ \omega}(\OO)}  \right] \lvv h\rvv_{L^2_{\widetilde \omega}(\OO)}^2 ,
\end{multline*}
for some constant $C_\omega >0$.
Indeed we have
$$
\la \NNNN_g h, h\ra_{L^2_{\widetilde \omega}(\OO)}  = \alpha \left( \EE_{g} \la \Delta_v h , h\ra_{L^2_{\widetilde \omega}(\OO)}  +   \EE_{h} \la \Delta_v  \FFFF^\alpha , h\ra_{L^2_{\widetilde \omega}(\OO)} \right) =: \alpha( \NN_1 + \NN_2),
$$
and we will control each term separately. On the one hand using integration by parts we compute
\bean
\NN_1 &=& \EE_g \int_\OO \Delta h\, h\, {\widetilde \omega}^2 = \EE_g \left( -\int_\OO \lv \grad_v (h {\widetilde \omega})\rv^2 + \int_\OO h^2{\widetilde \omega}^2 \, {\lv \grad_v {\widetilde \omega}\rv^2 \over {\widetilde \omega}^2}  \right) \\
&\leq & C_\omega^1\,  \lvv g\rvv_{L^2_\omega(\OO)} \left( \lvv \grad_v (h{\widetilde \omega}) \rvv_{L^2(\OO)}^2 + \left\lvv { \grad_v {\widetilde \omega} \over {\widetilde \omega}}\right\rvv_{L^\infty(\R^d)}^2   \lvv h\rvv_{L^2_{\widetilde \omega}(\OO)}^2     \right),
\eean
where we have set $C_\omega^1 = \lvv v \, \omega^{-1}\rvv_{L^2}$, and we have used the Cauchy-Schwartz inequality to obtain the second line. 
On the other hand we have that
\bean
\NN_2 &=& \EE_h \int_\OO \Delta \FFFF^\alpha\, h\, {\widetilde \omega}^2 = \EE_h \left( -\int_\OO \grad_v \FFFF^\alpha \cdot \grad_v (h {\widetilde \omega}) \, {\widetilde \omega} - \int_\OO (\grad_v \FFFF^\alpha \cdot \grad_v {\widetilde \omega} )\, h {\widetilde \omega} \, \right) \\
&\leq & C_{\widetilde \omega}^1 \lvv h\rvv_{L^2_{\widetilde \omega}(\OO)}\left( \lvv \grad_v \FFFF^\alpha \rvv_{L^2_{\widetilde \omega}(\OO)} \lvv \grad_v (h{\widetilde \omega}) \rvv_{L^2(\OO)} + C_{\widetilde \omega}^2 \lvv \grad_v \FFFF^\alpha \rvv_{L^2_{\widetilde \omega}(\OO)} \lvv  h \rvv_{L^2_{\widetilde \omega}(\OO)}\right)\\
&\leq & {a^* C_{\widetilde \omega}^1 \, c_A  \over 2}  \lvv \grad_v \FFFF^\alpha \rvv_{L^2_\omega(\OO)} \lvv \grad_v (h{\widetilde \omega}) \rvv_{L^2(\OO)}^2 +  c_A C_{\widetilde \omega}^1 \left( C_{\widetilde \omega}^2 + {1\over 2a^*}\right)  \lvv \grad_v \FFFF^\alpha \rvv_{L^2_\omega(\OO)} \lvv h\rvv_{L^2_{\widetilde \omega}(\OO)}^2 ,
\eean
where $C_{\widetilde \omega}^2 = \lvv (\grad_v \widetilde \omega) / \widetilde \omega\rvv_{L^\infty}$ and we have used integration by parts on the first line, the Cauchy-Schwartz inequality to obtain the second and the Young inequality together with \eqref{eq:omega&omegatilde} to obtain the third line. 

\smallskip\noindent
We then remark that $C^1_\omega <\infty$ due to the very definition of $\omega$ and $C^1_{\widetilde \omega} + C^2_{\widetilde \omega} <\infty$ by arguing as during \eqref{eq:ControlGradDeltaTilde}.
Altogether the previous computations imply the inequality presented at the beginning of this step. 

\medskip\noindent
\emph{Step 3. (Conclusion)}
We choose $\eps_1 = \tau_0/(8C_{ \omega})$ and $a^*=\tau_0/(8C_{ \omega} \lvv \grad_v \FFFF^\alpha \rvv_{L^2_{\omega}(\OO)})$ where we recall that $C_\omega$ is given by the Step 2. Then, since $\lvv g\rvv_{L^2_{\omega}(\OO)}\leq \eps$ with $\eps\leq \eps_1$, we deduce that
$$
\frac12\frac{d}{dt} \int_{\OO} h^2 \, \widetilde \omega^2 \leq -\frac {\tau_0}4  \lvv \grad_v (h \widetilde\omega) \rvv_{L^2(\OO)}^2 + (\kappa + \kappa^\star)  \lvv h \rvv_{L^2_{\widetilde\omega}(\OO)}^2,
$$
where 
$$
\kappa^\star = C_{\omega} \left(\left( 1+ {1\over 2a^*} \right)  \lvv \grad_v \FFFF^\alpha \rvv_{L^2_{\omega}(\OO)}+ \lvv g\rvv_{L^2_{\omega}(\OO)} \right).
$$
We remark that $\kappa^\star$ is finite due to \eqref{eq:PropsSS}, we conclude the proof by remarking that \eqref{eq:bddGrowthPerturbed} is a consequence of the Grönwall lemma and \eqref{eq:bddGradPerturbed} is obtained by arguing as during the Step 6 of the proof of Lemma \ref{lem:GrowthLpPrimal}.
\end{proof}

\subsection{Well-posedness in a weighted $L^2$ framework} In this subsection we will extend the well-posedness theory presented in Subsection \ref{ssec:KolmogorovPrimal} to a framework fitting Equation \eqref{eq:KFPstabilityLin}. 

\smallskip\noindent
We note that the main difficulty comes from the presence of a term including $\Delta_v \FFFF^\alpha \in L^2_x H^{-1}_v$. We remark, however, that the main technical tools to achieve such a well-posedness result have been recently developed in \cite[Subsection 2.3]{CM_Landau_domain}, and we will be using them during the proof of the following proposition.

\begin{prop}\label{prop:KFPL2Perturbed}
Let $\omega$ be an admissible weight function. There is $\eps_2>0$ such that if $\lvv g\rvv_{L^2_\omega(\UU)}\leq \eps$ for any $\eps\in (0,\eps_2)$ there holds that for any initial datum $h_0 \in L^2_\omega(\OO)$, there exists a unique global weak solution $h \in C(\R_+, L^2_\omega(\OO)) \cap  \HHH_\omega(\UU)$  to Equation \eqref{eq:KFPstabilityLin}, where we recall that the Hilbert space $\HHH_\omega$ is defined in \eqref{def:ExistenceHilbert}. 
More precisely, for every test function $\varphi\in \DD'(\bar \UU)$ there holds
\begin{multline}\label{eq:Perturbed_renormalized_formulation}
\int_\UU \grad_v\varphi \cdot \grad_v h \left( \Lambda^\star + \alpha\EE_g\right) + \int_\UU h \left(-\partial_t \varphi +  v\cdot \grad_v \varphi - \GGG^* \varphi -v\cdot \grad_x \varphi  \right) \\ -\int_\UU \EE_h \, \grad_v \FFFF^\alpha \cdot \grad_v \varphi + \int_\Gamma \gamma h \, \varphi (n_x\cdot v) = \int_\OO h_0\, \varphi(0,\cdot)
\end{multline}
where we remark that the trace $\gamma h$ is defined by \cite[Theorem 2.8]{CM_Landau_domain} and satisfies $\gamma h \in L^2_\omega(\Gamma,d\xi^2_\omega)$ as well as  the Maxwell boundary condition \eqref{eq:BoundaryConditions} pointwisely. Finally there also holds $h(0,\cdot) = h_0$ pointwisely. 
\end{prop}

\begin{rem}
The proof of Proposition \ref{prop:KFPL2Perturbed} is based on \cite[Theorem 3.3]{CGMM24} and \cite[Theorem 2.11]{CM_Landau_domain} using the a priori estimate from Proposition \ref{prop:GrowthLpPrimal+NNNN} and \cite[Proposition 2.8]{CM_Landau_domain}.
\end{rem}

\begin{rem}\label{rem:Def_S_QQQ_g}
In particular, Proposition \ref{prop:KFPL2Perturbed} implies the existence of a strongly continuous semigroup, which will be denoted as $S_{\QQQ_g}$, associated to the solutions of Equation \eqref{eq:KFPstabilityLin}. 
\end{rem}

\begin{proof}[Proof of Proposition~\ref{prop:KFPL2Perturbed}] We proceed to the proof into four steps.

\medskip\noindent
\textit{Step 1.} We consider $\widetilde \omega$ as defined in \eqref{def:omegaA}, a finite $T>0$, $\HH \in L^2_{t,x, v} + L^2_{t,x} H^{-1}_v$ such that $\HH = \HH_0 + \Div_v \HH_1$ with $\HH_0,\HH_1 \in L^2_{\widetilde \omega}(\UU_T)$, and the boundary data $\HH_b \in L^2(\Gamma_{T,-}; d\xi^1_{\widetilde \omega})$.
We consider the inflow problem
\begin{equation}\label{eq:linear_primal_inflow}
\left\{\begin{array}{rcll}
 \partial_t h &=& - v \cdot \nabla_x h + \widetilde \Lambda \Delta_v h + v\cdot \grad_v h  +\HH \quad &\text{ in } \UU_T  \\
 \gamma_{-} h  &= &\HH_b \quad &\text{ on } \Gamma_{T,-} \\
 h_{t=0} &=& h_0  \quad &\text{ in }   \OO,
\end{array}\right.
\end{equation}
where $\widetilde \Lambda =  \alpha \EE_{\FFFF^\alpha } + \alpha \EE_g + (1-\alpha) \tau$.
By setting $\eps_2^1 = \tau_0 / (2 \lvv \la v\ra^2 \omega^{-1} \rvv_{L^2})$ we observe that if $\lvv g\rvv_{L^2_\omega(\UU)}\leq \eps$ for $\eps\in (0,\eps_2^1)$, there holds
 $$
 {\tau_0 \over 4} \leq \widetilde \Lambda \leq 2\EE_{\FFFF^0} + {\tau_0 \over 2}  + \tau_1.
 $$
We define, for a constant $\lambda>0$, the bilinear form
$\EEE : \HHH_{\widetilde \omega} (\UU_T) \times C_c^1([0,T)\times \OO \cup \Gamma_{T,-}) \to \R$ by 
\beqn
\EEE(h,\varphi) := \int_{\UU_T} \widetilde \Lambda \, \grad_v h \cdot \grad_v (\varphi \widetilde \omega^2)  + \int_{\UU_T} h \left(\lambda \varphi \, \widetilde \omega^2 - (\partial_t \varphi)\,   \widetilde \omega^2 +  \Div_v(v \varphi \, \widetilde \omega^2) -v\cdot \grad_x (\varphi \, \widetilde \omega^2)  \right).
\eeqn
We observe that, due to our choice of $\eps_2^1$, we can follow the same computations as during the proof of Lemma~\ref{lem:GrowthLpPrimal} and we obtain that there is $\lambda_0 \in \R$ such that 
 \beqn\label{eq:EEE-coercive}
\EEE(\varphi,\varphi) 
\ge (\lambda-\lambda_0) \| \varphi \|^2_{L^2_{\tilde\omega}} +  \| \varphi \|^2_{H^{1,\dagger}_{\tilde\omega}} 
, 
\eeqn
for any $\varphi \in C_c^1([0,T)\times \OO \cup \Gamma_-)$. 
Taking then $\lambda>\lambda_0$, the  bilinear form $\EEE$ is coercive and using Lions' variant of the Lax-Milgram theorem from \cite[Chap~III,  \textsection 1]{MR0153974}, we deduce the existence of a function $h_\lambda \in \HHH_{\widetilde \omega} (\UU_T)$
which satisfies the variational equation 
\beqn\label{eq:EEE-cequation}
\EEE(h_\lambda,\varphi) = \int_{\Gamma_{T,-}} \HH_b \, e^{-\lambda t} \, \varphi\,  \widetilde\omega^2 \, \d\xi^1_1  + \int_\OO h_0 \, \varphi(0,\cdot) \, \widetilde\omega^2  \, \d v \, \d x + \int_{\UU_T} \HH \, \varphi \, \widetilde \omega^2   \,\d v\, \d x \, \d t  , 
\eeqn
for every $\varphi \in C_c^1([0,T) \times \OO \cup \Gamma_-)$ and where we remark that the last term is defined as the product between functions in $L^2_{t,x}H^{-1}_v$ and $L^2_{t,x}H^{1}_v$ in $\UU_T$.

\smallskip\noindent
Defining now $h := h_\lambda e^{\lambda t}$ we deduce that $h \in \HHH_{\widetilde \omega}(\UU_T)$ 
is a weak solution to the inflow problem \eqref{eq:linear_primal_inflow} and by using \cite[Proposition 2.8-(3)]{CM_Landau_domain} with the choices $\sigma_{ij} = \widetilde \Lambda$, $\nu = v$ and $G=\HH$, we further have that $h \in C([0,T), L^2_{\widetilde \omega}(\OO))\cap \HHH_{\widetilde \omega}(\UU_T)$ and it has a trace $\gamma h\in L^2(\Gamma_T; \d\xi^1_{\widetilde \omega} \, \d t)$. Another consequence of \cite[Proposition 2.8-(3)]{CM_Landau_domain} is that $h$ is a renormalized solution to Equation \eqref{eq:linear_primal_inflow}, i.e there holds 
\bear\label{eq:FPK-traceL2}
&& \int_{\UU_t} \left( \beta(h) \left( -\partial_t \varphi -v\cdot \grad_x \varphi - \widetilde \Lambda \Delta_v \varphi + \Div_v (v\varphi) \right) + \beta''(h) \, \widetilde \Lambda \lv \grad_v h\rv^2     \varphi \right)  \dv \, \dx \, \d s
 \\ \nonumber
&&\qquad
+  \int_{\Gamma_t} \beta(\gamma \, h) \, \varphi \,  (n_x \cdot {v}) \, \dv \, \d\sigma_{\! x} \, \d s
+ \left[ \int_\OO \beta(h_t)  \varphi(t,\cdot) \, \dx \, \dv \right]_0^t =  \int_{\UU_t} \HH  \, \beta'(h) \,  \varphi \, \d v\, \d x\, \d s
\eear
for any $t\in [0,T]$, any renormalizing function $\beta\in \BBBB$, and any test function $\varphi \in \DD'(\bar \UU_T)$. 

\smallskip\noindent
Using now the renormalized formulation \eqref{eq:FPK-traceL2} with $\beta (s) = s^2$ and $\varphi = \widetilde \omega^2 \chi_R$ for any $R>0$, we proceed similarly as during the proof of Proposition \ref{prop:GrowthLpPrimal+NNNN}, we take the limit as $\R\to \infty$, and using the integral version of the Grönwall lemma we obtain that
$$
\lvv h\rvv_{L^2_{\widetilde \omega}(\UU_t)}^2 \leq e^{\lambda_0 t} \lvv h_0\rvv_{L^2_{\widetilde \omega}(\OO)}^2 + C \int_0^t e^{\lambda_0 (t-s)}  \left(  \lvv  \HH_{0s} \rvv_{L^2_{\widetilde \omega}(\OO)}^2 +  \lvv \HH_{1s} \rvv_{L^2_{\widetilde \omega}(\OO)}^2   +   \lvv \HH_{bs} \rvv_{L^2(\Sigma_-; \d\xi^1_{\widetilde \omega})}^2  \right)  \d s,
$$
for every $t\in [0,T]$. Considering then $h_1, h_2 \in  C([0,T), L^2_{\widetilde \omega}(\OO))\cap \HHH_{\widetilde \omega}(\UU_T)$ two solutions of Equation \eqref{eq:linear_primal_inflow}, we take $\psi = h_1-h_2$ and we remark that $\psi$ solves Equation \eqref{eq:linear_primal_inflow} in the weak sense with $h(0) = \HH = \HH_b = 0$. We immediately deduce that $\psi\equiv 0$ from the above energy estimate, thus the uniqueness of the solution to Equation \eqref{eq:linear_primal_inflow}.

\medskip\noindent
\textit{Step 2.} We take $\delta\in (0,1)$, $h^0=0$ and we define the sequence $h^k$ recurrently as follows
\begin{equation}\label{eq:linear_primal_inflow_K}
\left\{\begin{array}{lcll}
 \partial_t h^{k+1} &=& - v \cdot \nabla_x h^{k+1} + \widetilde \Lambda \Delta_v h^{k+1} + v\cdot \grad_v h^{k+1}  +\HH^k  &\text{ in } \UU_T \\
 \gamma_{-} h^{k+1}  &= &\delta \, \RRR \gamma_+ h^k  &\text{ on }  \Gamma_{T, -} \\
 h_{t=0}^{k+1} &=& h_0   &\text{ in }   \OO,
\end{array}\right.
\end{equation}
where $\HH^k = dh^k+ \GGG h^k+ \EE_{h^k} \, \Delta_v \FFFF^\alpha$. We observe that the sequence $h^k$ is well defined by using an induction argument: given $h^{k}\in C([0,T], L^2_{\widetilde \omega}(\OO))\cap  \HHH_{\widetilde \omega}(\UU_T)$ with a trace $\gamma h^{k}\in L^2(\Gamma_T;  \d\xi^1_{\widetilde \omega} \, \d t)$
we have that $\HH^{k}\in L^2_{t,x}H^{1}_v + L^2_{t,x}H^{-1}_v$ and since $\RRR$ satisfies \eqref{eq:ConditionL2RRR}, as proved during the Step 1 of the proof of Lemma \ref{lem:GrowthLpPrimal}, we also have that
$\RRR \gamma_+ h^{k}\in L^2(\Gamma_{T, -}, \d\xi^1_{\widetilde \omega} \, \d t)$.
Applying then the results from Step~1 we obtain the existence of $h^{k+1} \in C([0,T), L^2_{\widetilde \omega}(\OO))\cap \HHH_{\widetilde \omega}$ with an associated trace function $\gamma h^{k+1} \in L^2(\Gamma_T; \d\xi^1_{\widetilde \omega}\, \d t)$, renormalized solution of Equation \eqref{eq:linear_primal_inflow_K}. 

\smallskip\noindent
We take then $\eps_2 = \min(\eps_1, \eps_2^1)$ where $\eps_1$ is given by Proposition \ref{prop:GrowthLpPrimal+NNNN}, and using the renormalized formulation \eqref{eq:FPK-traceL2} with the choices of $\beta(s)=s^2$ and $\varphi = \widetilde \omega^2\chi_R$ for any $R>0$ we can follow the computations performed during the proof of Proposition \ref{prop:GrowthLpPrimal+NNNN}, we pass to the limit as $R\to \infty$, and we get the energy estimate
\begin{multline}\label{eq:Energy_Estimate_K}
\lvv h^{k+1}_t \rvv_{L^2_{\widetilde \omega}(\OO)}^2 + \lvv h^{k+1} \rvv_{H^{1,\dagger}_\omega(\OO)}^2 + \lvv \gamma_+ h^{k+1} \rvv_{L^2(\Gamma_{t,+}; d\xi^1_{\widetilde \omega})}^2 \leq \lvv h^{k+1}_0 \rvv_{L^2_{\widetilde \omega}(\OO)}^2
\\ +\delta \lvv \gamma_+ h^k \rvv_{L^2(\Gamma_{t,+}; d\xi^1_{\widetilde \omega})}^2 + \lambda_0 \lvv h^{k+1}\rvv_{L^2_{\widetilde \omega}(\UU_t)}^2 + \lambda_1 \lvv h^{k}\rvv_{L^2_{\widetilde \omega}(\UU_t)}^2 ,
\end{multline}
for every $t\in [0,T]$ and some constants $\lambda_0, \lambda_1 >0$. 
Using then the integral version of the Grönwall lemma we obtain
\begin{multline*}
\| h^{k+1}_t \|^2_{L^2_{\tilde\omega}(\OO)} + \int_0^t \left( \| \gamma_+ h^{k+1}_s \|^2_{L^2(\Sigma_+; d\xi^1_{\widetilde \omega})} +   \| h^{k+1}_s  \|^2_{H^{1,\dagger}_{\tilde\omega}(\OO)} \right) \, e^{\lambda_0(t-s)} \, d s 
\\
\le
 \| h_0 \|^2_{L^2_{\tilde\omega}(\OO)} e^{\lambda_0 t} + \int_0^t  \left( \delta \| \gamma_+ h^k_s \|^2_{L^2_{\tilde\omega}(\Sigma_-; d\xi_1)}  +  \lambda_1 \lvv h^{k}_s \rvv_{L^2_{\widetilde \omega}(\OO)}^2  \right) \, e^{\lambda_0(t-s)} \, d s.
\end{multline*}
Choosing then $T>0$ small enough, the previous estimate implies that $h^k$ and $\gamma h^k$ are Cauchy sequences in the spaces $ C([0,T], L^2_{\widetilde \omega}(\OO))\cap  \HHH_{\widetilde \omega}(\UU_T)$ and $L^2 (\Gamma_T;  \d\xi^1_{\widetilde \omega} \, \d t)$ respectively.
Therefore we deduce that there are some functions $h \in C([0,T), L^2_{\widetilde \omega}(\OO))\cap  \HHH_{\widetilde \omega}$
and $\bar \gamma \in L^2 (\Gamma_T; \d\xi^1_{\widetilde \omega} \, \d t)$ such that $h^k \to h$ strongly in $ C([0,T), L^2_{\widetilde \omega}(\OO))\cap  \HHH_{\widetilde \omega}$ and $\gamma h^k \to \bar \gamma$ strongly in $L^2 (\Gamma_T; \d\xi^1_{\widetilde \omega}\, \d t)$
as $k\to \infty$. 

\smallskip\noindent
By taking then $\beta$ one-to-one, any $\varphi \in \DD'(\bar\UU_T)$, and using the Lebesgue dominated convergence theorem, we may pass to the limit in the weak formulation of Equation \eqref{eq:linear_primal_inflow_K} and we obtain that $h$ is a weak solution of the problem
 \begin{equation}\label{eq:linear_primal_inflow_infty}
\left\{\begin{array}{lcll}
 \partial_t h &=& \QQQ_g h 
 &\text{ in } \UU_T \\
 \gamma_{-} h  &= &\delta \, \RRR \gamma_+ h \quad &\text{ on }  \Gamma_{T, -} \\
 h_{t=0} &=& h_0  \quad &\text{ in }   \OO,
\end{array}\right.
\end{equation}
where we recall that $\QQQ_g$ is defined in \eqref{eq:DefQQQg}. Using again \cite[Proposition 2.8-(3)]{CM_Landau_domain} we have that $h$ is also a renormalized solution of Equation \eqref{eq:linear_primal_inflow_infty} with a trace function $\gamma h$ and we immediately deduce that $\gamma h = \bar \gamma$ a.e in $\Gamma_T$.

\smallskip\noindent
We now take again $\beta(s)=s^2$ and  $\varphi := (n_x \cdot v) \la v \ra^{-2} \omega^2(v)$ in the renormalize formulation of Equation \eqref{eq:linear_primal_inflow_infty} and following the same arguments leading to the energy estimate \eqref{eq:Energy_Estimate_K} we obtain that
\be\label{eq:ControlMaxwell}
\int_{\Gamma} (\gamma h)^2 \, \d\xi^2_{ \omega}  \, \d t 
\lesssim   \| h_0 \|^2_{L^2_{ \omega}(\OO)} e^{\lambda_2 T}. 
\ee
for some constant $\lambda_2 >0$. 

\smallskip\noindent
Moreover, using the renormalized formulation with the choices $\beta(s)=s^2$ and $\varphi = \widetilde \omega^2\chi_R$ for any $R>0$, arguing as during the proof of Proposition \ref{prop:GrowthLpPrimal+NNNN}, and taking the limit as $R\to \infty$, we obtain that there is $\kappa >0$ for which there holds
\bear\label{eq:linear_gak_bdd_BEFORE}
&&\quad \| h_{t} \|^2_{L^2_{\widetilde\omega}(\OO)} + \int_0^t \left\{ (1-\delta) \| \gamma h_{s} \|^2_{L^2(\Sigma_{+}; \d\xi^1_{\widetilde\omega})}
+  
 \| h_{s} \|^2_{H^{1,\dagger}_{\widetilde\omega}(\OO)} 
\right\} \, e^{\kappa (t-s)} \, \d s 
\le e^{\kappa t} 
 \| h_0 \|^2_{L^2_{\widetilde\omega}(\OO)}  , 
\eear
for every $t\in [0, T]$. 
We remark then that the linearity of Equation \eqref{eq:linear_primal_inflow_infty} together with the estimate \eqref{eq:linear_gak_bdd_BEFORE} give the uniqueness of the solution.

\medskip\noindent
\textit{Step 3.} For a sequence $\delta_k \in (0,1)$, $\delta_k \nearrow 1$, we consider $h_k \in C([0,T), L^2_{\widetilde \omega}(\OO))\cap  \HHH_{\widetilde \omega}$ as the renormalized solutions to the modified Maxwell reflection boundary condition problem
\begin{equation}\label{eq:linear_gak}
\left\{\begin{array}{lcll}
 \partial_t h_k &=& \QQQ_g h_k  &\text{ in }  \UU_T \\
 \gamma_{-} h_k   &=&  \delta_k \RRR \gamma_{+} h_k   &\text{ on } \Sigma_{T, -} \\
 h_{k, t=0} &=& h_0   &\text{ in }   \OO,
\end{array}\right.
\end{equation}
obtained by using Step~2.
By following the arguments leading to the estimate \eqref{eq:linear_gak_bdd_BEFORE}, we have that there is $\kappa >0$ such that $h_k$ satisfies 
\bear\label{eq:linear_gak_bdd}
&&\quad \| h_{kt} \|^2_{L^2_{\widetilde\omega}(\OO)} + \int_0^t \left\{ (1-\delta_k) \| \gamma h_{ks} \|^2_{L^2(\Sigma_{+};\d\xi^1_{\widetilde\omega})}
+  
 \| h_{ks} \|^2_{H^{1,\dagger}_{\widetilde\omega}(\OO)} 
\right\} \, e^{\kappa (t-s)} \, \d s 
\le e^{\kappa t} 
 \| h_0 \|^2_{L^2_{\widetilde\omega}(\OO)} , 
\eear
for any $t \in [0,T]$ and any $k \ge 1$. Moreover, for any constants $T, R>0$ we may use the renormalized formulation as during the proof of Proposition \ref{prop:GrowthLpPrimal+NNNN} and we obtain that
\beqn
\int_0^t \lvv \grad_v h_{ks} \rvv_{L^2_\omega (\OO)} \d s \lesssim e^{\kappa t}  \lvv h_0\rvv_{L^2_\omega(\OO)},
\eeqn
for any $0 < t\leq  T$. Additionally, from \eqref{eq:ControlMaxwell}, we also have the following energy estimate on the boundary
\beqn
\int_{\Gamma_T} (\gamma h_k)^2 \, \d\xi^2_{\widetilde \omega}  \, \d t 
\lesssim   \| h_0 \|^2_{L^2_{\widetilde \omega}(\OO)} e^{\lambda_2 T}. 
\eeqn
From the above estimates, we deduce that, up to the extraction of a subsequence, there exist $h \in \HHH_{\widetilde \omega} (\UU_T)  \cap L^2_{\widetilde \omega}(\UU_T)) \cap L^2([0,T]\times \Omega; H^1_\omega(\R^d))$ and $\nu_\pm \in L^2(\Gamma_{T, \pm}; \d\xi^2_{\widetilde \omega} \d t)$ such that 
\be\label{eq:Weak_Cvge_L2}
h_k \wto h \hbox{ weakly in } \ \HHH_{\widetilde \omega}(\UU_T) \cap L^2_{\widetilde \omega}(\UU_T) \cap L^2([0,T]\times \Omega; H^1_\omega(\R^d)), 
\ee
and
$$
\gamma_\pm h_k \wto \nu_\pm \hbox{ weakly in } \ L^2(\Gamma_{T, \pm}; \d\xi^2_{\widetilde \omega} \d t). 
$$
We now establish more properties regarding the convergence of the previous subsequences to their limit in Steps 3.1 and 3.2 and we will conclude the existence of weak solutions in Step 3.4.

\medskip\noindent
\textit{Step 3.1. (Strong convergence in $L^2(\UU_T)$)} Since $h_{k}\in  L^2([0,T]\times \Omega; H^1_\omega(\R^d))$ and satisfies Equation \eqref{eq:linear_gak} in the weak sense, for any truncated (in $t$ and $x$) version $\left(\bar h_k\right)$ of $\left(h_k\right)$ we may apply \cite[Theorem~1.3]{MR1949176} and we deduce that $\left(\bar h_k\right)$ is bounded in $H^{1/4}(\R_t \times \R^d _x\times \R^d_v)$. Using now the version of the Rellich-Kondrachov theorem for fractional Sobolev spaces (see for instance \cite[Corollary 7.2]{MR2944369} or \cite[Lemma 6.11]{MR3081641}) we deduce that 
\be\label{eq:Stron_Local_Cvge_L2}
h_k \to h \hbox{ strongly in } L^2([0,T]\times \OO_R), 
\ee
for any $R>0$ and where we have defined $\OO_R := \{(x,v)\in \OO; \, d(x,\Omega^c)>1/R, \, \lv v\rv <R \}$. 

\smallskip\noindent
Let now $\varsigma = \omega \la v\ra^{-1/2}$, we have that the sequence $h_{k}$ is tight in $L^2_\varsigma(\UU_T)$. Indeed we observe that
\be\label{eq:Tightness_hk_L2}
\lvv h_{kt} \rvv_{L^2_{\varsigma}(\OO_R^c)} \leq \frac 1{\la R\ra^{1/2}}  \lvv h_{kt} \rvv_{L^2_{\omega}(\OO_R^c)}   \leq \frac 1{\la R\ra^{1/2}} \lvv h_{kt} \rvv_{L^2_\omega(\OO)} \lesssim \frac 1{\la R\ra^{1/2}} e^{\kappa t} \lvv h_0\rvv_{L^2_\omega(\OO)},
\ee
for every $t\in[0,T]$, and where we remark that we have used \eqref{eq:linear_gak_bdd} together with \eqref{eq:omega&omegatilde} to obtain the last inequality. Combining \eqref{eq:Stron_Local_Cvge_L2} and \eqref{eq:Tightness_hk_L2} we classically obtain that $h_k \to h$ strongly in $L^2_{\varsigma}(\UU_T)$ as $k\to \infty$. Using this convergence, the Cauchy-Schwartz inequality, and the fact that $\langle v \rangle^{2+1/2}  \omega^{-1} \in L^2(\R^d)$, we deduce then that
\be\label{eq:Convergence_Energy_L2}
\EE_{h_k} \to \EE_h \qquad \text { as } k\to \infty.
\ee

\noindent
\textit{Step 3.2. (Pointwise convergence for the boundary term)}
Because $\langle v \rangle  \omega^{-1} \in L^2(\R^d)$, 
we have that $L^2(\Gamma_T;  \d\xi^2_{\widetilde \omega} \, \d t) \subset L^1(\Gamma_T; \d\xi_1^1 \, \d t)$ and recalling that, from the very definition of the Maxwell reflection operator given by \eqref{eq:BoundaryConditions}, there holds
\beqn\label{eq:KolmogorovWP-hypL1}
\RRR : L^1(\Sigma_+;d\xi^1_1) \to L^1(\Sigma_-;d\xi^1_1), \quad   \| \RRR \|_{L^1(\Sigma;  d\xi^1_1 )} \le 1, 
\eeqn
we deduce that $\RRR(\gamma h_{k+})  \wto \RRR(\mathfrak{h}_+)$   weakly in $L^1(\Gamma_{T,-}; \d\xi_1^1\, \d t)$. 
On the other hand, from \cite[Theorem 3.2]{CGMM24}, we have $\gamma h_k \wto \gamma h$ weakly in $\Lloc^2 (\Gamma_T; \d\xi^2_1\, \d t)$. 
Using both convergences in the boundary condition $\gamma_- h_k = \RRR(\gamma_+ h_k)$, we obtain that $\gamma_- h = \RRR(\gamma_+ h)$ a.e.

\medskip\noindent
\textit{Step 3.3. (Conclusion of the existence of weak solutions)}
Using \eqref{eq:Weak_Cvge_L2} and \eqref{eq:Convergence_Energy_L2} we may take the limit in the weak formulation of Equation \eqref{eq:linear_gak} and we obtain that $h$ is a weak solution of Equation \eqref{eq:linear_gak} complemented with the Maxwell reflection boundary condition and associated to the initial datum $h_0$, i.e there holds \eqref{eq:Perturbed_renormalized_formulation}.

\smallskip\noindent
Moreover, \cite[Proposition 2.8-(2)]{CM_Landau_domain} implies that $h \in C([0,T]; L^2_{\widetilde \omega}(\OO)) \cap \HHH_{\widetilde \omega}(\UU_T)$ and that $h$ is also a renormalized solution of Equation \eqref{eq:linear_gak}. Repeating the arguments leading to \eqref{eq:linear_gak_bdd} we deduce the energy estimate
\be\label{eq:bddL2&HHH}
 \| h_{t} \|^2_{L^2_{\widetilde\omega}(\OO)} + 2 \int_0^t  \| h_{ks} \|^2_{H^{1,\dagger}_{\widetilde\omega}(\OO)}  \, e^{\kappa(t-s)} \, \d s 
\le e^{\kappa t}
 \| h_0 \|^2_{L^2_{\widetilde\omega}(\OO)}   \qquad \forall t\in [0, T].
\ee

\medskip\noindent
\textit{Step 4.} We consider now two solutions $h_1$ and $h_2 \in  C([0,T];L^2_{\widetilde \omega}(\OO)) \cap \HHH_{\widetilde \omega}(\UU_T)$ to Equation \eqref{eq:KFPstabilityLin} associated to the same initial datum $h_0$. We now define $\psi := h_2 - h_1 $ and we note that $ \psi \in  C([0,T]; L^2_{\widetilde \omega}(\OO)) \cap \HHH_{\widetilde \omega}(\UU_T)$  is  a weak solution to Equation \eqref{eq:KFPstabilityLin} associated to the initial datum $\psi(0) = 0$. We immediately get from \eqref{eq:bddL2&HHH} that $\psi \equiv 0$ a.e. 

\smallskip\noindent
Finally, by repeating this result in every time interval $[jT, (j+1)T]$ for $j\in \N$, we obtain the existence and uniqueness of a global weak solution to Equation \eqref{eq:KFPstabilityLin}
\end{proof}

\section{Hypodissipativity} \label{sec:HypoDissipativity}

We start this section by taking a glance at the equation
\be\label{eq:KFP_PPP}
\left\{\begin{array}{rcll} 
 \partial_t h &=& \PPP h   & \text{ in } \UU  \\
\displaystyle \gamma_- h &=& \RRR \gamma_+ h &\text{ on } \Gamma_{-} \\
\displaystyle  h_{t=0} &=& h_0 &\text{ in } \OO,
 \end{array}\right.
 \ee
 where we recall that $h_0$ is given by \eqref{eq:PerturbedInitialDatum}, and $\PPP$ is defined in \eqref{eq:DefPPPg}.
 We have then the following lemma.
 
 \begin{lem}\label{lem:Hypo_PPP}
Consider $\omega$ and admissible weight function. The following statements hold.

\smallskip\noindent
 {\sl (1)} Let $h_0 \in L^2_\omega(\OO)$, there is $h \in L^2_\omega(\UU) \cap \HHH_\omega(\UU)$  unique global weak solution to Equation \eqref{eq:KFP_PPP}. 
 
 \smallskip\noindent
In particular, there is a strongly continuous continuous semigroup $S_\PPP:L^2_\omega(\OO) \to L^2_\omega(\OO)$ associated to the solutions of Equation \eqref{eq:KFP_PPP}. Moreover there are constants $\lambda>0$ and $C_P>0$ such that there holds
\be\label{eq:DecayPPP}
\lvv  S_\PPP(t) h_0  \rvv_{L^2_\omega(\OO)} \leq C_P \, e^{-\lambda t } \lvv h_0\rvv_{L^2_\omega(\OO)} ,
\ee
and
\be\label{eq:bddGradPPP}
 \int_0^t \| \grad_v S_\PPP(s) h_0 \|^2_{L^2_{\omega}(\OO)} ds \lesssim_{C_P}  \| h_0 \|^2_{L^2_{\omega}(\OO)} + 
 \int_0^t \| S_\PPP(s) h_0 \|^2_{L^2_{\omega}(\OO)} ds  ,
\ee
for all $t > 0$.

\smallskip\noindent
 {\sl (2)} There is a function $\HHHH  \in L^2(\Omega, H^1(\R^d))$ unique steady solution of Equation \eqref{eq:KFP_PPP}. Moreover for every admissible weight function $\varsigma$ there is a constant $C_\HHHH>0$ such that 
\be\label{eq:BdHHHH}
\lvv \HHHH \rvv_{L^2_\varsigma(\OO)} \leq C_\HHHH. 
\ee

\smallskip\noindent
 {\sl (3)} For any final time $T>0$ and any final datum $\psi_T\in L^2_m(\OO)$, with $m=\omega^{-1}$, there is $\psi\in L^2_m(\UU_T)\cap \HHH_m (\UU_T)$ unique global weak solution of the dual backwards in time equation
\be\label{eq:KFP_PPP_dual}
\left\{\begin{array}{rcll} 
 \partial_t \psi &=& \PPP^* \psi := v\cdot \grad_x \psi + \Lambda^\star \Delta_v \psi - v\cdot \grad_v \psi + \GGG^* \psi   & \text{ in } \UU_T  \\
\displaystyle \gamma_+ \psi &=& \RRR^* \gamma_- \psi &\text{ on } \Gamma_{T, +} \\
\displaystyle  \psi_{t=T} &=& \psi_T &\text{ in } \OO,
 \end{array}\right.
 \ee 
 where we recall that $\GGG^*$ is defined in \eqref{eq:Def_GGG_Dual} and $\RRR^*$ is defined \eqref{eq:DefDualMaxwellBC}. In particular, there is a strongly continuous semigroup $S_{\PPP^*} : L^2_m (\OO) \to L^2_m(\OO)$ which is dual to the semigroup $S_\PPP$, i.e there holds \eqref{eq:identite-dualite}. Moreover, there is a constant $C_P^*>0$ such that
\be\label{eq:DecayPPP*}
\lvv  S_{\PPP^*} (0) \, \psi_T  -\la \psi_T, \HHHH\ra \rvv_{L^2_m(\OO)}  \leq C_P^* \, e^{-\lambda T } \lvv \psi_T  -\la \psi_T, \HHHH\ra  \rvv_{L^2_m(\OO)} ,
\ee
and
\be\label{eq:bddGradPPP*}
 \int_0^T \| \grad_v S_{\PPP^*}(0) \psi_T \|^2_{L^2_{m}(\OO)} ds \lesssim_{C_P^*}  \| \psi_T \|^2_{L^2_{m}(\OO)} + 
 \int_0^T \| S_{\PPP^*} (0) \psi_T  \|^2_{L^2_{m}(\OO)} ds .
\ee
 \end{lem}

\begin{proof} 
We first remark that Equation \eqref{eq:KFP_PPP} fits the framework of Equation \eqref{eq:KFPtau} with $\Lambda = \Lambda^\star$. Therefore {\sl(1)} and {\sl(2)}  are a consequence of Lemma \ref{lem:GrowthLpPrimal}, Theorem \ref{theo:WellPosednessL2}, and Theorem \ref{theo:KRDoblinHarris}. Finally, {\sl(3)} is a consequence of Lemma \ref{lem:GrowthLpDual}, Proposition \ref{prop:WellPosednessL2*}, Proposition \ref{prop:DualityRigorous}, and Proposition \ref{prop:SteadySolutionLDual}. 
\end{proof}

\smallskip\noindent
We now aim to prove that, under suitable assumptions on the function $g$ and on the parameter $\alpha$, the decay properties of the semigroup $S_\PPP$ can be extended to $S_{\QQQ_g}$ by treating the term $\NNNN_g$ as a perturbation. The core idea of the proof is to define a new norm, in the spirit of \cite[Proposition 3.6]{CM_Landau},  \cite[Proposition 3.2]{Carrapatoso_2016} and \cite[Proposition 4.1]{MR3591133}, that leverages both the dissipativity of $\PPP$ and the control on $\QQQ_g$ established in Proposition \ref{prop:GrowthLpPrimal+NNNN}.

\begin{prop}\label{prop:Hypo}
Let $\omega$ be an admissible weight function. There are constants $\alpha^{\star\star} \in (0,\alpha^\star)$, $\eta >0$ and $\eps_3>0$ such that for every $\alpha \in (0,\alpha^{\star\star})$, if $\lvv g_t \rvv_{L^\infty_\omega(\OO)} \leq \eps$ holds for any $\eps\in (0,\eps_3)$, then if $h_0 \in L^2_\omega(\OO)$, every solution $h$ of Equation \eqref{eq:KFPstabilityLin} satisfies
\be\label{eq:Hypo}
\lvv h_t\rvv_{L^2_\omega(\OO)} \leq Ce^{-\eta t} \lvv h_0\rvv_{L^2_\omega(\OO)} \qquad \forall t\geq 0,
\ee
for some constant $C\geq 1$.
\end{prop}

\begin{proof}
We introduce $\beta>0$ to be fixed later, and we define the norm
$$
\lvvv h_t \rvvv^2 = \beta \lvv h_t \rvv_{L^2_{\widetilde \omega}(\OO)}^2 + \int_0^\infty \lvv S_ \PPP(\tau) h_t \rvv_{L^2_{\omega}(\OO)}^2\d\tau, 
$$
where we recall that $S_\PPP$ is given by Lemma \ref{lem:Hypo_PPP}-(1) and $\widetilde \omega$ is defined in \eqref{def:omegaA}.
We first observe that the norm $\lvvv \cdot \rvvv$ is well defined due to \eqref{eq:DecayPPP}, and we also have that
\be\label{eq:HypoNormEquiv}
\beta c_A^{-1} \lvv h_t \rvv_{L^2_{\omega}(\OO)}^2 \leq \beta \lvv h_t \rvv_{L^2_{\widetilde \omega}(\OO)}^2 \leq \lvvv h_t \rvvv^2  \leq \left( \beta c_A + {(C_P)^2 \over 2\lambda} \right) \lvv h_t \rvv_{L^2_{ \omega}(\OO)}^2 , 
\ee
where we have used \eqref{eq:omega&omegatilde} to obtain the first inequality, and \eqref{eq:DecayPPP} together with  \eqref{eq:omega&omegatilde} again to obtain the last inequality. 
\\

We take $\eps_3^1 = \min(\eps_1, \eps_2)$, where $\eps_1>0$ is given by Proposition \ref{prop:GrowthLpPrimal+NNNN} and $\eps_2>0$ is given by Proposition \ref{prop:KFPL2Perturbed}, and we assume during the sequel that $\lvv g_t \rvv_{L^2_\omega(\OO)} \leq \eps$ for $\eps \in (0, \eps_3^1)$. We recall that $\QQQ_g$ is defined in \eqref{eq:DefQQQg}, and arguing as during the proof of Proposition \ref{prop:GrowthLpPrimal+NNNN} (see also the arguments leading to Proposition \ref{prop:KFPL2Perturbed}) we have that there is $\kappa>0$ such that
\bean
\frac d{dt} \lvvv h_t \rvvv^2 &=& \beta \la h_t, \QQQ_g \, h_t\ra_{L^2_{\widetilde\omega}(\OO)} + \frac d{dt} \int_0^\infty \lvv S_ \PPP(\tau) h_t \rvv_{L^2_{\omega}(\OO)}^2\d\tau  \\
&\leq&  -\beta \frac{\tau_0}2 \lvv \grad_v (h_t \widetilde \omega)\rvv_{L^2(\OO)}^2  + \beta \kappa \lvv h_t\rvv_{L^2_{\widetilde \omega}(\OO)}^2 +  {d\over dt}  \int_0^\infty \lvv S_ \PPP(\tau) h_t \rvv_{L^2_\omega(\OO)}^2\d\tau\\
&\leq & -\beta \frac{\tau_0}2 \lvv \grad_v h_t \rvv_{L^2_{\widetilde \omega}(\OO)}^2  + \beta \left( \kappa + C^0_\omega \frac{\tau_0}2 \right) \lvv h_t\rvv_{L^2_{\widetilde \omega}(\OO)}^2 +  {d\over dt}  \int_0^\infty \lvv S_ \PPP(\tau) h_t \rvv_{L^2_\omega(\OO)}^2\d\tau\\
&\leq &  -\beta c_A^{-1} \frac{\tau_0}2 \lvv \grad_v h_t \rvv_{L^2_{ \omega}(\OO)}^2  + \beta c_A \left( \kappa + C^0_\omega \frac{\tau_0}2 \right) \lvv h_t\rvv_{L^2_{ \omega}(\OO)}^2 +  {d\over dt}  \int_0^\infty \lvv S_ \PPP(\tau) h_t \rvv_{L^2_\omega(\OO)}^2\d\tau
\eean
where we have have used \eqref{eq:omega&omegatilde} to obtain the last line and we have defined the positive constant $C_\omega^0 = \lvv (\grad_v\widetilde \omega)/\widetilde \omega\rvv_{L^\infty}<\infty$, due to \eqref{eq:ControlGradTilde}. We then divide the rest of the proof into four steps. 

\medskip\noindent
\emph{Step 1.} We observe that Lemma \ref{lem:Hypo_PPP}-(1) and the Lebesgue dominated convergence theorem imply that
$$
{d\over dt}  \int_0^\infty \lvv S_ \PPP(\tau) h_t \rvv_{L^2_{\omega}(\OO)}^2\d\tau =  \int_0^\infty  {d\over dt}  \lvv S_ \PPP(\tau) h_t \rvv_{L^2_{\omega}(\OO)}^2\d\tau.
$$
Using Remark \ref{rem:Def_S_QQQ_g} we notice that we may write $h(t) = S_{\QQQ_g} h_0$, where $S_{\QQQ_g}$ is the strongly continuous semigroup given by Proposition \ref{prop:KFPL2Perturbed} and Remark \ref{rem:Def_S_QQQ_g}. Moreover from Lemma \ref{lem:Hypo_PPP}-(1) we also know that $S_\PPP$ is a strongly continuous semigroup. Thus, using \cite[Chapter 1, Corollary 1.4-(d), Corollary 2.3, and Theorem 2.4]{MR710486}, we may perform the following computations 
\begin{multline}\label{eq:DerivSemigroupPPP}
{d\over dt} S_ \PPP (\tau) h_t = S_ \PPP (\tau) {d\over dt} S_{\QQQ_g}(t) h_0 = S_ \PPP (\tau)\left( ( \PPP + \NNNN_g) h_t \right) \\ 
= S_ \PPP (\tau)  \PPP h_t + S_ \PPP(\tau) \NNNN_g h_t = {d\over\d\tau} S_ \PPP(\tau) h_t + S_ \PPP(\tau) \NNNN_g h_t,
\end{multline}
where we have used repeatedly the results from \cite{MR710486} to obtain the previous chain of equalities. 
Using \eqref{eq:DerivSemigroupPPP} we deduce then that 
\bean
{d\over dt} \lvv S_ \PPP(\tau) h_t \rvv_{L^2_{\omega}(\OO)}^2  &=& 2 \left\la  {d\over dt}  S_ \PPP (\tau) h_t ,\,  S_ \PPP(\tau) h_t \right\ra_{L^2_\omega(\OO)}\\
&=& 2 \left\la   {d\over\d\tau} S_ \PPP(\tau) h_t + S_ \PPP(\tau) \NNNN_g h_t ,\,  S_ \PPP(\tau) h_t \right\ra_{L^2_\omega(\OO)}\\
&=& {d\over\d\tau} \lvv  S_ \PPP(\tau) h_t \rvv_{L^2_\omega(\OO)}^2  + 2\left\la  S_{ \PPP} (\tau)  \NNNN_g  h_t , \,  S_ \PPP(\tau) h_t  \right\ra_{L^2_\omega(\OO)},
\eean 
and we now proceed to compute each term separately.
On the one hand, due to the fundamental theorem of calculus, \cite[Corollary 1.4-(d)]{MR710486}, and Lemma \ref{lem:Hypo_PPP}-(1), we obtain that
$$
\int_0^\infty \frac d{d\tau} \lvv  S_ \PPP(\tau) h_t \rvv_{L^2_\omega(\OO)}^2\d\tau  = -  \lvv  h_t \rvv_{L^2_\omega(\OO)}^2.
$$
On the other hand, Lemma \ref{lem:Hypo_PPP}-(3) implies the existence of $S_{\PPP^*}$, dual semigroup to $S_\PPP$, thus we may compute
\bean
\int_0^\infty \left\la  S_{ \PPP} (\tau)  \NNNN_g  h_t , \,  S_ \PPP(\tau) h_t  \right\ra_{L^2_\omega(\OO)} \d \tau= \int_0^\infty \left\la \NNNN_g  h_t , \,  S_{ \PPP^*} (0) \left[ (S_ \PPP(\tau) h_t) \omega^2 \right] \right\ra_{L^2(\OO)} \d\tau  =: P_1 + P_2 ,
\eean
and we control now $P_1$ and $P_2$ separately.

\medskip\noindent
\emph{Step 2. (Control of $P_2$)} We set $m=\omega^{-1}$ and we have
\bean
P_2 &:=& \int_0^\infty \left\la  \alpha \EE_{h_t} \Delta_v \FFFF^\alpha , \,  S_{ \PPP^*} (0) \left[ (S_ \PPP(\tau) h_t) \omega^2 \right]  \right\ra_{L^2(\OO)}\d\tau \\
&=& -\int_0^\infty \alpha  \EE_{h_t} \la \grad_v \FFFF^\alpha , \grad_v \left(S_{ \PPP^*} (0) \left[ (S_ \PPP(\tau) h_t) \omega^2\right] \right) \ra _{L^2(\OO)}\d\tau \\
&\leq & \alpha C_\omega^1 \int_0^\infty \lvv h_t \rvv_{L^2_\omega(\OO)}  \lvv \grad_v \FFFF^\alpha\rvv_{L^2_\omega(\OO)} \left\lvv \grad_v \left(S_{ \PPP^*}(0) \left[  (S_ \PPP(\tau) h_t) \omega^2 \right] \right) \right\rvv _{L^2_m(\OO)}\d\tau ,
\eean 
where we remark that we have used integration by parts to obtain the second line and the Cauchy-Schwartz inequality to obtain the third one. Moreover we have defined $C_\omega^1 = \lvv \la v\ra^2 \omega^{-1}\rvv_{L^2}$ and we deduce that $C_\omega^1<\infty$ due to the very definition of $\omega$.

We now remark that $(S_ \PPP(\tau) h_t) \omega^2 \in L^2_m(\OO)$, indeed there holds
$$
\left( \int_\OO (S_ \PPP(\tau) h_t)^2 \omega^4 m^2 \right)^{1/2}= \lvv S_\PPP(\tau) h_t \rvv_{L^2_\omega(\OO)} \lesssim e^{-\lambda \tau + \kappa t} \lvv h_0 \rvv_{L^2_\omega(\OO)},
$$
where we have successively used the fact that $m=\omega^{-1}$, \eqref{eq:DecayPPP} and Proposition \ref{prop:GrowthLpPrimal+NNNN}. Coming back then to $P_2$ we further compute
\bean
P_2 &\leq & \alpha C_\omega^1 \lvv h_t \rvv_{L^2_\omega(\OO)}   \lvv \grad_v \FFFF^\alpha\rvv_{L^2_\omega(\OO)} \underset{A\to \infty}{\lim} \int_0^A  \left\lvv \grad_v \left(S_{ \PPP^*}(0) \left[ \left(S_ \PPP(\tau) h_t\right) \omega^2\right] \right) \right\rvv _{L^2_m(\OO)}\d\tau  \\
&\leq & \alpha C_\omega^1 \lvv h_t \rvv_{L^2_\omega(\OO)}   \lvv \grad_v \FFFF^\alpha\rvv_{L^2_\omega(\OO)} \underset{A\to \infty}{\lim}  \left\lvv (S_ \PPP(A) h_t) \omega^2 \right\rvv _{L^2_m(\OO)}  \\
&& + \alpha C_\omega^1 \lvv h_t \rvv_{L^2_\omega(\OO)}   \lvv \grad_v \FFFF^\alpha\rvv_{L^2_\omega(\OO)} \underset{A\to \infty}{\lim}  \int_0^A  \left\lvv S_{ \PPP^*}(0) \left[ \left( S_ \PPP(\tau) h_t \right) \omega^2 \right] \right\rvv _{L^2_m(\OO)} \d\tau  \\
&\leq & \alpha C_\omega^1 \lvv h_t \rvv_{L^2_\omega(\OO)}   \lvv \grad_v \FFFF^\alpha\rvv_{L^2_\omega(\OO)}  \int_0^\infty  \left\lvv S_{ \PPP^*}(0) \left[ \left( S_ \PPP(\tau) h_t \right) \omega^2 \right] -\la S_ \PPP(\tau) h_t , \HHHH\ra_{L^2_\omega(\OO)} \right\rvv _{L^2_m(\OO)}  \d\tau \\
&&+  \alpha C_\omega^1 \lvv h_t \rvv_{L^2_\omega(\OO)}   \lvv \grad_v \FFFF^\alpha\rvv_{L^2_\omega(\OO)}  \int_0^\infty   \la S_ \PPP(\tau) h_t , \HHHH\ra_{L^2_\omega(\OO)} \left\lvv m \right\rvv _{L^2(\OO)}  \d \tau\\
&=:& P_2^1 + P_2^2
\eean
where we have used \eqref{eq:bddGradPPP*} to obtain the second inequality and the triangular inequality to obtain the last inequality. 
On the one hand we have
\bean
P_2^1&:=&  \alpha C_\omega^1 \lvv h_t \rvv_{L^2_\omega(\OO)}   \lvv \grad_v \FFFF^\alpha\rvv_{L^2_\omega(\OO)}  \int_0^\infty  \left\lvv S_{ \PPP^*}(0) \left[ \left( S_ \PPP(\tau) h_t \right) \omega^2 \right] -\la S_ \PPP(\tau) h_t , \HHHH\ra_{L^2_\omega(\OO)} \right\rvv _{L^2_m(\OO)}  \d\tau \\
&\leq & \alpha C_\omega^1 \lvv h_t \rvv_{L^2_\omega(\OO)}   \lvv \grad_v \FFFF^\alpha\rvv_{L^2_\omega(\OO)}  C_P^* \int_0^\infty   e^{- \lambda \tau} \left\lvv \left( S_ \PPP(\tau) h_t \right) \omega^2  -\la S_ \PPP(\tau) h_t , \HHHH\ra_{L^2_\omega(\OO)} \right\rvv _{L^2_m(\OO)}  \d\tau \\
&\leq & \alpha C_\omega^1 \lvv h_t \rvv_{L^2_\omega(\OO)}   \lvv \grad_v \FFFF^\alpha\rvv_{L^2_\omega(\OO)} C_P^* \int_0^\infty   e^{- \lambda \tau} \left\lvv S_ \PPP(\tau) h_t \right\rvv_{L^2_\omega(\OO)}  \d \tau \\
& & +  \alpha C_\omega^1 \lvv h_t \rvv_{L^2_\omega(\OO)}   \lvv \grad_v \FFFF^\alpha\rvv_{L^2_\omega(\OO)} C_P^* \int_0^\infty \la S_ \PPP(\tau) h_t , \HHHH\ra_{L^2_\omega(\OO)} \left\lvv  m \right\rvv _{L^2(\OO)}  \d\tau 
\eean
where we have used \eqref{eq:DecayPPP*} to obtain the first inequality, the triangular inequality to obtain the second inequality. Using then \eqref{eq:DecayPPP} and the Cauchy-Schwartz inequality we further have that
\bean
P_2^1 &\leq & \alpha C_\omega^1 \lvv h_t \rvv_{L^2_\omega(\OO)}   \lvv \grad_v \FFFF^\alpha\rvv_{L^2_\omega(\OO)} C_P^* C_P \int_0^\infty   e^{- 2\lambda \tau} \left\lvv h_t \right\rvv_{L^2_\omega(\OO)}  \d\tau \\
& & +  \alpha C_\omega^1 \lvv h_t \rvv_{L^2_\omega(\OO)}   \lvv \grad_v \FFFF^\alpha\rvv_{L^2_\omega(\OO)} C_P^*  \int_0^\infty \left\lvv S_ \PPP(\tau) h_t  \right\rvv_{L^2_\omega(\OO)}  \left\lvv \HHHH\right\rvv_{L^2_\omega(\OO)} \left\lvv  m \right\rvv _{L^2(\OO)} \d\tau \\
&\leq & \alpha \frac{C_\omega^1}{2\lambda} \lvv h_t \rvv_{L^2_\omega(\OO)} ^2  \lvv \grad_v \FFFF^\alpha\rvv_{L^2_\omega(\OO)} C_P^* C_P \\
& & +  \alpha C_\omega^1 \lvv h_t \rvv_{L^2_\omega(\OO)}   \lvv \grad_v \FFFF^\alpha\rvv_{L^2_\omega(\OO)} C_P^* \left\lvv \HHHH\right\rvv_{L^2_\omega(\OO)} \left\lvv  m \right\rvv _{L^2(\OO)} C_P \int_0^\infty e^{-\lambda \tau} \left\lvv  h_t  \right\rvv_{L^2_\omega(\OO)}   \d\tau \\
&\leq & \alpha C_\omega^1 \lvv h_t \rvv_{L^2_\omega(\OO)}^2 \lvv \grad_v \FFFF^\alpha\rvv_{L^2_\omega(\OO)}  C_P^* C_P \left( \frac 1{2\lambda} + \frac 1\lambda \left\lvv \HHHH\right\rvv_{L^2_\omega(\OO)} \left\lvv  m \right\rvv _{L^2(\OO)}  \right)
\eean
 where we have used again \eqref{eq:DecayPPP} to obtain the last inequality.

\smallskip\noindent
On the other hand we compute 
\bean
P_2^2& :=& \alpha C_\omega^1 \lvv h_t \rvv_{L^2_\omega(\OO)}   \lvv \grad_v \FFFF^\alpha\rvv_{L^2_\omega(\OO)}  \int_0^\infty   \la S_ \PPP(\tau) h_t , \HHHH\ra_{L^2_\omega(\OO)} \left\lvv m \right\rvv _{L^2(\OO)} \d\tau\\
&\leq &   \alpha C_\omega^1 \lvv h_t \rvv_{L^2_\omega(\OO)}   \lvv \grad_v \FFFF^\alpha\rvv_{L^2_\omega(\OO)}  \int_0^\infty \left\lvv S_ \PPP(\tau) h_t  \right\rvv_{L^2_\omega(\OO)}  \left\lvv \HHHH\right\rvv_{L^2_\omega(\OO)} \left\lvv  m \right\rvv _{L^2(\OO)} \d\tau\\
&\leq & \alpha C_\omega^1 \lvv h_t \rvv_{L^2_\omega(\OO)}   \lvv \grad_v \FFFF^\alpha\rvv_{L^2_\omega(\OO)}  \left\lvv \HHHH\right\rvv_{L^2_\omega(\OO)} \left\lvv  m \right\rvv _{L^2(\OO)} C_P  \int_0^\infty e^{-\lambda \tau} \left\lvv  h_t  \right\rvv_{L^2_\omega(\OO)}   \d\tau \\
&\leq & \alpha C_\omega^1 \lvv h_t \rvv_{L^2_\omega(\OO)}^2 \lvv \grad_v \FFFF^\alpha\rvv_{L^2_\omega(\OO)}  
 \frac 1\lambda \left\lvv \HHHH\right\rvv_{L^2_\omega(\OO)} \left\lvv  m \right\rvv _{L^2(\OO)}  C_P
\eean
where we have used successively the Cauchy-Schwartz inequality to obtain the second line and \eqref{eq:DecayPPP} to obtain the third line. 

\smallskip\noindent
Putting together the previous computations we observe that we have obtained  
$$
P_2 \leq \alpha C_\omega^1 \lvv h_t \rvv_{L^2_\omega(\OO)}^2 \lvv \grad_v \FFFF^\alpha\rvv_{L^2_\omega(\OO)}  C_P  \left( \frac {C_P^*}{2\lambda} + \frac { C_P^* + 1 }\lambda \left\lvv \HHHH\right\rvv_{L^2_\omega(\OO)} \left\lvv  m \right\rvv _{L^2(\OO)}  \right),
$$
and this concludes this step.

\medskip\noindent
\emph{Step 3. (Control of $P_1$)} We compute now for the term $P_1$ as follows
\bean 
P_1&:=&  \int_0^\infty \alpha \EE_{g_t} \la \Delta_v h_t , S_{ \PPP^*}(0) \left[ S_ \PPP(\tau) h_t \omega^2\right] \ra _{L^2(\OO)}d\tau \\
&=& -\int_0^\infty \alpha  \EE_{g_t} \left\la \grad_v h_t , \grad_v \left(S_{ \PPP^*} (0) \left[ \left(S_ \PPP(\tau) h_t\right) \omega^2)\right] \right) \right\ra _{L^2(\OO)}\d\tau \\
&=& -\int_0^\infty \alpha  \EE_{g_t} \left\la \grad_v h_t  , \grad_v \left(S_{ \PPP^*} (0) \left[ \left( S_ \PPP(\tau) h_t\right) \omega^2)\right]  \right) \right\ra _{L^2(\OO)}\d\tau \\
&\leq & \alpha C_\omega^1 \int_0^\infty \lvv g_t \rvv_{L^2_\omega(\OO)}  \lvv \grad_v h_t  \rvv_{L^2_\omega(\OO)} \lvv \grad_v (S_{ \PPP^*}(0) \left[ \left(S_ \PPP(\tau) h_t\right) \omega^2\right] ) \rvv _{L^2_m(\OO)}\d\tau , 
\eean
where we have successively used integration by parts and the Cauchy-Schwartz inequality. Arguing exactly as for the term $P_2$ during the Step 2 we deduce that 
$$
P_1 \leq \alpha C_\omega^1 \lvv h_t \rvv_{L^2_\omega(\OO)} \lvv g_t \rvv_{L^2_\omega(\OO)}  \lvv \grad_v h_t\rvv_{L^2_\omega(\OO)} C_P  \left( \frac {C_P^*}{2\lambda} + \frac { C_P^* + 1 }\lambda \left\lvv \HHHH\right\rvv_{L^2_\omega(\OO)} \left\lvv  m \right\rvv _{L^2(\OO)}  \right).
$$
Using then the Young inequality we further have that
\bean
P_1 &\leq & \frac \alpha 2 C_\omega^1  \lvv g_t \rvv_{L^2_\omega(\OO)}  \lvv \grad_v h_t\rvv_{L^2_\omega(\OO)}^2 C_P  \left( \frac {C_P^*}{2\lambda} + \frac { C_P^* + 1 }\lambda \left\lvv \HHHH\right\rvv_{L^2_\omega(\OO)} \left\lvv  m \right\rvv _{L^2(\OO)}  \right)\\
&&+ \frac \alpha 2 C_\omega^1 \lvv h_t \rvv_{L^2_\omega(\OO)}^2 \lvv g_t \rvv_{L^2_\omega(\OO)} C_P  \left( \frac {C_P^*}{2\lambda} + \frac { C_P^* + 1 }\lambda \left\lvv \HHHH\right\rvv_{L^2_\omega(\OO)} \left\lvv  m \right\rvv _{L^2(\OO)}  \right).
\eean

\medskip\noindent
\emph{Step 4. (Choice of parameters and conclusion)} Putting together the computations performed during the Steps 1, 2 and 3, and using \eqref{eq:PropsSS} and \eqref{eq:BdHHHH}, we have that 
\bean
\frac d{dt} \lvvv h_t \rvvv^2 
&\leq & -\beta c_A^{-1} \frac{\tau_0}2 \lvv \grad_v h_t\rvv_{L^2_{ \omega}(\OO)}^2  + \beta c_A \left( \kappa + C^0_\omega \frac{\tau_0}2 \right) \lvv h_t\rvv_{L^2_{ \omega}(\OO)}^2 -  \lvv  h_t \rvv_{L^2_\omega(\OO)}^2 \\
&& + \alpha C_\omega^1 \lvv h_t \rvv_{L^2_\omega(\OO)}^2 C_{\FFFF^\alpha}  C_P \left( \frac {C_P^*}{2\lambda} + \frac { C_P^* + 1 }\lambda C_\HHHH \left\lvv  m \right\rvv _{L^2(\OO)}  \right) \\
&&+   C_\omega^1  \lvv g_t \rvv_{L^2_\omega(\OO)}  \lvv \grad_v h_t\rvv_{L^2_\omega(\OO)}^2 C_P  \left( \frac {C_P^*}{2\lambda} + \frac { C_P^* + 1 }\lambda C_\HHHH  \left\lvv  m \right\rvv _{L^2(\OO)}  \right)\\
&&+  C_\omega^1 \lvv h_t \rvv_{L^2_\omega(\OO)}^2 \lvv g_t \rvv_{L^2_\omega(\OO)}  C_P \left( \frac {C_P^*}{2\lambda} + \frac { C_P^* + 1 }\lambda C_\HHHH \left\lvv  m \right\rvv _{L^2(\OO)}  \right),
\eean
where the constants $C_{\FFFF^\alpha}>0$ and $C_\HHHH>0$ are given by \eqref{eq:PropsSS} and \eqref{eq:BdHHHH} respectively. 
We then choose 
\bean
\alpha &= &\frac 14 \left( C_\omega^1 C_{\FFFF^\alpha} C_P \left( \frac {C_P^*}{2\lambda} + \frac { C_P^* + 1 }\lambda C_\HHHH \left\lvv  m \right\rvv _{L^2(\OO)}  \right)\right)^{-1}, \\
\beta &=& \frac 14 c_A^{-1} \left( \kappa + C^0_\omega \frac {\tau_0} 2 \right),\\
\eps_3^2 &=& \min\left(  \frac 14, \, {\beta c_A^{-1}\over 4}  \right) \left( C_\omega^1C_P \left( \frac {C_P^*}{2\lambda} + \frac { C_P^* + 1 }\lambda C_\HHHH \left\lvv  m \right\rvv _{L^2(\OO)}  \right)\right)^{-1},
\eean
and we set $\eps_3 = \min (\eps_3^1, \eps_3^2)$. This selection of parameters implies that 
$$
\frac d{dt} \lvvv h_t \rvvv^2 \leq - \frac 14 \lvv  h_t \rvv_{L^2_\omega(\OO)}^2, 
$$
and we conclude the proof by using \eqref{eq:HypoNormEquiv} and the Grönwall lemma. 
\end{proof}

\section{Proof of Theorem \ref{theo:GlobalSolution}}\label{sec:ProoTheo3}

We consider the non-linear equation
\be\label{eq:KFPstabilityNL}
\left\{\begin{array}{rcll} 
 \partial_t h &=& \QQQ_h h & \text{ in } \UU  \\
\displaystyle \gamma_- h &=& \RRR \gamma_+ h &\text{ on } \Gamma_{-} \\
\displaystyle  h_{t=0} &=& h_0 &\text{ in } \OO,
 \end{array}\right.
 \ee 
where we recall that $\QQQ_h$ is defined in \eqref{eq:DefQQQg},
and we will dedicate this section to prove an equivalent version of Theorem \ref{theo:GlobalSolution} in terms of the solutions of Equation \eqref{eq:KFPstabilityNL}. 
 
 \smallskip\noindent
 Indeed, we recall from Remark \ref{rem:NonLPerturbed} that if there is a fixed point for the map that associates $g$ to a solution $h$ of Equation \eqref{eq:KFPstabilityLin}, then $f = \FFFF^\alpha + h$ will be a solution of Equation \eqref{eq:NonLKFP}-\eqref{eq:BoundaryConditions}-\eqref{eq:InitialDatum}. Moreover the decay estimate \eqref{eq:DecaySmallness} will be a direct consequence of Proposition \ref{prop:Hypo}.

\begin{prop}\label{prop:FixedPointPerturbed}
Let $\omega$ be an admissible weight function. There is a constant $\delta>0$ such that, for any $\alpha \in (0, \alpha^{\star\star})$, where $\alpha^{\star\star}>0$ is given by Proposition \ref{prop:Hypo},
and any initial datum $h_0\in L^2_\omega(\OO)$ such that
$$
\lvv h_0\rvv_{L^2_\omega(\OO)} \leq \delta,
$$
there is $h \in L^2_\omega(\UU)$ unique global weak solution of Equation \eqref{eq:KFPstabilityNL} in the sense of Proposition \ref{prop:KFPL2Perturbed}. Furthermore, the decay estimate \eqref{eq:Hypo} holds for the solutions of Equation \eqref{eq:KFPstabilityNL}.
\end{prop}

\begin{rem}
The proof follows the main ideas from the proof of \cite[Theorem 1.1]{CM_Landau_domain}.
\end{rem}

\begin{proof}
We define the ball
$$
\ZZ := \{ g \in L^2_\omega(\UU), \, \| g_t \|_{L^2_\omega(\OO)} \le \eps_3 \quad \forall t\geq 0 \},  
$$
where $\eps_3>0$ is given by Proposition \ref{prop:Hypo}, and it is worth remarking that $\eps_3 \leq \min(\eps_1, \eps_2)$, where $\eps_1, \eps_2>0$ are given by Propositions \ref{prop:GrowthLpPrimal+NNNN} and \ref{prop:KFPL2Perturbed} respectively, thus making them valid. 
We set $\delta =\eps_3 /C_0$ where $C_0\geq 1$ is given by Proposition \ref{prop:Hypo}, and we define the map
$$
\Phi : \ZZ \to \ZZ, \quad g \mapsto \Phi(g) = G  := S_{\QQQ_{g}} h_0. 
$$
It is worth emphasizing that, due to our choice of $\delta$, Proposition \ref{prop:Hypo} will ensure that $S_{\QQQ_{g}} h_0 \in \ZZ$ for every $g\in \ZZ$.
Finally, we endow $\ZZ$ with the weak  topology induced by $L^2_\omega(\OO)$, which makes $\ZZ$ a convex and compact set.

\medskip\noindent
\emph{Step 1. (Continuity of the map $\Phi$)} 
We consider a sequence $(g_n)$ in $\ZZ$ such that  $g_n \wto g$ weakly in $\ZZ$ as $n \to \infty$ and we define $G_n := S_{\QQQ_{g_n}} h_0$. 
Using Proposition \ref{prop:Hypo} together with \eqref{eq:bddGradPerturbed}, we have that
$$
\| G_n (t,\cdot) \|_{L^2_{\omega}(\OO)} +  \lvv \grad_v G_n  \rvv_{L^2_\omega(\UU)} 
\le C_0  \| h_0 \|_{L^2_\omega}  \leq \eps_3 \quad \forall \, t \ge 0, 
$$
so that $G_n \in \ZZ \cap L^2(\R_+ \times \Omega; H^1_\omega(\R^d))$. Thus there exist a subsequence $(G_{n'})$ and a function $G \in \ZZ \cap L^2(\R_+ \times \Omega; H^1_\omega(\R^d))$ such that 
\be\label{eq:Weak_Conv_Gn}
G_{n'} \wto G \quad \text{ weakly in } \ZZ \cap L^2(\R_+ \times \Omega; H^1_\omega(\R^d)) \text{ as } n' \to \infty.
\ee

\smallskip\noindent
We emphasize that $G_{n'}$ solves 
\be\label{eq:Perturbed_Equation_Schauder_nprime}
\partial_t G_{n'} = \PPP G_{n'} + \NNNN_{g_{n'}} G_{n'}, \quad \gamma_- G_{n'} = \RRR \gamma_+ G_{n'}  , \quad
(G_{n'})_{| t=0} = h_0, 
\ee
in the sense provided by Proposition~\ref{prop:KFPL2Perturbed}. Therefore, by arguing as during the Step 3.1 of the proof of Proposition~\ref{prop:KFPL2Perturbed}, we have that for any truncated (in $t$ and $x$) version $(\bar G_{n'})$ of $(G_{n'})$, we may apply \cite[Theorem~1.3]{MR1949176}, which gives that $(\bar G_{n'})$ is bounded in $H^{1/4} ( \R_t \times \R^d_x \times \R^d_v)$. Using then the version of the Rellich-Kondrachov theorem for fractional Sobolev spaces (see for instance \cite[Corollary 7.2]{MR2944369} or \cite[Lemma 6.11]{MR3081641}) we deduce that 
\be\label{eq:Relative_Compact_L2}
(G_{n'} ) \hbox{ is relatively compact in } L^2((0,T) \times \OO_R). 
\ee
for any $T,R > 0$, and we recall that $\OO_R := \{ (x,v) \in \OO ; \, d(x,\Omega^c) > 1/R, \, |v| < R \}$. 

\smallskip\noindent
Moreover, setting $\varsigma = \omega \la v\ra^{-1/2}$ we observe that the sequence $\left( G_{n'}\right)$ is tight in $L^2_\varsigma(\UU)$. Indeed we have that
\be\label{eq:Tightness_Gn_L2_1}
\lvv G_{n't} \rvv_{L^2_{\varsigma}( \OO_R^c)} \leq \frac 1{\la R\ra^{1/2}} \lvv G_{n't} \rvv_{L^2_{\omega}(\OO_R^c)} \leq \frac 1{\la R\ra^{1/2}} \lvv G_{n't} \rvv_{L^2_\omega(\OO)} \lesssim \frac 1{\la R\ra^{1/2}} e^{-\eta t} \lvv h_0\rvv_{L^2_\omega(\OO)},
\ee
for every $t \geq 0$, and where we remark that we have used \eqref{eq:Hypo} to obtain the last inequality and $\eta>0$ is given by Proposition \eqref{prop:Hypo}. Integrating \eqref{eq:Tightness_Gn_L2_1} in the time interval $(T, \infty)$ we further deduce that
\be\label{eq:Tightness_Gn_L2}
\lvv G_{n't} \rvv_{L^2_{\varsigma}((T, \infty)\times \OO_R^c)}  \lesssim \frac 1{\la R\ra^{1/2}} e^{-2\eta T} \lvv h_0\rvv_{L^2_\omega(\OO)},
\ee
which gives the tightness of the sequence $G_{n'}$ in $L^2_\varsigma(\UU)$. We then classically deduce from \eqref{eq:Relative_Compact_L2} and \eqref{eq:Tightness_Gn_L2} that
$$
G_{n'}  \to G \  \hbox{ strongly in } \ L^2_\varsigma((0,\infty) \times \OO) \text{ as } n'\to \infty.
$$
Using this convergence, the Cauchy-Schwartz inequality, and the fact that $\langle v \rangle^{5/2} \omega^{-1} \in L^2(\R^d)$, we deduce then that
\be\label{eq:Convergence_Energy_L2_Gn}
\EE_{G_{n'}} \to \EE_G \qquad \text { as } n'\to \infty.
\ee
Using \eqref{eq:Weak_Conv_Gn} and \eqref{eq:Convergence_Energy_L2_Gn} we may then pass to the limit in the weak formulation associated to Equation \eqref{eq:Perturbed_Equation_Schauder_nprime} and we obtain that $G$ solves
\beqn\label{eq:Perturbed_Equation_Schauder_limit}
\partial_t G = \PPP G+ \NNNN_g G, \quad \gamma_- G  = \RRR \gamma_+ G, \quad  G_{| t=0} = h_0, 
\eeqn
in the sense of Proposition~\ref{prop:KFPL2Perturbed}. Moreover, from Proposition \ref{prop:Hypo} and Proposition \ref{prop:GrowthLpPrimal+NNNN} we also have that 
$$
\| G \|_{L^2_\omega(\UU)} \le \eps_3, \quad \text{ and } \quad \int_0^\infty \!\! \int_\OO |\nabla_v  G|^2 \omega^2 \,  \d v \, \d x \, \d t \lesssim  \| h_0 \|^2_{L^2_\omega(\OO)}.
$$
Finally, due to the uniqueness given by Proposition~\ref{prop:KFPL2Perturbed} we get that $G$ is the only possible function in $L^2_\omega(\UU)$ such that $G = S_{\QQQ_{g}} h_0$. By the uniqueness of the possible limit we deduce that the map $\Phi$ is continuous.

\medskip\noindent
\emph{Step 2. (Conclusion)} We may apply now Schauder's fixed point theorem in the space $\ZZ$ with the map $\Phi$, and we obtain that there exists $h \in \ZZ$ such that $h = \Phi (h)$. Moreover, due to the very definition of $\Phi$, we have that $h$ is a global weak solution of Equation \ref{eq:KFPstabilityNL} in the sense of Theorem~\ref{prop:KFPL2Perturbed}, and
it satisfies the decay estimate \eqref{eq:Hypo}. This concludes the proof.
\end{proof}

\bigskip\noindent
\textbf{Acknowledgements.} The authors warmly thank Kleber Carrapatoso and Stéphane Mischler for the discussions during the process of the article and for pointing out important references. 
J. Evans is supported by a Royal Society University Research Fellowship R1\_251808 (since October 2025). 
R. Medina acknowledges the funding received  from the European Union’s Horizon 2020 research and innovation programme under the Marie Skłodowska-Curie grant agreement No 945332 \includegraphics[width = .5cm]{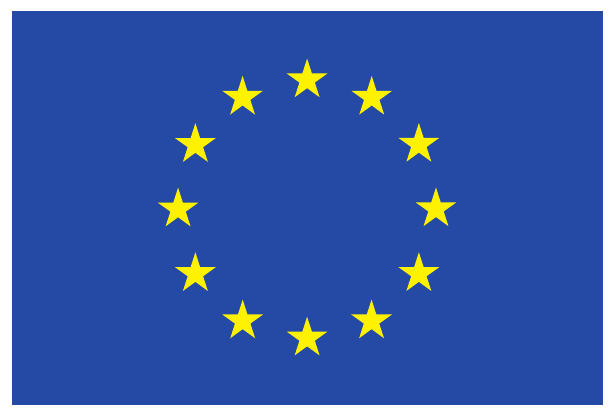}.
For the purpose of open access, the authors have applied a Creative Commons Attribution (CC-BY) license to any Author Accepted Manuscript version arising from this submission.

\bigskip
\bigskip
\bibliographystyle{plain}
\bibliography{1.KFPnonlinear.bib}

\end{document}